\documentclass[11pt, reqno, fullpage]{amsart}
\usepackage{mathrsfs}
\usepackage{amsfonts}
\usepackage{amsmath}
\usepackage{amsthm}
\usepackage{amssymb, stackrel}
\usepackage{latexsym}
\usepackage[all]{xy}
\usepackage[mathscr]{eucal}
\usepackage{array}
\usepackage{palatino}
\usepackage[left=2.5cm, right=2.5cm]{geometry}

\usepackage[dvipsnames]{xcolor} \usepackage{graphicx}{

\usepackage{mathtools}
\usepackage{stmaryrd}
\hyphenation{speci-fically}
\usepackage{graphicx}
\usepackage[bookmarks=false]{hyperref}
\setcounter{tocdepth}{1}

\usepackage{hyperref}
\pagestyle{plain}

\theoremstyle{plain}
\newtheorem{thm}[equation]{Theorem}
\newtheorem{cor}[equation]{Corollary}
\newtheorem{prop}[equation]{Proposition}
\newtheorem{lem}[equation]{Lemma}

\newtheorem{quest}[equation]{Question}

\theoremstyle{definition}

\theoremstyle{remark}

\newtheorem{rem}[equation]{Remark}
\newtheorem{rems}[equation]{Remarks}
\newtheorem{claim}[equation]{Claim}

\newtheorem{examp}[equation]{Example}
 
\makeatletter
\renewcommand{\subsection}{\@startsection{subsection}{2}{0pt}{-3ex
plus -1ex minus -0.2ex}{-2mm plus -0pt minus
-2pt}{\normalfont\bfseries}} \makeatother

\numberwithin{equation}{subsection}
\allowdisplaybreaks

\newlength{\dhatheight}

\newcommand{\hdot}{{\:\raisebox{2pt}{\text{\circle*{1.5}}}}}
\newcommand{\idot}{{\:\raisebox{2pt}{\text{\circle*{1.5}}}}}

\newcommand{\dc}[1]{{#1}^\blacktriangledown_\rho}
\newcommand{\ogr}[1]{{\overline{\Gr_{{#1}}}}}

\DeclareMathOperator{\mmod}{\!\text{-}\mathrm{mod}}
\DeclareMathOperator{\grfmod}{\!\text{-}\mathrm{fmod}}

\DeclareMathOperator{\Map}{{\mathrm{Map}}}

\DeclareMathOperator{\Ext}{\mathrm{Ext}}
\DeclareMathOperator{\RHom}{{\mathrm{RHom}}}
\DeclareMathOperator{\sym}{\mathrm{Sym}}

\DeclareMathOperator{\im}{\mathrm{Im}}
\DeclareMathOperator{\supp}{\mathrm{Supp}}
\DeclareMathOperator{\coh}{\mathrm{Coh}}
\DeclareMathOperator{\qcoh}{\mathrm{QCoh}}
\DeclareMathOperator{\Ker}{\mathrm{Ker}}

\DeclareMathOperator{\Hom}{\mathrm{Hom}}

\DeclareMathOperator{\rk}{{\mathrm{rk}}}

\DeclareMathOperator{\Rep}{\mathrm{Rep}}

\DeclareMathOperator{\Lie}{\mathrm{Lie}}
\DeclareMathOperator{\Tr}{\mathrm{tr}}

\DeclareMathOperator{\Ad}{\mathrm{Ad}}

\DeclareMathOperator{\Ind}{{Ind}\!}

\DeclareMathOperator{\proj}{{\text{-}\mathrm{proj}}}

\DeclareMathOperator{\Spec}{\mathrm{Spec}}

\newcommand{\BM}{{\textrm{\tiny {BM}}}}
\newcommand{\erems}{\hfill$\lozenge$\end{rems}}
\newcommand{\tsl}{/\!/\!/}

\newcommand{\ssh}{{\mathrm{Shv}}}
\newcommand{\gln}{{\mathfrak{gl}_n}}

\newcommand{\chh}{{\mathcal H}}

\newcommand{\iso}{{\;\stackrel{_\sim}{\to}\;}}

\newcommand{\vect}{{\mathbb{V}_{\!}\mathrm{ect}}}
\newcommand{\Gr}{{\mathrm{Gr}}}
\newcommand{\ccong}{\,\cong \,  }
\newcommand{\erem}{\hfill$\lozenge$\end{rem}}
\newcommand{\lo}{{\stackrel{{}_L}\o}}

\newcommand{\ce}{{\mathcal E}}
\newcommand{\sstt}{{\mathsf{st}}}
\newcommand{\Id}{{\operatorname{Id}}}

\newcommand{\rgo}{{\operatorname{R}\!\Gamma_{G(O)}}}

\newcommand{\ccc}{{\mathcal C}}
\newcommand{\hh}{{\mathbb{H}}}

\newcommand{\ccr}{{\mathcal C}_{reg}}

  \newcommand{\chom}{{\mathscr{H}\textit{om}}}

\newcommand{\cz}{{\mathcal Z}}
\newcommand{\sfr}{{\mathsf{R}}}

\newcommand{\fcr}{{{\mathfrak C}_r}}

\newcommand{\BF}{{\mathbb F}}
\newcommand{\HF}{{\mathrm{H}\Phi}}
\newcommand{\hrm}{{\mathrm{H}}}

\renewcommand{\th}{\theta }

\newcommand{\shh}{{H}}

\newcommand{\WH}{{\mathsf{Wh}_\dg}}
\newcommand{\WHG}{{\mathsf{Wh}_{H\to G}}}

\newcommand{\WW}{{\mathsf{Wh}}}

\newcommand{\perv}{{{\mathcal P}erv}}
\newcommand{\per}{{{\mathcal P}erv}_{G(O)}(\Gr)}

\newcommand{\sfc}{{\mathsf C}}

\newcommand{\indd}{\text{Ind}D_{\gc(O)}(\Gr)}
\newcommand{\ct}{{\mathscr T}}

\renewcommand{\k}{{\Bbbk}}
\newcommand{\sfh}{{\mathsf H}}
\newcommand{\CM}{{\mathcal{CM}}}

\newcommand{\mix}{ D_T^{\text{mix}}}

\newcommand{\da}{{\dagger}}

\newcommand{\la}{\lambda}
\newcommand{\fii}{{\mathfrak i}}

\newcommand{\sg}{{\sym(\dgl[-2])}}
\newcommand{\wh}{{{{\mathcal W}h}_\dg}}

\newcommand{\dg}{{\check G}}
\newcommand{\dgl}{{\check\g}}
  \newcommand{\dt}{{\check T}}
\newcommand{\dtl}{{\check\t}}
  \renewcommand{\dh}{{\check H}}
  \newcommand{\dhl}{{\check{\mathfrak h}}}
  \newcommand{\du}{{\check U}}
  \newcommand{\chev}{{\mathscr C}\textit{h}}
    \newcommand{\dul}{{\check{\mathfrak u}}}

\newcommand{\pre}{\unlhd }
\newcommand{\pree}{\lhd }
\newcommand{\fcc}{{\mathfrak c}}

\newcommand{\greg}{{{\check\g}^*_r}}

\newcommand{\dom}{\BX_{+}}
\newcommand{\pos}{^{\text{pos}}}
\newcommand{\vrho}{\varrho }
\newcommand{\jm}{\jmath }

\newcommand{\sat}{{D_{\gc(O)}(\Gr)}}

\newcommand{\be}{\beta }
\newcommand{\si}{{\sigma}}

\newcommand{\ol}{\overline}

\newcommand{\cy}{{\mathcal Y}}

\newcommand{\fj}{{\mathfrak j}}
\newcommand{\cll}{{\mathcal L}}

\newcommand{\nana}{^\flat }

\newcommand{\hv}{{H}}
\newcommand{\cx}{{\mathcal X}}

\newcommand{\bplus}{\mbox{$\bigoplus$}}

\newcommand{\bone}{{\boldsymbol{1}}}

\newcommand{\BD}{{\mathbb D}}
\newcommand{\dd}{{\mathscr{D}}}
\newcommand{\oo}{{\mathcal{O}}}
\renewcommand{\aa}{{\mathcal{A}}}

\newcommand{\bn}{{\mathbf{N}}}
\newcommand{\pt}{{\operatorname{pt}}}

\newcommand{\fu}{{\mathfrak u}}
\newcommand{\chb}{{\C[\hb]}}

\newcommand{\cg}{\mathcal{G} }

\newcommand{\back}{\backslash }

\newcommand{\veps}{\varepsilon }

\renewcommand{\o}{\otimes }
\newcommand{\so}{{\stackrel{{}_!}{\otimes} }}

\newcommand{\Si}{\Sigma }

\renewcommand{\t}{{\mathfrak t}}

\newcommand{\vpi}{\varpi}

\newcommand{\IC}{{\operatorname{IC}}}

\newcommand{\gc}{{G}}
\newcommand{\De}{\Delta }
\newcommand{\BX}{{\mathbb X}}

\newcommand{\tv}{{T}}

\newcommand{\hb}{\hbar }

\newcommand{\hdt}{{H_T^\hdot}}

\newcommand{\ga}{\gamma }

\newcommand{\beq}{\begin{equation}\label}
\newcommand{\eeq}{\end{equation}}
\def\ccirc{{{}_{\,{}^{^\circ}}}}

\newcommand{\al}{\alpha }

\newcommand{\bx}{{\Phi}}
\newcommand{\wt}{\widetilde }

\newcommand{\fc}{{\mathrm{S}}}

\newcommand{\td}{\text{-}}
\newcommand{\BA}{{\mathbb{A}}}

\newcommand{\eps}{\epsilon }

\newcommand{\C}{\Bbbk}
\newcommand{\g}{\mathfrak{g}}

\newcommand{\h}{\mathfrak{h}}
\newcommand{\La}{\Lambda }

\newcommand{\rr}{{\mathcal R}}

\newcommand{\Gm}{{{\mathbb G}_m}}

\newcommand{\inv}{^{-1}}

\newcommand{\cf}{{\mathcal F}}
\newcommand{\Z}{{\mathbb Z}}
\setlength{\unitlength}{1pt}

\newcommand{\en}{{\enspace}}

\newcommand{\vi}{${\en\sf {(i)}}\;$}
\newcommand{\vii}{${\;\sf {(ii)}}\;$}
\newcommand{\viii}{${\sf {(iii)}}\;$}
\newcommand{\iv}{${\sf {(iv)}}\;$}

\newcommand{\sset}{\subseteq}
\newcommand{\sminus}{\smallsetminus}
\newcommand{\intoo}{\,\xymatrix{\ar@{^{(}->}[r]&}\,}
\newcommand{\ontoo}{\,\xymatrix{\ar@{->>}[r]&}\,}
\newcommand{\into}{\,\hookrightarrow\,}
\newcommand{\too}{\,\longrightarrow\,}
\newcommand{\mto}{\mapsto}
\newcommand{\onto}{\,\twoheadrightarrow\,}

\newcommand{\bm}{\mathbf{M}}

\newcommand{\Ga}{\Gamma }
\newcommand{\gm}{\Gm }

\newcommand{\om}{\omega }

\setlength{\oddsidemargin}{0.2in}
\setlength{\evensidemargin}{0.2in}

\author{Victor Ginzburg}
\email{vityaginzburg@gmail.com}

\title{\small{Pointwise purity, derived Satake, and Symplectic duality}}
  
\begin{document}

  \begin{abstract}
    It has been known for a long time  that $\Ext$'s between
$\IC$-sheaves may often be expressed
in terms of $\Hom$'s between 
cohomology groups.
We prove a  more general result under weaker assumptions.
The result is used to describe the action of the derived Satake
equivalence  on $!$-pure objects and show that the
equivalence enjoys a new kind of functoriality 
with respect to morphisms
of reductive groups.
We find and prove normality of  the symplectic dual $X^!$
  for many  smooth affine Hamiltonian $G$-varieties $X$, including
    $X=T^*(G/H)$ for all connected reductive subgroups $H\sset  G$.
    We also describe the symplectic duals  ${\mathbf M}^!$ in the case of
  Coulomb branches   and prove that ${\mathbf M}^!$ has symplectic singularities.
\end{abstract}

\maketitle
{\small
\tableofcontents
}
\section{Main results}

   \subsection{Pointwise purity}\label{form sec}
   In this paper, we  use the yoga of `weights' either in the \'etale setting of
   \cite{BBD} or in the sense of mixed Hodge modules of Saito \cite{Sa}.
    Given a connected linear algebraic group  $G$ and a $G$-variety $X$
   in either of these settings,
   there is a mixed equivariant derived category
   $D_G^{\text{mix}}(X)$, often viewed
   as the mixed derived category of the  stacky quotient $X/G$,  cf. \cite[4.2]{HL}
   and \cite{Ach}.
      Accordingly, we let
   $\k$, the coefficient field, be
   either $\bar{\mathbb Q}_\ell$ or $\k=\mathbb{C}$, respectively.

  Fix a stratification
  $X=\sqcup_{\la\in\La}\, X_\la$   by locally closed smooth
   $G$-stable  connected  subvarieties where $\La$ is a finite labeling set.
    We say that $\cf\in D_G^{\text{mix}}(X)$
  is  $*$-{\em pure}, resp. $!$-{\em pure},
  if for all $\la\in\La$ and $m\in\Z$, the $m$-th cohomology sheaf
  ${\mathscr H}^m\jm^*_\la\cf$,  resp. ${\mathscr H}^m\jm^!_\la\cf$,
  where  $\jm_\la: X_\la \into X$ is the imbedding
  is geometrically constant and
  pointwise pure of weight $m$. We  say that $\cf$ is pure if it is both $*$-pure  and $!$-pure.
%
  %
  We write    $X^G\sset X$ for the fixed point set.
   
    Below, we consider the case where the group  is a split torus $T$
    and write $\hh(\td)=\hdt(\td)$ for $T$-equivariant cohomology.
    Let   $X$  be a not necessarily connected  projective $T$-variety
    with irreducible connected components
    and $X=\sqcup_{\la\in \La}\, X_\la$  a stratification as above by  $T$-stable   subvarieties. 
    Let $\fii: Y\into X$ be the closed imbedding of a $T$-stable  subvariety
    and
$Y_\la:=Y\cap X_\la$. Thus,
we have $Y=\sqcup_{\mu\in\Si}\, Y_\mu$, where
$\Si=\{\mu\in\La\mid Y_\la\neq\emptyset\}$.
We consider  $*$-pure, resp. $!$-pure and  pure,
objects of $D_T^{\text{mix}}(X)$, resp. $D_T^{\text{mix}}(Y)$,
with respect to the stratification $X=\sqcup_\la X_\la$,
resp. $Y=\sqcup_\mu Y_\mu$.

We make the following
assumptions:

\begin{description}
\item[S1] For every $\la\in\La$ there is a  vector space $V_\la$ with a linear $T$-action
  and a $T$-equivariant isomorphism $f_\la: \en X_\la\iso V_\la$ such that
    $f_\la(Y_\la)$ is a vector subspace of
  $V_\la$ if $\la\in\Si$.
  
\item[S2]     For all $\la\in\Si$  one has $(X_\la)^T=\{x_\la\}$
  where  $x_\la:=f_\la\inv(0)$.

\item[S3] For every $\la\in\Si$ there  is a Zariski open  $T$-stable  subset $U_\la\sset X$ and a
  $\gm$-action on $U_\la$ that commutes with the $T$-action and
   contracts
   $U_\la$ to the point $x_\la$.
   \end{description}

Given a $\Z$-graded algebra $A$ and a pair 
$M=\oplus_{k\in\Z}\, M_k,\
  N=\oplus_{k\in\Z}\, N_k$, of graded $A$-modules,
 we let $\Hom^n_A(M,N)$   denote the group
  of $A$-linear maps $f: M\to N$ such that $f(M_k)\sset N_{k+n}$ for all $k$.

  Our  main general result on pointwise pure sheaves, to be proved in Sect. \ref{pf vthm}, reads

    \begin{thm}\label{Vthm} If conditions S1-S3 hold then
  for any 
     pure $\ce_X\in\mix(X)$ and  $!$-pure $\cf_Y\in\mix(Y)$
        the functor $\hh(\td)$ induces  an isomorphism of
    graded $\hh(X)$-modules
     \[
   \Ext^\hdot_{D_T(X)}(\ce_X,\fii_*\cf_Y)\ \iso\ \Hom^\hdot_{\hh(X)}(\hh(\ce_X),\, \hh(\fii_*\cf_Y)).
   \]
  \end{thm}

  By duality, from the theorem one obtains, cf.  \eqref{fiii i}:
  \begin{cor}\label{j cor} Let  $\ce\in \mix(X)$  be pure, $\la\in\Si$ and $\iota_\la: \{x_\la\}\into X$
    the imbedding. Then there is a natural  isomorphism
\[    \hh(\iota_\la^!\ce) \iso \Hom^\hdot_{\hh(X)}(\hh(\{x_\la\}),\,
  \hh(\ce)).
\]
\end{cor}
 
    In the special case where $Y=X$ and  $\cf$ is both $!$-pure  {\em and} $*$-pure 
    the isomomorphism of the theorem  has been  proven in \cite{BY}
 following
 the strategy used in \cite{Gi} in the nonequivariant setting, cf. also \cite{CMNO}
    for a related result.
 Some of the arguments used in {\em op cit}
 do not work in our more general setting and a nonequivariant analogue of
 Theorem \ref{Vthm}
 is false, in general. Indeed, let $X$ be a projective $T$-variety with
 isolated fixed points which  is not rationally  smooth and let $\ce:=\IC(X)$.
Then, the  dimensions of the groups $\hh(\iota_\la^!\IC(X))$
  are not
 the same for all $\la\in\La$.
 On the other hand, for any $\la$, the non-equivariant cohomology group
 $H^\hdot(\{x_\la\})$
 is a one dimensional vector space  isomorphic to $H^\hdot(X)/H^{>0}(X)$
 as an $H^\hdot(X)$-module.
 Therefore, the  space  $\Hom^\hdot_{H^\hdot(X)}\big(H^\hdot(\{x_\la\}),\,
  H^\hdot(\IC(X))\big)$ is independent of $\la$. Thus,
  a non-equivariant analogue of the isomorphism of Corollary \ref{j cor},
  hence of Theorem \ref{Vthm},
 does not hold.

 \begin{rems} 
   \vi It follows from conditions S1-S2  that one has $\{x_\la,\,\la\in\La\}\sset X^T$
      and $\{x_\la,\, \la\in \Si\}=Y^T$.
      The inclusion $\{x_\la\}\sset (X_\la)^T$ 
   is not required to be an equality for $\la\in\La\sminus\Si$.
        \vskip 2pt

        \vii     Condition S2  is equivalent to the following condition:       
 \begin{description}
 \item[S2']  Each connected component of the fixed point variety $X^T$ either
   does not meet $Y$ or 
   meets $Y$ at a single point $x_\la$ and in this case  one must have
   $(X_\la)^T=\{x_\la\}$.   
 \end{description}
 This condition holds trivially if $X^T= \{x_\la,\, \la\in \La\}$, so condition S2 
 hasn't appeared in earlier works on the subject where $T$-fixed points in
 $X$ were assumed to be isolated.
 \vskip 2pt

 \viii  The set $U_\la$ that appears in condition S3 is not required to contain the stratum $X_\la$.
\end{rems}

Our  results on derived Satake rely on  Theorem
\ref{Vthm} in the case where
$X$  is (a sufficiently large projective subvariety of)  the affine Grassmannian
$\Gr_G$ of a reductive group $G$ and  $Y=\Gr_H\into X=\Gr_G$
is the imbedding of the affine Grassmannian of  a
reductive subgroup of $G$.
 Conditions S1-S3 hold for such an imbedding, 
 Sect. \ref{iwahori sec}.
The pure sheaf $\ce_X$  is going to be the regular  perverse sheaf on $\Gr_G$.
Most of the  $!$-pure sheaves $\cf_Y$ that we will consider are either not
$*$-pure or are objects of an  ind-completion, in which case  the notion of $*$-purity is
not  well defined.

\subsection{Derived Satake functor on $!$-pure objects}
\label{sat int}
Below, we will frequently encounter actions of group schemes. 
Let $\mathcal J$ be a group scheme over  a scheme $S$
and  $X, Y$ a pair  of schemes over $S$
equipped with $\mathcal J$-actions $\mathcal J\times_S X\to X,\,\mathcal J\times_S Y\to Y$.
If  $X$  is a $\mathcal J$-torsor we will use the notation  $X\times_S ^{\mathcal J} Y$
for the quotient of
$X\times_S Y$ by the diagonal $\mathcal J$-action, that is, a twist of $Y$ by the
torsor $X$.

Let $G$ be a split connected  reductive group  either over $\mathbb C$ or
$\BF_q$. Let $\dg$ be the Langlands dual group over $\k$ with Lie algebra
$\dgl$. Let  $\greg\sset \dgl^*$ be the set of regular elements,
$\vpi: \greg\to\fcc:= \dgl^*/\!/\dg$  the coadjoint quotient, and
$J=J_\dg\to \fcc$   the group scheme of
universal centralizers.

Fix a maximal unipotent subgroup $\du\subset \dg$ and
a nondegenerate character $\psi:\dul=\Lie \du\to\C$. Let
$\WH:=T^*_\psi(\dg/\du)=\dg\times_\du(\psi+\dul^\perp)$ be
the Whittaker cotangent bundle on $\dg/\du$.
This is a smooth affine symplectic variety equipped with commuting Hamiltonian
actions of $\dg$ and $J$ with moment maps
 $\pi: \WH\to\greg,\,(g,\xi)\mto \Ad g(\xi)$ and  $p=\vpi\ccirc\pi: \WH
 \xrightarrow{\pi}
 \greg\xrightarrow{\vpi} \fcc$, respectively.
 The map $\pi$, resp. $p$, makes $\WH$ a $\dg$-equivariant $\vpi^*J$-torsor,
 resp. $\dg$-torsor.


 It will be convenient to view $\fcc$ as a variety equipped with the trivial $J$-action.
There is a canonical  isomorphism $\fcc/J\cong  \greg/\dg$ of quotient stacks,
hence an equivalence $\qcoh(\fcc/J)=\qcoh^{J}(\fcc)\cong \qcoh(\greg/\dg)=\qcoh^{\dg}(\greg)$,
between  the  respective equivariant categories of
 quasi-coherent sheav-es.
This equivalence can be described more explicitly as follows. 
 The pull-back $p^*M$ of $M\in \qcoh^{J}(\fcc)$ comes equipped with
a $\dg\times J$-equivariant structure  induced
by the $\dg\times J$-action  on $\WH$. There is also an 
action $\gamma:  J\times_\fcc\, p^*M\to p^*M$ induced by
the $J$-action on $M$.
We define a new $J$-equivariant structure on $p^*M$ by twisting the original
 $J$-equivariant structure by $\gamma$.
The resulting sheaf 
descends to a $\dg$-equivariant sheaf $\wh(M)$ on $\WH/J\cong \greg$ and
the equivalence $\qcoh^{J}(\fcc)\iso\qcoh^{\dg}(\greg)$ sends $M$ to
$\wh(M)$.

An important special case is $M=\mu_*\oo_X$ where
$X$ is an  affine  Poisson scheme and  $\mu: X\to \fcc$
is the  moment map  for a Hamiltonian
$J$-action on $X$.
In this case, $\wh(M)$ is an $\oo_\greg $-algebra
and the relative spectrum of that algebra 
is a  quasi-affine scheme $\WH(X):=\WH\times_\fcc ^JX$ that may be viewed as a
twist of $X$ by the $\vpi^*J$-torsor $\WH\to \greg$.
The scheme $\WH(X)$ comes  equipped
with a  Poisson structure and a Hamiltonian $\dg$-action with moment map $\WH(X)\to\greg$.
If $X$ is a smooth  symplectic variety then so is $\WH(X)$.

Write
$\sg$ for the algebra $\sym\dgl$
viewed as a dg algebra with   zero differential and
such that $\sym^i\dgl$ is assigned degree $2i$.
Let  
$D^\dg_{\text{perf}}(\sg)$ be the
perfect derived category of $\dg$-equivariant dg
modules over $\sg$, resp. $(\sym\dgl,\dg)\mmod$
the full subcategory of 
graded $\dg$-equivariant  $\sym\g$-modules, equipped with monoidal structure 
$(\td)\lo_{\sg}(\td)$, resp. $(\td)\o_{\sym\dgl}(\td)$.
For any dominant weight $\la$
one has a free  $\dg$-equivariant  $\sym\g$-module
$\sg\o V_\la$, where $V_\la$ is an irreducible representation of $\dg$ with highest
weight $\la$.

Let   $\Gr_\gc=\gc(K)/\gc(O)$ be  the  affine  Grassmannian where $K$, resp $O\sset K$,
is the field of formal Laurent, resp. power, series
either  over $\mathbb C$ or $\overline{\mathbb F}_q$.
 Let $D_{\gc(O)}(\Gr_G)$ be the equivariant derived
    Satake category equipped with  convolution monoidal structure.
   Bezrukavnikov and Finkelberg \cite{BF}
 established  a monoidal 
 equivalence, the derived Satake equivalence:
 \beq{Psi} D_{\gc(O)}(\Gr_G)\iso D^\dg_{\text{perf}}(\sg),\quad
 \IC(\ogr{\la})\mto \sg\o V_\la.
 \eeq
 It will be more convenient for our purposes to use
a slightly different equivalence $\Phi$
introduced in \cite[Sect.\,5(ii)]{BFN2}.
For  $\cf\in D_{\gc(O)}(\Gr_G)$, let $\HF(\cf)\in (\sym\dgl, \dg)\mmod$ denote the  cohomology of $\Phi(\cf)$. We also consider
the equivariant cohomology  $H^\hdot_{\gc(O)}(\cf)$,
a finitely generated module over the algebra $H^\hdot_{\gc(O)}(\pt)\cong \C[\fcc]$.
  Let $\chh\Phi(\cf)\in \qcoh^\dg(\dgl^*)$, resp.
 $\chh_{\gc(O)}(\cf)\in \coh(\fcc)$, denote 
 an associated   quasi-coherent sheaf on $\dgl^*$,
 resp. coherent sheaf   on $\fcc$.
 The construction of \cite[Sect. 5.1-5.3]{YZ} equips the sheaf $\chh_{\gc(O)}(\cf)$
 with a $J$-equivariant structure.

We say that an object $\cf\in\sat$ is $!$-pure, resp. pure,
            if $\cf$ has a lift
           $\tilde \cf\in D^{\text{mix}}_{\gc(O)}(\Gr_G)$ which is $!$-pure, resp. pure,
            with respect to   the stratification  by the $\gc(O)$-orbits, cf. Sect. \ref{gen} for more details.

  \begin{thm}\label{sat thm}
    If $\cf\in D_{\gc(O)}(\Gr_G)$ is $!$-pure     then the dg module $\bx(\cf)$  is formal, i.e,
      quasi-isomorphic to its cohomology equipped with zero differential,
     and there is a canonical  isomorphism
      \[\chh\bx(\cf)\cong \jmath_*\wh(\chh_{\gc(O)}(\cf))   \]
      in $\qcoh^{\dg}(\dgl^*)$,  where  $\jmath : \greg\into \dgl^*$ is the open imbedding and $\jmath_*$
      a {\em nonderived} pushforward.
    \end{thm}

      The theorem shows that  the Satake image of a $!$-pure object
  $\cf$  is  completely determined by 
   the $\gc(O)$-equivariant cohomology of $\cf$ and that
      the adjunction $\chh\bx(\cf)\to \jmath_*\jmath^*\chh\bx(\cf)$
    is an isomorphism.
    
    \begin{rems} \vi 
The functor $M\mto \jmath_*\wh(M)$ is a right adjoint of the functor of Kostant-Whittaker reduction considered in \cite{BF}, cf. also \cite{Ga}, Sect. 2.3 and \cite{GG}, Sect. 3.
      
    \vii Most of the results of this paper have `quantum analogues' for
    the category $D_{\gm\ltimes G(O)}(\Gr)$ that includes equivariance  with respect to   loop rotation.
    The proofs in the  quantum setting  are essentially the same as in the classical setting.
    However,  the statements of quantum analogues of   theorems of the  paper
    involve nil DAHA and related bialgebroid structures introduced in \cite{GG}. Therefore,
    we decided to omit discussion of  the quantum case in the present paper.
  \end{rems}

  \subsection{Ring objects}
    We will consider commutative ring objects $\cf$ of an ind-completion  $\indd$
    with respect to $!$-restrictions, cf. eg.  \cite[Sect. 2(i)]{BFN2} and also Sect. \ref{gen} below.
    If $\cf$ is $!$-pure, hence equivariantly formal, 
    the cohomology
    $\HF(\cf)$, resp. $H_{G(O)}^\hdot(\cf)$,
             has the structure of a supercommutative 
             algebra, by the equivariant K\"unneth formula.  Thus, there is an affine
             $J$-scheme
             $\Spec H_{G(O)}^{\text{even}}(\cf)$ over $\fcc$,    resp. $\dg$-scheme
             $\Spec\hrm^{\text{even}}\Phi(\cf)$ over $\dgl^*$.
                                 Theorem \ref{fin} below gives a 
          sufficient condition  for these schemes be
          of finite type in terms of the
          restriction of $\cf$ to the affine Grassmannian of a maximal torus $T\sset G$.
          The scheme $\Spec H_{G(O)}^{\text{even}}(\cf)$
          is often easier to study than $\Spec\hrm^{\text{even}}\Phi(\cf)$. 
          Theorem \ref{fin}  provides a construction
          of $\Spec\hrm^{\text{even}}\Phi(\cf)$ in terms of  $\Spec H_{G(O)}^{\text{even}}(\cf)$
          which allows, in many cases, to derive algebro-geometric properties
          of the former from similar properties of the latter.

          Let  $\rho: \hv\to \gc$ be a morphism of reductive groups
          and $\rho_\Gr: \Gr_\hv\to \Gr_\gc$
 an  induced
   ind-proper morphism.
 Let  $\rho^\dagger$ be the composite  of the following  functors
    \[
            D_{\gc(O)}(\Gr_{\gc})\ \xrightarrow{\en\text{Obl}^{\gc(O)}_{\hv(O)}\en}\,
      D_{\hv(O)}(\Gr_{\gc}) \ \xrightarrow{\;\rho_\Gr^!\;}\
      D_{\hv(O)}(\Gr_{\hv}).
    \]

                     It was  observed  in  \cite[Sect. 2(vi)]{BFN2} in a special case
      and proved in \cite{Ma} in full generality
      that  the functor $\rho^\dagger$ has a lax monoidal structure, in particular,
      it sends ring objects to ring objects. 
      The notion of $!$-purity (as opposed to $*$-purity) makes sense for  ind-objects
    and  $\rho^\dagger$  takes
           $!$-pure objects to $!$-pure objects.

     To state  Theorem \ref{fin},
     let         $\BX=\mathbb{X}_*(\tv)$ be the cocharacter lattice of  $T$
     and
       $\dom\sset\BX$  the dominant Weyl chamber.
      The natural imbedding $\veps: \BX\cong\Gr_T\into \Gr_G$
    identifies $\BX$ with the $T$-fixed point set $(\Gr_G)^T$.  
For  $\cf\in\Ind D_{\gc(O)}(\Gr_G)$, we  write  $\veps_\la^\dagger\cf$
     for the stalk of $\veps^\dagger\cf\in D_{T(O)}(\Gr_T)$ at $\la\in\BX$, so
     $H^\hdot_\tv(\veps^\dagger\cf)=\bigoplus_{\la\in\BX}\,H^\hdot_\tv(\veps_\la^\dagger\cf)$.
     If $\cf$ is a commutative
  ring object then $H^\hdot_\tv(\veps^\dagger\cf)$
      is an $\BX$-graded supercommutative algebra      and we define
      an $\dom$-graded subalgebra of $H^\hdot_\tv(\veps^\dagger\cf)$
      as follows

\beq{winv}
H(\veps^\dagger\cf)_+:=\bplus_{\la\in\dom}\  H^\hdot_\tv(\veps_\la^\dagger\cf).
\eeq

Let $\bar X:=\Spec \C[X]$ denote the  affine closure
  of a quasi-affine scheme ~$X$.

\begin{thm}\label{fin} If $\cf\in\Ind D_{\gc(O)}(\Gr_G)$ is a  $!$-pure commutative
  ring object such that the algebra $H(\veps^\dagger\cf)_+$  is finitely generated
  then   $\Spec H^{\text{\em even}}_{G(O)}(\cf)$ and  $\Spec \hrm^{\text{\em even}}\Phi(\cf)$
  are     schemes of finite type and
    there is an isomorphism of $\dg$-schemes over $\dgl^*$:
\[
    \Spec  \hrm^{\text{\em even}}\Phi(\cf)\ccong 
    \overline{\WH(\Spec{H^{\text{\em even}}_{G(O)}(\cf)})}.
  \]

 In particular,   $\Spec \hrm^{\text{\em even}}\Phi(\cf)$ is
      reduced  irreducible and normal
 if      so is $\Spec H^{\text{\em even}}_{G(O)}(\cf)$.
      \end{thm}

         \begin{rems}  \vi A    supercommutative algebra $A=A^{\text{even}} \oplus A^{\text{odd}}$
  is finitely generated if and only if $A^{\text{even}}$ is a finitely generated algebra and
  $A^{\text{odd}}$ is a finitely generated $A^{\text{even}}$-module.

  \vii All  ring objects $\cf$  that will appear in applications considered in this  paper
  come from ring objects of
  $D_{\gm\ltimes G(O)}(\Gr)$, the $\gm\ltimes G(O)$-equivariant derived category where $\gm$ acts by loop
  rotation. In that case the cohomology groups $H^\hdot_{G(O)}(\cf)$ acquire
  a natural Poisson
  structure that will appear without further references
  in some of the statements in subsequent sections.
  \end{rems}

\begin{examp} Let  $\jmath_\la: \Gr_\la\into \Gr$
  be  the locally closed imbedding of the $G(O)$-orbit associated with
  $\la\in \BX_+$, and   $\k_{\Gr_\la}$ the constant sheaf on $\Gr_\la$.
  The direct sum $\cf:=\oplus_{\la\in\BX_+}\, (\jmath_\la)_*\k_{\Gr_\la}$
  has the natural structure of a ring object.
  Combining Theorem \ref{fin} with the main result of \cite{GK}, it is not difficult to deduce
  that the algebra $\HF(\cf)$ is isomorphic to a subalgebra
  of $\C[T^*(\dg/\du)]$ generated by the algebras $\sym\dgl$ and $p^*\C[\dg/\du]$,
  where $p: T^*(\dg/\du)\to \dg/\du$ is the projection.
  \end{examp}

  \subsection{Organization of the paper} In Section \ref{duality}
  we state Theorem \ref{vincent} which relates derived Satake equivalences $\Phi_H$ and $\Phi_G$
  for a morphism $H\to G$ of reductive groups. The proof of Theorem \ref{vincent}
  will be given in Section \ref{iwahori sec}. In  Section \ref{symp sec}
  we discuss applications of  Theorem \ref{vincent} to symplectic duality
  introduced in the fundamental work \cite{BZSV}.
  In Sect. \ref{ssymp sec} we
  give a description of the symplectic dual of  $T^*(G/H)$
  for any connected reductive subgroup $H\sset G$.
  The main result of  Section \ref{Coul sec} is Theorem \ref{c thm} on the structure
  of the symplectic dual $\bm^!_G$ of a symplectic representation $\bm$.
  The theorem provides a description of $\bm^!_G$ in terms of the corresponding Coulomb
  branch ${\mathcal C}_{\bm,G}$;   it follows  that $\bm^!_G$ 
 is an irreducible normal variety; furthermore,
 $\bm^!_G$ has symplectic singularities if $\bm=T^*\bn$  is of cotangent type.
 We also relate our  results to the results 
  of Gannon-Webster \cite{GW} and  Teleman \cite{T};
  in  Remark \ref{rem bk} we  indicate a 
  connection with works of
  Braverman-Kazhdan  \cite{BK1}, \cite{BK2}.
  The proof of Theorem  \ref{c thm} is given in Sect. \ref{c pf}.
  In Sect. \ref{univ} we consider the `universal  Coulomb branch'.
  Section \ref{sec2} is devoted to the proof of Theorem \ref{Vthm}.
  In Section \ref{rem sec} we review the necessary material on derived Satake 
  and the functor $\Phi$. In particular, Sect. \ref{appl sat sec}
  contains a simple construction
  of commutativity
  constraint for convolution on the abelian Satake category, which seems to be new.
  In Sect. \ref{Jsec} we discuss functoriality for  universal centralizers
  based on   a  description of the universal centralizer given in \cite{DG},
  \cite{BFM}, and \cite{Ng1}. 
  Theorem \ref{sat thm} is proved in Section \ref{pf sec1} modulo
  a criterion for finite generation of the cohomology algebra of a
  $!$-pure ring object. The criterion is proved in Section \ref{gr pf}.

  \subsection{Acknowledgements} I thank Tom Gannon for 
discussions and for providing proof of  Corollary \ref{tom}.
Tsao-Hsien  Chen kindly shared with me a draft of \cite{Ch} and attracted my attention
to the work of Macerato  \cite{Ma}.
I am grateful to Nick Rozenblyum for explaining a general approach to
proving formality of dg objects based on $\infty$-categories which is used in the
proof of Theorem \ref{sat thm}.
Above all, I am grateful to Misha Finkelberg for
reading parts of a preliminary draft of the paper and for many discussions
on derived Satake and Coulomb branches which have played an important role
in this project.

  \section{Functoriality for derived Satake}
  \label{duality}
  \subsection{}
  We write  $\Lie J_\dg$ for the Lie algebra of the
    group scheme $J_\dg$. 
    This is an $\oo_{\fcc_\dh}$-module  
     with zero Lie bracket  which is  canonically  isomorphic
     to the cotangent sheaf $T^*_{\fcc_\dg}$, \cite[Sect.\,2.6]{BF},
     \cite[Lemma 5.1.5]{GK}. 
Let      $\rr_G:=\oplus_{\la\in \dom}\,
 \IC(\ogr{\la})\o V^*_\la$ 
 be the regular perverse sheaf, a ring object of $\Ind D_{G(O)}(\Gr_G)$.

 Let  $H$ and $G$  be  a   pair  of connected reductive groups.
  To simplify  notation we write  $S\dgl=\sg$, resp.
  $S\dhl=\sym(\dhl[-2])$, and
  use   subscripts  $\dh$ and $\dg$ to distinguish between objects
    associated with $\dh$ and $\dg$, respectively.

    Given  a  morphism  $\rho: \hv\to \gc$, we would like to relate
   derived
   Satake equivalences  $\Phi_G$ and $\Phi_H$.
   To this end,
observe first that the  map  $\rho$ induces  a morphism 
    from $\fcc_\dh=\dhl^*/\!/\dh=\h/\!/H$ to $\fcc_\dg=\dgl^*/\!/\dg=\g/\!/G$,
    and   we write $\rho^*(\td)=\fcc_\dh\times_{\fcc_\dg}(\td)$ 
  for base change.
  Pullback via the differential of the map $\fcc_\dh\to \fcc_\dg$
 gives a morphism
 $\rho^*T^*_{\fcc_\dg}\to T^*_{\fcc_\dh}$ that may be  viewed
  as a morphism
 $\rho^*\Lie J_\dg\to\Lie J_\dh$
 of abelian Lie algebras. The latter morphism
 can  be exponentiated to a morphism $\rho^*J_\dg \to J_\dh$
 of group schemes over $\fcc_\dh$, cf. Sect. \ref{Jsec}.

  Next, let $\varrho: \dh\into \dh\times\dg$ be
  the graph imbedding $h\mto (h, \rho(h))$. Then, we have the graph
  imbedding $\fcc_\dh\into \fcc_{\dh\times\dg}$ and the above construction yields
 a  morphism $\vrho^*J_{\dh\times\dg}\to J_\dh$
 of group schemes over $\fcc_\dh$.
 The action of $J_\dh$ on itself by translations, resp. on $\WW_\dh$,
 induces, via the morphism of group schemes, a $\vrho^*J_{\dh\times\dg}$-action
 on $J_\dh$, resp. $\WW_\dh$.
 Alternatively, one may  view $J_\dh$, resp.  $\WW_\dh$,
 as  a $J_{\dh\times\dg}$-scheme over  $\fcc_{\dh\times\dg}$
    via the graph imbedding $\fcc_\dh\into \fcc_{\dh\times\dg}$.

    Following V. Lafforgue   \cite{La}, we define
  \beq{vin0}
    \WHG:= \WW_{\dh\times\dg}(J_\dh)\,=\, J_\dh \times^{\varrho^*J_{\dh\times\dg}}_{\fcc_\dh}\varrho^*\WW_{\dh\times\dg}
    \,=\,T^*_\psi(\dh/U_\dh)\times_{\fcc_\dh}^{\rho^*J_\dg}\rho^*T^*_\psi(\dg/U_\dg),    
      \eeq
where we have
  used the same symbol $\psi$ for a nondegenerate character of
  either the group $U_\dh$ or $U_\dg$.
        By construction, $\WHG$ is a  quasi-affine smooth symplectic variety
             equipped with a hamiltonian action of the group $\dh\times\dg$ with moment map $\WHG\to\dhl_r^*\times\greg$.

 \begin{thm}[Functoriality for derived Satake]
   \label{vincent}  \vi For any morphism $\rho: H\to G$ of connected reductive
   groups,
      there is an isomorphism of functors 
      \[\th_{H\to G}:\, (\Phi_H\ccirc\rho^\dagger)(\td)\,\iso\,
        (\C[\WHG]\,\lo_{S\dgl}\,\Phi_G(\td))^\dg
      \]
      that makes
      the following
     diagram commute:
     \[
       \xymatrix{
         \Ind D_{G(O)}(\Gr_G) \  \ar[d]^<>(0.5){\Phi_G}\ar[rrrrr]^<>(0.5){\rho^\dagger}&&&&&
       \     \Ind  D_{H(O)}(\Gr_H)\  \ar[d]^<>(0.5){\Phi_H}\\
           \Ind D^\dg_{\text{\em perf}}(S\dgl)\  
           \ar[rrrrr]^<>(0.5){\en (\C[\WHG]\,\lo_{S\dgl}\,\td)^\dg\en }
           &&&&& \ \Ind   D^\dh_{\text{\em perf}}(S\dhl).\  
         }
       \]

       \vii The algebra
     $\C[\WHG]$, resp.   $\C[J_\dh \times_{\rho^*J_\dg} \rho^*\WH]$,
        is a  finitely generated integrally closed domain
     and  there  is an  isomorphism
     \begin{align}
       \Spec\HF_H(\rho^\dagger \rr_{\gc})&\cong\overline{\WHG},\en\text{\em resp.}\label{vin1}\\
\Spec H^\hdot_{H(O)}(\rho^\dagger\rr_G)&\cong\overline{\WH(J_\dh)},\label{vin2}
      \end{align}
     of affine Poisson varieties with a Hamiltonian  $\dh\times\dg$-action,
     resp. $\dg$-action.
               \vskip3pt

\viii          The isomorphism of functors  in {\em (i)}  is compatible with composition as follows.
Let  ${G_1}\xrightarrow{\rho_{12}} {G_2}\xrightarrow{{\rho_{23}}} {G_3}$
 be morphisms of connected reductive groups
       and    $\rho_{13}:=\rho_{23}\ccirc {\rho_{12}}$. Then,
          there is a $\dg_1\times \check {G_3}$-equivariant dg algebra isomorphism
          \[
            \th_{\text{\em alg}}:\
            (\C[\WW_{{G_1}\to {G_2}}]\,\lo_{S\dgl_2}\,
            \C[\WW_{{G_2}\to {G_3}}])^{\dg_2}\ \iso \  \C[\WW_{{G_1}\to {G_3}}],
          \]
          where the algebra $\C[\WW_{{G_1}\to {G_3}}] $ is equipped with zero differential,
          such that for all $\cf\in D_{G_3}(\Gr_{G_3})$ the following diagram commutes
           \[
       \xymatrix{
         \big(\C[\WW_{{G_1}\to {G_2}}]\,\lo_{S\dgl_2}\,(\C[\WW_{{G_2}\to {G_3}}]\,\lo_{S\dgl_3}
         \,\Phi_{G_3}(\cf))^{\dg_3}\big)^{\dg_2}
         \ar[rr]^<>(0.5){\th_\text{\em alg}\o \Id} \ar[d]^<>(0.5){\Id\o \th_{{G_2}\to {G_3}}}
          && (\C[\WW_{{G_1}\to {G_3}}]\,\lo_{S\dgl_3}\, \Phi_{G_3}(\cf))^\dg
           \ar[d]^<>(0.5){\th_{{G_1}\to {G_3}}}
           \\
             \big(\C[\WW_{{G_1}\to {G_2}}]\,\lo_{S\dgl_2}\,\Phi_{G_2}({\rho_{23}}^\dagger\cf)\big)^{\dg_2}
             \ar[rr]^<>(0.5){\th_{{G_1}\to {G_2}}}&&
             \Phi_{G_1}(\rho_{12}^\dagger{\rho_{23}}^\dagger\cf).
           }
           \]
         \end{thm}

  \begin{rems}\label{fin rem} \vi The   variety $\overline{\WHG}$ was considered
      by Lafforgue \cite{La} in connection with Langlands functoriality in the dual setting of a morphism $\dh\to \dg$ rather than $H\to G$.
    \vskip 2pt  

      \vii If $\cf\in \sat$ 
    is a commutative ring object then the isomorphism in (i)
    may be interpreted as an isomorphism
    \beq{dg scheme}
    \Spec\Phi_H(\rho^\dagger\cf)\ccong \big(\overline{\WHG}\               \stackrel{L}{\times}_{\check{\g}^*}\
      \Spec\Phi_G(\cf)\big)/\!/\dg,
    \eeq
    of affine {\em derived}  schemes.
    If $\cf$ comes from a ring object of $D_{\gm\ltimes G(O)}(\Gr_G)$
    by forgetting loop rotation then the right hand side of the above isomorphism
    may be viewed as a derived
    Hamiltonian reduction  with respect to the diagonal action
    of $\dg$, where the sign of the
    Poisson structure on the factor $\overline{\WHG}$ is changed.
    Such an  isomorphism agrees
               with expectations coming from  the QTFT considered by Moore and Tachikawa
               \cite{MT}, cf. also ~\cite{KRS}.
               If $\cf$ is $!$-pure  then  the objects $\Phi_G(\cf)$ and  $\Phi_H(\rho^\dagger\cf)$
               are formal, therefore, the
    derived scheme in the right hand side of \eqref{dg scheme}
        is  a nonderived (super)-scheme with coordinate
    ring $(\mathrm{Tor}^{S\dgl}(\k[\WHG],\, \HF_G(\cf)))^\dg$.
   \vskip 3pt

               \viii         In the special case where $H$ is a Levi subgroup of $G$
 isomorphism \eqref{vin2} was proved by Macerato \cite{Ma}.
 Our isomorphism is, in general,  different from the one in \cite{Ma}
 since our functor $\rho^\dagger $ doesn't involve {\em shearing}.
 This is because we construct \eqref{vin2} by describing the image of
 the adjunction $H^\hdot_{H(O)}(\rho_!\rho^\dagger\rr_G)\to H^\hdot_{H(O)}(\rr_G)$
 and this map is incompatible with shearings. If
  $H=T$ is the maximal torus then we have  $J_\dt=T^*(\dt)$ and
 $\ol{\WH(J_\dt)}\cong \ol{T^*(\dg/U_\dg)}$.
   In this case,  \eqref{vin2}  is equivalent to an algebra isomorphism
   $f: H^\hdot_{T(O)}(\rho^\dagger\rr_G)\iso \C[T^*(\dg/U_\dg)]$.
  For a simply connected group $G$  the isomorphism has been
          constructed earlier in
          ~\cite{GK} using the same map $f$.
          On the other hand, the isomorphism constructed  in \cite{Ma} (in the more
          general case of an arbitrary Levi) is given by a slightly different map $\tilde f$
          which agrees with parabolic restriction, as in  \cite{GR}. 
          The map  $\tilde f\ccirc f\inv$ is an automorphism 
          of the algebra $\C[T^*(\dg/U_\dg)]$ given by
          the action of a certain element of  $\dt\times\gm$, where the
          $\dt$-action, resp. $\gm$-action, is induced by the $\dt$-action on
          $\dg/U_\dg$ by right translations, resp. $\gm$-action by dilations
          along the fibers of the cotangent bundle.
             \vskip 3pt

  \iv  Theorem \ref{vincent} has a quantum counterpart 
                      for the derived Satake category  with loop rotation.
                      The algebra $\C[\WHG]$ has a natural quantization defined as follows,
                      cf. \cite[Remark 2-h]{La}.
%
                      Let $Z_\hb(\dgl)$ be the center
                      of the asymptotic enveloping algebra $U_\hb(\dgl)$ and
                      $\dd^{\psi_\ell}_\hb(\dg/U_\dg)$ the $\chb$-algebra of asymptotic Whittaker differential operators
                      on $\dg/U_\dg$,
                      cf. eg.  \cite[Sect. 3.3]{GG} for more details and notation. There are also algebras $Z_\hb(\dhl)$ and
                      $\dd^{\psi_r}_\hb(\dh/U_\dh)$ defined in a similar way.
                      The morphism $H\to G$ induces a homomorphism
                      $Z_\hb(\dgl)\to Z_\hb(\dhl)$, so  we get the composite algebra 
                      homomorphism $Z_\hb(\dgl)\to Z_\hb(\dhl)\to U_\hb(\dhl)\to 
                      \dd^{\psi_r}_\hb(\dh/U_\dh)$.
                      The desired quantization of $\C[\WHG]$ is 
                      a version of
                        quantum Hamiltonian   reduction (cf. \cite[Example 1.1.2(ii)]{GK})
                      defined as follows:
                                            \begin{align*}
                                              \dd_{H\to G}\  =\
                                                              \dd^{\psi_r}_\hb&(\dh/U_\dh)
                                                                                \,             \o_{Z_\hb(\dgl)}^{{\mathsf{Sph}}_\hb(\dg)}\,
                                                                                \dd^{\psi_\ell}_\hb(\dg/U_\dg)
                                                              \ \cong \
                                                                                   \big(\dd^{\psi_r}_\hb(\dh/U_\dh) \o_{Z_\hb(\dgl)}
                                      \dd^{\psi_\ell}_\hb(\dg/U_\dg)\big)^{Z_\hb(\dgl)}\\
                        &:=\
                        \{u\o u'\in \dd^{\psi_r}_\hb(\dh/U_\dh) \o_{Z_\hb(\dgl)}
                        \dd^{\psi_\ell}_\hb(\dg/U_\dg)\mid
                        (zu)\o u'=u\o (u'z)\ \forall z\in Z_\hb(\dgl)\}.
                        \end{align*}
                        Here,  ${\mathsf{Sph}}_\hb(\dg)\cong H^{\BM}_{\gm\ltimes G(O)}(\Gr_G)$
                        is the {\em spherical} nil DAHA
                        and in the second line we have used Sweedler notation
                        $u\o u'=\sum_i\, u_i\o u'_i$.
                        The spherical nil DAHA  has the structure of a {\em cocommutative} bialgebra and
                        the tensor product
                        $\dd^{\psi_r}_\hb(\dh/U_\dh) \o_{Z_\hb(\dgl)}
                        \dd^{\psi_\ell}_\hb(\dg/U_\dg)$
                        has the natural structure of a ${\mathsf{Sph}}_\hb(\dg)$-comodule,
                        \cite[Proposition 3.3.4]{GG}. The symbol  $\o_{Z_\hb(\dgl)}^{{\mathsf{Sph}}_\hb(\dg)}$
                        stands for  {\em balanced product}, \cite[Definition ~2.1.1]{GG}.
           \erems
\smallskip

      Here are a few examples of the variety
   $\overline{\WHG}$ for some natural  imbeddings $\rho: H\into G$:
   \vskip3pt
   
   \begin{description}
   \item[Cartesian product] Let $L$ be a reductive group
   with Lie algebra ${\mathfrak l}$       and
    $p: \WW_{\check L}\to\fcc_{\check L}$  the natural map,
    cf. Sect. \ref{sat int}.    Let  $G:=H\times L$
    and $\rho: H \into H\times L,\, h\mto (h,1_L)$,
     the imbedding.      Then,  we find
    \[
     \WW_{H\to G}=\WW_\dh\times_{\fcc_\dh\times\{0\}}^{J_\dh\times J_{\check L}}(\WW_\dh\times\WW_{\check L})\,=\,
     (\WW_\dh\times_{\fcc_\dh}^{J_\dh} \WW_\dh)\times (p\inv(0))/J_{\check L}.
   \]
        It is known, cf. e.g. \cite[Sect.\,5]{BFN2},  that
     $\overline{\,\WW_\dh\times_{\fcc_\dh}^{J_\dh} \WW_\dh\,}\cong T^*\dh$ (Moore-Tachikawa
     isomorphism). Further, writing ${\mathcal N}_{\check L}\sset
     {\check{\mathfrak l}}^*$ for the nilpotent cone, one has
     $(p\inv(0))/J_{\check L}\cong {\mathcal N}_{\check L}\cap {\check{\mathfrak l}}^*_r$,
     the principal nilpotent orbit.
     Hence, $\overline{(p\inv(0))/J_{\check L}}\cong {\mathcal N}_{\check L}$.

     We conclude that
       $\overline{\WW_{H\to G}}\cong T^*{\dh} \times {\mathcal N}_{\check L}$,
       and we obtain
       \[
         (\C[\WW_{H\to G}]\lo_{S\dgl}\; \td)^\dg
         =(\C[T^*\dh]\o \C[{\mathcal N}_{\check L}])\lo_{S\dhl\o S{\check{\mathfrak l}}}
         \; \td)^{\dh\times\check{L}}=
         (\C[{\mathcal N}_{\check L}]\,\lo_{S{\check{\mathfrak l}}}\; \td)^{\check{L}}.
             \]
              \vskip 3pt

        \item[Sicilian theory] Let $\rho: G \into G^n=
        G\times G\times\ldots \times G$ ($n$ factors) be the diagonal imbedding.
        Then, we have
        $\rho^\dagger_{\Gr}\rr_{G^n}=\rr_G\so\ldots\so \rr_G$.
        In this case, the existence of the variety $\overline{\WW_{G\to G^n}}$ was   conjectured
        by Moore and Tachikawa \cite{MT}.
        It follows from Theorem \ref{vincent}  that we have
         \[
     \overline{\WW_{G\to G^n}}\cong\Spec \HF(\rho^\dagger_{\Gr_G}\rr_{G^n})\cong
      \overline{T^*_\psi(\dg/U_\dg)^{\times_\fcc (n+1)}/J_n},
    \]
   were, $T^*_\psi(\dg/U_\dg)^{\times_\fcc (n+1)}:=
T^*_\psi(\dg/U_\dg)\times_{\fcc_\dg}\ldots \times _{\fcc_\dg}
T^*_\psi(\dg/U_\dg)$ ($n+1$ factors)
and $J_n$ is the kernel of the  multiplication map
$J_\dg\times_{\fcc_\dg}\ldots \times _{\fcc_\dg} J_\dg\to J_\dg,\,
(\gamma_1,\ldots,\gamma_{n+1})\mto \gamma_1\cdots\gamma_{n+1}$.
The  proof of the above
 isomorphism was found  in \cite[Sect. 5]{BFN2}. 
         The fact that the algebra $\C[T^*_\psi(\dg/U_\dg)^{\times_\fcc n}/J_{n-1}]$
        is  finitely generated  was  proved later by
        Arakawa \cite{A} using representation theory of
        affine $\mathcal W$-algebras. 
\vskip 3pt

\item[Levi subgroups]  Let $H\into G$  be the imbedding of a Levi subgroup
  of   a parabolic subgroup  $P\sset G$. Let   $U_{\check P}$  be the unipotent radical
  of
 ${\check P}$.  Then
     it was shown by  Gannon \cite{Ga} (and predicted in \cite{La})
     that one has 
     $\overline{\WHG}\cong \overline{T^*(\dg/U_{\check P})}$.
          The action of $\dh$ in 
     $\overline{T^*(\dg/U_{\check P})}$ is induced by the action of ${\check P}/U_{\check P}$
     on $\dg/U_{\check P}$ by right
     translations.
     Using the      isomorphism
     $\HF(\rr_G)\cong   \C[\WHG]$ of Theorem \ref{vincent}(ii)
     (which, in this case, easily follows from
       \cite{Ma}),  we obtain  $\Spec \HF_H(\rho^\dagger\rr_G)\cong
     \overline{T^*(\dg/U_{\check P})}$.
             \vskip 3pt

      \item[Theta correspondence (following V. Lafforgue)] Let $(V,\om)$,
        resp. $(E, \be)$, be a symplectic, resp.  quadratic, vector space,
        so the vector space $E\o V$ has the natural symplectic structure.
        In either of the following two cases one has   $\dh=SO(E)$ and  $\dg=Sp(V)$:
        \begin{enumerate}
        \item $H=SO_{2n} \into G=SO_{2n+1}$, where $2n=\dim E=\dim V$;
                \item $H=SO_{2n+1} \into G=SO_{2n+2}$, where $2n=\dim E-2=\dim V$.
          \end{enumerate}
                   In these cases, by \cite{La}  there is an isomorphism
          $\overline{\WHG}\cong E\o V$ such that  the moment map is the natural map
          $            E\o V\to \dhl^*\times \dgl^* \cong \wedge^2E\times \sym^2V$ given by
          \[
            \mbox{$\sum_i$}\, e_i\o v_i\,\mto\, \big(\mbox{$\sum_{k<\ell}$}\, \om(v_k,v_\ell) e_k\wedge e_\ell,\;
            \mbox{$\sum_{k,\ell}$}\, \be(e_k,e_\ell) v_kv_\ell\big).
             \]
           \end{description}
           \medskip
           
 From  Theorem \ref{vincent}(ii)
             and the isomorphism $\HF(\rr_G)\cong \C[T^*(\dg/U_{\check P})]$
             in the Levi case above, changing the notation from $\dg$ to $G$, we deduce 
             \begin{cor}\label{GU cor} Let $P$ be a parabolic  of a connected
               reductive group $G$ and $U_P$ the unipotent radical of $P$. Then
               the algebra $\C[T^*(G/U_P)]$
               is  finitely generated.\qed
               \end{cor}
           \begin{rems}
             \vi    The statement of the above corollary was previously known
             in the case where $P$ is a Borel subgroup, \cite{GR},
             and in the case of an arbitrary parabolic of
             a group  of  type $A$.  In type $A$,  the varieties $\ol{\WHG}$ that arise either from
           Sicilian theories or Levi subgroups
                     may be realized as  Coulomb branches,
                     see \cite{BFN2}, resp. \cite{GW}. The finite generation of  Coulomb branches
                     has been proved in \cite{BFN1}.

         \vii  Let  ${\check B}\sset \check P$ be  a Borel subgroup,
         $\fu_{\check P}=\Lie U_{\check P}$, and ${\mathbb P}(\fu_{\check P}^\perp)$
          the projectivization  of the vector space $\fu_{\check P}^\perp\sset \dgl^*$.
         By a recent result  \cite[Theorem 1.6(i)]{FL},
Corollary \ref{GU cor} implies  that the variety
$\ol{\dg\times_{\check B}{\mathbb P}(\fu_{\check P}^\perp)}$
is a Mori dream space.
       \vskip2pt

                      %

             \viii  According to \cite[Theorem 1.6(ii)]{FL}, the variety
          $\overline{T^*(\dg/U_{\check P})}$ has 
                     symplectic singularities  if and only if  $\dg\times_{\check B}{\mathbb P}(\fu_{\check P}^\perp)$
                     is of {\em Fano type} in the sense of {\em op cit}.
                     In the special case where $\check P=\check B$ is a   Borel subgroup,
                     $\overline{T^*(\dg/U_{\check B})}$ has 
                     symplectic singularities, by \cite{Ga}.
                     Also, it was shown by Bellamy  \cite{Bel} that all Coulomb branches have
                     symplectic singularities.

                     \begin{quest}\label{symp sing}
                       Does $\overline{\WHG}$ have symplectic singularities ?
                       \hfill$\lozenge$
           \end{quest}
 \end{rems}

The moment maps $\mu_{\dh}$ and $\mu_{\dg}$ for the actions in $\WHG$
of $\dh$ and $\dg$, respectively, fit into a  commutative diagram
    \[
      \xymatrix{
        \dhl^*_{r}\  \ar[d]_<>(0.5){ \ \vpi_\dh \ }   &\ \WHG\  \ar[l]_<>(0.5){ \ \mu_{\dh}\ }
        \ar[r]^<>(0.5){ \ \mu_{\dg}\ } & \ \greg \ar[d]^<>(0.5){ \ \vpi_\dg \ }\\
        \fcc_\dh \   \ar[rr]^<>(0.5){ \ \rho \ } &&\  \fcc_\dg
}
    \]
       The image of the map  $\mu_{\dh}\times \mu_{\dg}$, resp. $ \mu_{\dg}$,  equals
       $\dhl^*_{r} \times_{\fcc_\dg}\greg$, resp. $\rho^*\greg=\fcc_\dh\times_{\fcc_\dg}\greg$.
       We obtain a chain of maps
                     \beq{WHGtors}
\WHG\  \xrightarrow{\  \wt\mu\ }\   \dhl^*_{r}\times_{\fcc_\dh} \rho^*\greg\  
\xrightarrow{\  p_2\ } \  \rho^*\greg\  \xrightarrow{\  \rho^*\vpi_\dg \ }\   \rho\inv(\fcc_\dg)\into
\fcc_\dh,
         \eeq
         where      the  map $\wt\mu$ is induced by $\mu_{\dh}\times \mu_{\dg}$
       and $p_2$ is the second
       projection.  The following   statements may be helpful
       for understanding the geometry of $\WH$:

         \vbox{\begin{align}      
&\text{$\bullet\en$  The map
  $\tilde\mu$ is an $\dh\times\dg$-equivariant  $\wt\vpi^*J_\dh $-torsor,
  where  $\tilde\vpi=\rho^*\vpi_\dg\ccirc p_2$;}\label{torsor1}\\
&\text{$\bullet\en$ The map $p_2\ccirc\tilde\mu$ is a $\dg$-equivariant  $\dh$-torsor.}
  \label{torsor2}\\
&\text{$\bullet\en$    There are isomorphisms
       $\;{\overline{\mu_\dg}}\inv(\greg)\cong \WHG \cong (\dh\times \fcc_\dh\times\dg)/\rho^*J_\dg$,
  where}\label{torsor3}\\
           &\text{\quad\ $\overline{\mu_\dg}: \overline{\WHG}\to\dgl^*$ is an extension of
             $\mu_\dg$ to the affine closure and $\dh$, resp. $\dg$, }\nonumber\\
           &\text{\quad\  acts on the first, resp. third, factor of
             $\dh\times \fcc_\dh\times\dg$.}\nonumber\\
&\text{$\bullet\en$  We have an equality\en\ 
  $\dim\WHG=(\dim\dh+\rk\dh)+(\dim\dg-\rk\dg).$}\label{torsor4}\\
&\text{$\bullet\en$    There is an isomorphism
  of  $\dh$-varieties, resp. $\dg$-varieties:}\nonumber\\
&\hskip 2cm \text{     $\overline{\WHG}/\!/\dg\cong \Spec \C[\WW_\dh]^{\rho^*J_\dg}$,\en
  resp.\en 
       $\overline{\WHG}/\!/\dh\cong\fcc_\dh\times_{\fcc_\dg}\dgl^*.$}\label{torsor5}\\
 &\text{$\bullet\en$    We have
        $\overline{\WHG}/\!/(\dh\times\dg)\cong \fcc_\dh$; therefore,
   $\ol{\WHG}$ is a hyperspherical $\dh\times\dg$-variety.}\nonumber
        \end{align}
         }
 
         Here, the
         second isomorphism in \eqref{torsor3}  is proved by choosing
         a Kostant slice in $\dhl^*_r$, resp. $\greg$,
 which provides an isomorphism $\WW_\dh\cong \fcc_\dh\times \dh$, resp.
 $\WW_\dg\cong\fcc_\dg\times \dg$.
 Other claims in \eqref{torsor1}-\eqref{torsor5}, with the exception of  the
 second isomorphism in \eqref{torsor5}, easily follow from \eqref{WHGtors}
 and we omit the proofs.

                      \begin{proof}[Proof of the second isomorphism in \eqref{torsor5}]
                        Let 
          $\rho^*\dgl^*:=\fcc_\dh\times_{\fcc_\dg}\dgl^*$ and
   let $pr_1: \rho^*\dgl^*\to \fcc_\dh$, resp. $pr_2: \rho^*\dgl^*\to\dgl^*$, be the first, resp. second, projection.
   Since 
$\WHG\to \rho^*\greg$
   is an $\dh$-torsor, one has
     $\C[\overline{\WHG}/\!/\dh]\cong \C[\rho^*\greg]$.
  Thus, proving the desired isomorphism amounts to showing that the
   natural map $\overline{\rho^*\greg}\to \rho^*\dgl^*$
   is an isomorphism.
   To this end, observe   that  the morphism
   $pr_1$  is obtained from the adjoint quotient $\dgl^*\to \fcc_\dg$
   by base change. Hence $pr_1$ is a flat morphism such that
   the restriction of $pr_1$ to $\fcc_\dh\times_{\fcc_\dg}\greg$ is smooth.
     It follows that
   $\rho^*\dgl^*$  is a generically reduced complete intersection
   in $\fcc_\dh\times\dgl^*$ of dimension  $\dim \rho^*\dgl^*=
   \dim\fcc_\dh+ (\dim\dgl^*-\dim\fcc_\dg)$.
   We deduce that  $\rho^*\dgl^*$  is reduced and 
  the codimension of 
   $\rho^*\dgl^* \sminus \rho^*\greg$ in
   $\rho^*\dgl^*$  is $\geq2$, since a similar inequality holds for  the codimension of
   $\dgl^*\sminus\greg$ in $\dgl^*$.
   It follows that
   $\oo_{\rho^*\dgl^*}$  is a  maximal Cohen-Macaulay sheaf.
   This implies that the adjunction $\oo_{\rho^*\dgl^*}\to j_*j^*\oo_{\rho^*\dgl^*}$,
   where $j: \rho^*\greg\into \rho^*\dgl^*$ is the open imbedding, is an isomorphism.
   Hence, $\Ga(\rho^*\dgl^*,\, \oo_{\rho^*\dgl^*}) \to \Ga(\rho^*\greg,\, \oo_{\rho^*\greg})$
   is an isomorphism.\end{proof}

 Observe that
 the map $\fcc_\dh\to\fcc_\dg$ may be viewed as a map from the set of
coadjoint orbits in $\dhl_r^*$ to the set of coadjoint orbits in $\greg$.  Lafforgue
interpreted this map in terms of Hamiltonian geometry. Specifically,
we see from \eqref{torsor1}-\eqref{torsor2} that
for any coadjoint orbit $\Omega\sset \dhl^*_{r}$
         there is a unique coadjoint orbit $\Omega'\sset \greg$ such that
         $\mu_\dh\inv(\Omega)=\mu_\dg\inv(\Omega')$ and
         the map $p_2: \mu_\dg: \mu_\dh\inv(\Omega)\to \Omega'$ is an $H$-torsor
         that induces an
         isomorphism  $\mu_\dh\inv(\Omega)/\dh\iso \Omega'$
         of hamiltonian $\dg$-varieties. The assignment $\Omega\mto \Omega'$
         provides the desired hamiltonian 
     interpretation for the    map between
  the sets of  coadjoint orbits in $\dhl^*_r$ and $\greg$, respectively.

   \section{Symplectic duality}\label{symp sec}
\subsection{}\label{ssymp sec}
Let ${N}$ be
a smooth affine $G$-variety,  ${N}(K)$ the
formal loop space, resp.   ${N}(O)\subset {N}(K)$ the formal arc space.
In the so-called  placid case, specifically, assuming the  conditions
given in  \cite[Sect. 7.3.1]{BZSV} hold,  it is possible 
to define the dualizing sheaf
of ${N}(O)$ and
the {\em basic object} $i_*\om_{G(O)\back {N}(O)}$,
where we write $i$ 
either for  the imbedding
$N(O)\into N(K)$ or the induced morphism of stacks $G(O)\back N(O)\into G(O)\back N(K)$.
Associated with the basic object, there is a commutative
ring object $\aa_{{N},G}={\underline{\mathcal{E}^{\!}\textit{nd}}}_{\sat}(i_*\om_{G(O)\back {N}(O)})
\in \Ind\sat$, an internal endomorphism object with respect to the action
of the monoidal category $\sat$ on $D_{G(O)}(N(K))$, \cite[8.1]{BZSV}.
It follows from definitions that one has an isomorphism
\[\Phi_G(\aa_{{N},G})=\RHom_{D_{G(O)}(N(K))}(i_*\om_{G(O)\back {N}(O)},\, \rr_G\star
  i_*\om_{G(O)\back {N}(O)}).
  \]

  Let $M=T^*N$. The  scheme  $M_G^!:=\Spec \hrm^{\text{even}}\Phi(\aa_{{N},G})$,
  resp. ${\mathcal C}_{M,\gc}:=\Spec H^\hdot_{\gc(O)}(\aa_{M,\gc})$,
  is  usually referred to as the {\em symplectic dual}, resp.
  {\em  Coulomb branch},   of $M$.
  These schemes
  come equipped with a Poisson structure and
  a Hamiltonian action of $\dg$ and  $J_\dg$, respectively.

\begin{rem}
  In \cite{BZSV}, the authors also considered a suitably defined Whittaker setting
  where $T^*N$ is replaced by certain twisted versions  $T^*(N,\Psi)$.
  \end{rem}

  There is  an equivalent,   more explicit construction of the ring
  object $\aa_{{N},G}$ in terms of the scheme $\cz_{{N},G}$
  referred to as  the {\em variety of triples}  in \cite{BFN1}, resp. {\em relative Grassmannian}
  in \cite[8.2]{BZSV}. The scheme $\cz_{{N},G}$,
  which is a formal loop version of the
  finite dimensional Steinberg variety,
  is defined by the cartesian square in the diagram
    \beq{cart1}
    \xymatrix{
      \cz_{{N},G}\ar[rr]^<>(0.5){f_1}\ar[d]^<>(0.5){f_2} \ar @{} [drr]|{\square}&&
      G(K)\times_{G(O)} {N}(O)
      \ar[d]^<>(0.5){a_G}\ar[r]^<>(0.5){pr_1} & G(K)/G(O)=\Gr_G\\
      {N}(O)\ar@{^{(}->}[rr]^<>(0.5){i}&& {N}(K) &
    }
    \eeq
    where the map $a_G$ is the action
    $(g, n)\mto g.n$.
    Let $\om_{\cz_{{N},G}}$ be the dualizing sheaf and $q_G= pr_1\ccirc f_1:\, \cz_{{N},G}\to \Gr_G$.    
    Placidity assumtions, see \cite[Sect. 8]{BZSV},
    insure that    there is a well-defined
    (renormalized) pushforward $(q_G)_*\om_{\cz_{{N},G}}$ and there is an isomorphism
    $\aa_{{N},G}\cong(q_G)_*\om_{\cz_{{N},G}}$. It follows that
    $H^\hdot_{\gc(O)}(\aa_{N,\gc})\cong H^\hdot_{G(O)}(q_G)_*\om_{\cz_{{N},G}})
    \cong H^\BM_{G(O)}(\cz_{{N},G})$, see \cite[Sect. 2(i)-(ii)]{BFN2} in a special case
    and \cite{BZSV} in general.

  Given a morphism $\rho: H \to G$ of  reductive groups one may view the $G$-variety
  ${N}$ as an $H$-variety.  It was shown  in \cite{BFN2} (in the special case
  where ${N}$ is a representation of $G$; the
  same argument works in general, cf. \cite[Remark 8.3.2]{BZSV})
  that
  $\aa_{{N},H}  \cong\rho^\dagger\aa_{{N},G}$.

We prove the following

    \begin{prop}\label{bm pure} \vi The sheaf $\aa_{{N},G}$
      is $!$-pure if and only if  the Borel-Moore homology
      $H^{\BM}_{T(O)}(\cz_{{N},T})$ is pure, where $T$ is
  a maximal torus of $G$.

      \vii If  $\aa_{{N},G}$    is $!$-pure and the algebra $H^\hdot_{T(O)}(\cz_{{N},T})_+$
            is finitely generated
            then  the following holds: 
            \begin{enumerate}
            \item The scheme ${\mathcal C}_{M,\gc}$ is  a   scheme of finite type flat
              over $\fcc_\dg$;
              
            \item The scheme $M^!_G$  is 
            of finite type and  there are isomorphisms
                        \[
                          M_G^!\ccong \overline{\WH({\mathcal C}_{M,\gc})},\en\ \text{\em resp.}\en\
                          M_G^!/\!/\dg\,\ccong\,  {\mathcal C}_{M,\gc}/\!/J_\dg.
                          \]
\end{enumerate}
\end{prop} 

\begin{proof} Part  (i) is a direct consequence of
   the isomorphism
  $H^\hdot_{T(O)}(\rho^\dagger\aa_{{N},G})
  \cong H^{\BM}_{T(O)}(\cz_{{N},T})$ and our definition of $!$-purity.
  All statements in (ii) except for the isomorphism
  $M_G^!/\!/\dg\,\ccong\,  {\mathcal C}_{M,\gc}/\!/J_\dg$
  are consequences of  Theorem \ref{fin}. To prove the isomorphism
  observe that for any $!$-pure sheaf $\cf$, by Theorem \ref{sat thm},
  one has  canonical isomorphisms
  \[\HF_G(\cf)^\dg\,\iso\, \Ga\big(\dgl^*,\, \jmath_*\wh(\chh_{\gc(O)}(\cf))\big)^\dg  \, \iso\,
    \Ga(\fcc, {\mathcal H}_{G(O)}(\cf)^{J}).
    \]
    It is clear that  if $\cf$ is  a commutative ring object then  these
    are algebra isomorphisms. Applying this to $\cf=\aa_{{N},G}$
    we obtain
    $\C[M_G^!]^\dg=
    \HF_G(\aa_{{N},G})^\dg \cong \C[{\mathcal C}_{M,\gc}]^{J_\dg}=\C[{\mathcal C}_{M,\gc}/\!/J_\dg]$,
    as desired.
  \end{proof}

 \begin{rems} \vi From the isomorphism $M_G^!/\!/\dg\,\cong\,  {\mathcal C}_{M,\gc}/\!/J_\dg $
  we see that if
   ${\mathcal C}_{M,\gc}/\!/J_\dg\to \fcc_\dg$
  is a finite  morphism then the algebra $\C[M_G^!]^\dg$ is finitely generated as a $\C[\dgl^*]^\dg$-module.
  The latter property forces the Poisson bracket on  $\C[M_G^!]^\dg$ vanish.
  It follows that if  $M_G^!$ is an irreducible normal affine variety,
  then this variety is hyperspherical if and only if  the morphism ${\mathcal C}_{M,\gc}/\!/J_\dg\to \fcc_\dg$
  is finite.
\vskip 3pt

  \vii The isomorphism  $M_G^!/\!/\dg\,\cong\,  {\mathcal C}_{M,\gc}/\!/J_\dg $
  also implies that the natural map
    $\C[\fcc_\dg]=(\sym\dgl)^\dg\to \HF(\cf)^\dg$ is an isomorphism if and only if
    general $J_\dg$-orbits are dense in ${\mathcal C}_{M,\gc}$.
    \vskip 3pt

    \viii It is expected that in some cases where the affine $G$-variety  $N$
    is singular the role of the dualizing sheaf
    $\om_{N(O)}$ in the definition of the basic object is played
    by the sheaf $\IC(N(O))$, where the $\IC$-sheaf is
     defined using a result of Grinberg-Kazhdan and Drinfeld.
    Assuming this $\IC$-sheaf is pure one might expect  the ring object
    $\aa_{N,G}$ be $!$-pure. If that is the case then  the isomorphisms of
    Proposition \ref{bm pure}(2)  hold for the singular variety $N$.
  \erems

  In the case where  ${N}=G/H$ for a connected reductive subgroup $H\sset G$,
  there is an isomorphism
      of stacks $G(O)\backslash\cz_{G/H, G}\cong H(O)\backslash \Gr_H$
    such that the map $q_G$ corresponds to the  morphism
    $H(O)\backslash \Gr_H\to G(O)\backslash\Gr_G$ induced by the
    imbedding $\rho: H\into G$, \cite[8.2.2]{BZSV}.
    It follows that $H^{\BM}(G(O)\backslash\cz_{G/H,\, G})$
    $      \cong H^{\BM}(H(O)\backslash \Gr_H)\cong \C[J_\dh]$. Thus, there is a
    natural isomorphism
             \[{\mathcal C}_{T^*(G/H),\, \gc}\cong J_\dh.
        \]
           Further, it was shown by Chen \cite[Sect. 4.1]{Ch} that for $N=G/H$,
      using the notation of diagram \eqref{cart1},
            one has
      \begin{align}
        \RHom_{D_{G(O)}(N(K))}(i_*&\om_{G(O)\back {N}(O)},\, \rr_G\star
                                                              i_*\om_{G(O)\back {N}(O)})\label{chen iso}\\
        &\cong \RHom_{D_{G(O)}(N(O))}(\om_{G(O)\back {N}(O)},\,
        (f_2)_*q_G^!\rr_G)\cong R\Ga_{H(O)}(\Gr_H, \rho^\dagger\rr_G),\nonumber
      \end{align}
            It follows that there is an algebra isomorphism
      $\C[T^*(G/H)^!_G]\cong
      H^\hdot_{H(O)}(\rho^\dagger\rr_G)$.

      \begin{cor}\label{ch thm} There is an  isomorphism
        of $\dg$-schemes, resp. $J_\dg$-schemes:
             \[
   T^*(G/H)^!_G\, \cong\, 
   \overline{J_\dh\times_{\fcc_\dh}^{\rho^*J_\dg}\rho^*\WH},\
   \en\text{\em resp.}\en\
   T^*(G/H)^!_G/\!/\dg     \ccong \overline{J_\dh/\!/\rho^*J_\dg}.
 \]
 Moreover, $T^*(G/H)^!_G$ is an irreducible normal variety.
\end{cor}
 \begin{proof}
   The first isomorphism is a consequence of \eqref{chen iso} and \eqref{vin2}.
      The second  isomorphism  follows from algebra  isomorphisms
      \begin{align*}
        \C[J_\dh\times_{\fcc_\dh}^{\rho^*J_\dg}\rho^*\WH]^\dg&\ \cong \ 
        \k[J_\dh\times_{\fcc_\dh}\rho^*\WH]^{\dg\times \rho^*J_\dg}
    \     \cong\ 
          \k[J_\dh\times_{\fcc_\dh}\rho^*(\WH/\!/\dg)]^{\rho^*J_\dg}\\
        &\ \cong\k[J_\dh\times_{\fcc_\dh}\fcc_\dg]^{\rho^*J_\dg}\ \cong \ 
        \C[J_\dh]^{\rho^*J_\dg}\ =  \   \C[J_\dh/\!/\rho^*J_\dg].\qedhere
        \end{align*}
 \end{proof}

   \begin{rem} 
     In the case of   affine spherical varieties $G/H$ the isomorphisms of the corollary
     have been  conjectured, and  proved in several small rank cases,
          by Devalapurkar, cf. \cite[(3.14)]{Dev}.
  \end{rem}

\subsection{Coulomb branches}\label{Coul sec}
In this subsection we assume that $\C$ is the field of complex numbers
since we are going to use exponential functions.

Let $\gc\to Sp(\bm)$ be
  a  representation in a symplectic vector space $\bm$
  that satisfies the  `anomaly  cancelation'
condition from \cite{BDFRT}. 
In {\em op cit}  the authors used $\dd$-modules to associate
to the representation $\bm$ a ring object  $\aa_{\bm,G}\in\Ind\sat$.
The  anomaly  cancelation automatically
holds in the `cotangent case'  where $\bm=T^*\bn=\bn\oplus\bn^*$ and
the representation in $\bm$ is a direct sum
of a representation  $G\to GL(\bn)$ and the dual representation in $\bn^*$.
In that case
the construction of \cite{BDFRT}
reduces to earlier constructions of \cite{Na}, \cite{BFN1}, ~\cite{BFN2}.

Teleman \cite{T}, cf. also \cite[Sect. 2.1]{GW},
associated to a representation $\rho: G\to GL(\bn)$ 
a certain $W$-invariant rational section $\dtl^*\to\dt$,
It was essentially stated in {\em loc cit},
and follows from the description of
the universal centralizer due to \cite{DG}, \cite{BFM}, cf. also
\cite[Proposition 2.2.3]{Ng1}, which will be
recalled in Sect. \ref{Jsec} below,
that this section defines a rational section
$\eps_\rho: \fcc_\dg\dashrightarrow    J_\dg$.
Translation by $\eps_\rho$  gives a birational automorphism $\tau_\rho:
J_\dg\dashrightarrow  J_\dg,\, j\mto \eps_\rho\cdot j$,
where $j\mto \eps_\rho\cdot j$ denotes group multiplication in the group scheme $J_\dg$.
Thus, one obtains an automorphism $\tau^*_\rho$ of the field ${\mathbb C}(J_\dg)$
of rational functions.
Gannon and Webster  introduced
a  `{\em gluability}' condition  \cite[Definition 2.6]{GW} 
on the representation $\rho$ and proved, generalizing  work of Teleman,  that if
$\rho$ is gluable then the coordinate ring of the corresponding
Coulomb branch $\bm=\bn\oplus\bn^*$ of cotangent type  has the following description,
\cite[Theorem 2.7]{GW}:
\beq{gwd}
\C[{\mathcal C}_{\bm,\gc}] \cong \{f\in \C[J_\dg]\ \mid\  \tau_\rho^*(f)\text{ is a
    regular function on } J_\dg\}.
\eeq

Write
    $\mu_{\text{left}}$ and $\mu_{\text{right}}$ for the moment maps
    $T^*\dg\to \dgl^*$ induced by the action of $\dg$ on itself by left and
    right translations, respectively.     
Let  $\fcr:=
\{z\in T^*\dg\mid \mu_{\text{left}}(z)=\mu_{\text{right}}(z)\in\greg\}$.
This  is a smooth locally closed  $\Ad\dg$-stable subvariety of $T^*\dg$.
Using either left or right invariant trivialization $T^*\dg\cong\dg\times\dgl^*$
one obtains
$\fcr =\{(g,\xi)\in\dg\times\greg=T^*\dg\mid \Ad g(\xi)=\xi\}$.
The second projection
$\fcr\to \greg,\, (g,\xi)\mto\xi$, gives $\fcr$ the structure of an 
$\Ad\dg$-equivariant  group scheme  on $\greg$ called the group scheme of
  {\em regular centralizers}.
This  group scheme is  canonically
  isomorphic to  the group scheme $\vpi^*J_\dg$, a pullback of $J_\dg$
  via the coadjoint quotient $\vpi: \greg\to\fcc_\dg$.
Therefore, the pullback of
  the section $\eps_\rho$ along $\vpi$ gives  an $\Ad\dg$-equivariant rational section
  $\vpi^*\eps_\rho:
  \greg\dashrightarrow \vpi^*J_\dg\cong\fcr$.
  In either left or right invariant trivialization 
 of    $T^*\dg$ the map $\vpi^*\eps_\rho$
    has the form
    $\xi\mto (\gamma_\rho(\xi),\xi)$, where
    $\gamma_\rho: \greg\dashrightarrow \dg$ is a rational map
defined as a  composition of $\vpi^*\eps_\rho$ and  the natural projection
  $T^*\dg\to\dg$.
        It is not difficult to show that the image of $\vpi^*\eps_\rho$
    is an  $\Ad\dg$-stable smooth
    Lagrangian subvariety $\La_\rho\sset T^*\dg$.

  Let ${X}$ be an affine irreducible normal variety with a Hamiltonian $\dg$-action
  $\dg\ni g: x\mto g.x$, such that
  the moment map
  $\mu_X: {X}\to\greg$ is a flat dominant morphism.
 Translation by $\ga_\rho$
 gives a $\dg$-equivariant birational automorphism
  \beq{tau map}
  \boldsymbol{\tau}_{{X},\rho}: {X}\dashrightarrow {X},
  \en x\mto \gamma_\rho(\mu_X(x)).x,
  \eeq
  such that $\mu_X(\boldsymbol{\tau}_{{X},\rho}(x))=\mu_X(x)$ for all $x\in X$.
  Motivated by \eqref{gwd}, we define an affine scheme over $\dgl^*$ as follows:
  \beq{tau}
  \dc{X}:=\Spec A_{{X},\rho},\en\text{where}\en
  A_{{X},\rho}:=\{f\in \C[{X}]\ \mid\ \boldsymbol{\tau}_{{X},\rho}^*(f) \text{ is a
    regular function on } {X}\}.
\eeq
There are
two natural imbeddings $j_+, j_-: {X}\into\dc{X}$, induced by the
algebra maps $\C[{X}]\to A_{{X},\rho}$ which send $f$ to $f$, resp.
$\boldsymbol{\tau}_{{X},\rho}^*(f) $. The scheme $\dc{X}$ is,
by construction,
the affine closure of a quasi-affine scheme over $\greg$  obtained by gluing
two copies of the variety $X$ using $\boldsymbol{\tau}_{{X},\rho}$ as
a transition map.
There is a $\dg$-action on $\dc{X}$ that restricts, via  $j_+$ {\em and}  $j_-$,
to the $\dg$-action on $X$.
Similarly, there is  a $\dg$-equivariant morphism
$\dc{\mu}: \dc{X}\to \dgl^*$ such that $\dc{\mu}\ccirc j_+=\dc{\mu}\ccirc j_-=\mu$.

The proof of the following result will be given in Sect. \ref{c pf}.

  \begin{thm}\label{c thm}   For any
  representation 
  $\gc\to Sp(\bm)$ that satisfies the anomaly  cancelation condition, one has
  $\hrm^{\text{odd}}\Phi(\aa_{\bm,\gc})=0$ and the following holds:
  \begin{enumerate}
    \item  The scheme $\bm_{\gc}^!$ is an irreducible
  normal affine hyperspherical   variety  and    there is  an isomorphism 
   $\bm_{\gc}^!\ccong
   \overline{\WH({\mathcal C}_{\bm,\gc})}$,
   of Poisson varieties with Hamiltonian $\dg$-action.
   \vskip2pt

 \item If ${\bm=\bn\oplus\bn^*}$ is of cotangent type for a representation
   $\rho: G\to GL(\bn)$ then  the variety $\bm_G^!$ has
     symplectic singularities.
\vskip2pt

     \item If the representation $\rho$ in (2)
   is gluable  then there is an isomorphism
   $\bm^!_\dg\cong \dc{(\WH)}$, of $\dg$-schemes over $\dgl^*$.
 \end{enumerate}
\end{thm}

\begin{rems}\label{rem bk} \vi  The isomorphism $\bm_{\gc}^!\ccong
  \overline{\WH({\mathcal C}_{\bm,\gc})}$ in (1) 
  has been conjectured by Gannon and Webster
  \cite[Conjecture A.2]{GW}, cf. also \cite{Na2}.

  \vii  There is a  quantum counterpart of the isomorphism $\bm^!_\dg\cong \dc{(\WH)}$
    where the role of
    the algebra $\C[\WH]$
    is     played by  the algebra $\dd_\psi(\dg/U_\dg)$ of Whittaker differential operators.
        The natural quantization of the Lagrangian  subvariety $\La_\rho\sset T^*\dg$
    is a very central  holonomic $\dd$-module $\Phi_{\rho}$ on $\dg$
    considered in
    \cite[Sect. 1.3]{GG}.
    The role of
    the  automorphism $\boldsymbol{\tau}_{\WH,\rho}^*$
   is     played by  the functor of
   convolution with  $\Phi_{\rho}$.
   It was shown in \cite{GG} that the $\dd$-module  $\Phi_{\rho}$ has a canonical central
   structure and convolution with $\Phi_\rho$ is an exact functor.

     The  $\dd_\dg$-module  $\Phi_{\rho}$ is a   counterpart
    of the perverse sheaf on $\dg$, called $\gamma$-{\em sheaf},
    associated with the representation $\rho$  by Braverman and Kazhdan \cite{BK2}.
        In the setting of finite fields, the  function on $\dg({\mathbb F}_q)$ obtained from
    the  $\gamma$-sheaf by taking  trace of Frobenius
    is a generalization of the $\Gamma$-function, \cite{BK1}.
    This may be viewed as
a heuristic `explanation'  of the
    appearance of ($\hbar$-analogues of) $\Gamma$-functions in
    \cite[Sect.\,7]{T}.
  We are going to discuss      the construction of the quantization of $\bm_G^!$ in terms of
        the $\dd$-module $\Phi_\rho$, and its relation to the construction in {\em loc cit}
        in more detail  elsewhere.
\erems

\begin{examp}[{\textsf{The adjoint representation}}]
  Let a reductive group $G$, resp. $\gm$, act in $\g=\Lie G$ by the adjoint
  action, resp. by dilations.
       The Coulomb branch of cotangent type associated
   with the resulting representation of the group $G\times \gm$ in
   $\bn=\g$ is an integrable system
   $\mu_{\fcc\times{\BA^1}}: {\mathcal C}_{T^*\g,\gc\times\gm}\to\fcc\times{\BA^1}$
   and we have, cf.    \cite[Remark 6.20]{BFN1}:
   \[\mu_{\fcc\times{\BA^1}}\inv(\fcc\times\{\kappa\})\cong
     \begin{cases}
       (T^*(\dt))/W &\en\text{if} \en\kappa=0\\
       {\CM}_\kappa(\dt,W)&\en\text{if}\en \kappa\neq 0,
     \end{cases}
     \]
   where ${{\CM}_\kappa}(\dt,W)$ is the
   trigonometric {\em Calogero-Moser} variety with parameter $\kappa$.
   Using this and the main result of \cite{GK}, from 
   of Theorem \ref{c thm}(1) one deduces the following
 isomorphism: 
 \[\mu_{\dgl^*\times{\BA^1}}\inv(\dgl^*\times\{\kappa\})\cong
     \begin{cases}
   \overline{T^*(\dg/\du)}/W  &\en\text{if} \en \kappa=0\\
 \overset{\,}{\overline{T^*_\psi(\dg/\du)\times_\fcc^J\CM_\kappa(\dt,W)}}   &\en\text{if} \en \kappa\neq 0.
     \end{cases}
     \]
    Here,
     $\mu_{\dgl^*\times{\BA^1}}:
     (T^*\g)^!_{\dgl^*\times{\BA^1}}\to \dgl^*\times{\BA^1}$ is the moment map,
     $W$ acts on $\overline{T^*(\dg/\du)}$ via the Gelfand-Graev action
     constructed in \cite{GR}, \cite{GK}, and we have used
     that for any quasi-affine variety $X$ equipped with an action of a finite group
     $\Gamma$ one has $\overline{X/\!/\Gamma}=\Spec \C[X]^\Gamma= \overline{X}/\!/\Gamma$.
\end{examp}

\subsection{The universal Coulomb branch}
\label{univ}
    Let   $\bn=\sstt$ be  the defining representation of the group  $G=GL(\bn)$ in a vector space
    $\bn$ of dimension $n$.
    An associated Coulomb branch   $\bm=\sstt\oplus\sstt^*$
    is called the `universal Coulomb branch of 
    cotangent type',
    cf. \cite{BDFRT} for motivation. We have  $\dg\cong GL_n=GL_n({\mathbb C})$, resp.
   $\dgl=\gln=\gln({\mathbb C})$.
      The  $GL_n$-action on
   itself by left translations induces 
   a Hamiltonian action in $T^*GL_n$. We identify $\dgl^*$ with $\gln$
   via the trace pairing and let $T^*GL_n\cong GL_n \times\gln$
be the left invariant trivialization.

The representation $\sstt$ is gluable since the
    group $GL(\bn)$ contains the group
    $\gm$ of dilations, cf. \cite[Sect. 2.2]{GW}.
    From an explicit formula for the section $\eps_\sstt$ given in \cite{T}, \cite{GW},
   it is not difficult to find that  the   map
   $\gamma_\sstt:\greg\to\dg$ equals the restriction to the open set $\greg\sset \dgl^*=\gln$
   of the identity map $\xi\mto \xi$, the latter being
   viewed as a rational map $\gln\to GL_n$ with singularities
   at the divisor $\det\xi=0$.
   It follows that the Lagrangian subvariety
   $\La_\sstt\sset T^*GL_n$ is the diagonal
   $\La_\sstt=\{(x,\xi)\in GL_n\times \gln\mid x=\xi\}$,
   alternatively, $\La_\sstt$ equals the image of the section
   $d\Tr: GL_n\to T^*GL_n$ of the cotangent bundle, where $d\Tr$ is the differential
   of the function $GL_n\to {\mathbb C},\,g\mto \Tr(g)$.
           The quantization of $\La_\sstt$
    is  the Braverman-Kazhdan  $\dd_{_{GL_n}}$-module 
    $\Phi_{\sstt}=\dd_{_{GL_n}}(\Tr)$   generated
    by this function,
    cf. \cite{GG}.

    Next, let the group $\dg\times\dg=GL_n\times GL_n$ act on the
    vector space $\gln$ by left and right multiplication.
    The  hamiltonian $GL_n$-action on $T^*\gln$
   induced by the action
   by left, resp.
  right, multiplication
  will be referred  to   as the   `left', resp. `right', action.
    Let $U\subset GL_n$ be the maximal unipotent subgroup
    of upper-triangular matrices with $1$'s on the diagonal and
    $T^*\gln\tsl_\psi U$,
    resp. $U_\psi\back\!\back\!\back T^*\gln\tsl_\psi U$,
     the Whittaker reduction of $T^*\gln$ with respect
     to the right action of $U$, resp. left and right action
     of $U\times U$.
         We identify
    $GL_n$ with     the subset of $\gln$ of  invertible matrices,
    which is a $GL_n\times GL_n$-stable open subset. Thus,
    we obtain an     open imbedding $T^*GL_n\into T^*\gln$,
    and we will use similar notation for Whittaker reductions
    of $T^*GL_n$; in particular, we have $T^*GL_n\tsl_\psi U=\WW_{GL_n}$.

      \begin{prop}\label{gln claim}
        The imbedding $T^*GL_n\into T^*\gln$, resp.
        \[T^*GL_n\tsl_\psi U\ \into \ T^*\gln\tsl_\psi U\en\;\text{and}\en\;
                    U_\psi\back\!\back\!\back T^*GL_n\tsl_\psi U
                  \   \into\  U_\psi\back\!\back\!\back T^*\gln\tsl_\psi U,
                    \]
        extends to  an   isomorphism
        $\dc{(T^*GL_n)}\iso T^*\gln$, resp.
        \[\dc{(\WW_{GL_n})}\ \iso\  T^*\gln\tsl_\psi U\en\;\text{and}\en\;
          \dc{(U_\psi\back\!\back\!\back T^*GL_n\tsl_\psi U)}\ \iso\ 
          U_\psi\back\!\back\!\back T^*\gln\tsl_\psi U.
          \]
      \end{prop}

\begin{proof}[Sketch of proof]
  We use  the natural identification $T^*\gln\cong \gln\times\gln$
   In terms of this identification,
  the left, resp. right, action of an element  $g\in GL_n$  is given by the formula
  $g: (a,b)\mto (ga, bg\inv)$, resp. $g: (a,b)\mto (ag\inv,\,gb)$, and an associated
  moment map $\gln\times\gln\to \gln$ 
  is the map $\mu_{\text{left}}: (a,b)\mto ab$, resp. $\mu_{\text{right}}: (a,b)\mto ba$.
      left invariant trivialization sends $(a,b)\in $ to $(a, ab)$.
  The map $\gamma_\sstt$ being the identity map in  the
   left invariant trivialization

  From \eqref{tau map} and the formula for the map $\gamma_\sstt$, 
   we see that for $X=T^*\gln=\gln\times\gln$,
  viewed as a $GL_n$-variety with respect to the left action,
  the rational map
  $\boldsymbol{\tau}_{{X},\sstt}$ reads
  \[ \boldsymbol{\tau}_{_{{T^*\gln},\sstt}}:\, \gln\times\gln
    \dashrightarrow\gln\times\gln,\en
     (a,b)\mto ((ab)a,\,b(ab)\inv)= (aba, bb\inv a\inv)=(aba, a\inv).
  \]

  Now, the identification of $GL_n$ with an open subset of $\gln$
  induces the identification $T^*GL_n= GL_n\times \gln\subset \gln\times\gln$
  (which is different from the identification
  $T^*GL_n= GL_n\times \gln$ provided by the left invariant trivialization).
  For any $a\in GL_n$ the linear map $\gln\to\gln,\, b\mto aba$, is injective,
  hence it is an isomorphism. Therefore, the map
  $GL_n\times \gln\to GL_n\times \gln,\,(a,b)\mto (a, aba)$, is an automorphism.
  Hence, from the displayed  formula above we deduce that
  $\boldsymbol{\tau}_{_{{T^*\gln},\sstt}}(GL_n\times \gln)=\gln\times GL_n$.
  It follows that the complement in $\gln\times \gln$ of the union of open
  sets $GL_n\times \gln$ and $\boldsymbol{\tau}_{_{{T^*\gln},\sstt}}(GL_n\times \gln)$
  is the set $\{(a,b)\in \gln\times \gln\mid \det a=0\;\&\; \det b=0\}$,
  an intersection of two divisors. We conclude that any rational function
  $f\in \C(\gln\times \gln)$ which is regular on
  $GL_n\times \gln$ and on $\boldsymbol{\tau}_{_{{T^*\gln},\sstt}}(GL_n\times \gln)$
  must be a regular function on the whole of $\gln\times\gln$.
  This proves the isomorphism $\dc{(T^*GL_n)}\iso T^*\gln$.

  To prove the second isomorphism of the proposition, let
  ${\mathfrak s}\sset \greg  \sset \gln$ be the Kostant slice.
  The Whittaker reduction $T^*\gln\tsl_\psi U=(\gln\times\gln)\tsl_\psi U$
    may be identified with the variety
  $\{(a,b)\in \gln\times\gln\mid ba\in {\mathfrak s}\}$, where we have used
  that $\mu_r(a,b)=ba$.
  The desired isomorphism $\dc{(\WW_{GL_n)}}\iso$  $T^*GL_n\tsl_\psi U$
    follows from this description of $T^*\gln\tsl_\psi U$ by mimicing the proof of the
    first isomorphism. The proof of the third isomorphism is similar.
    \end{proof}

    \begin{cor} In the case of the `universal Coulomb branch'
   $\bm=\sstt\oplus\sstt^*$   there are natural isomorphisms
   \[(\sstt\oplus\sstt^*)^!_{G}\cong T^*\gln\tsl_\psi U,
     \en\text{resp.}\en
     {\mathcal C}_{\sstt\oplus\sstt^*, G}\cong 
     U_\psi\back\!\back\!\back T^*\gln\tsl_\psi U.
   \]
 \end{cor}
 \begin{proof}
     The first isomorphism is a consequence of the isomorphism
        $\dc{(\WW_{GL_n})}\cong  T^*\gln\tsl_\psi U$ of the proposition
        and   the  isomorphism of Theorem \ref{c thm}(1).
        The second  isomorphism follows from this since the Coulomb branch
        ${\mathcal C}_{\bm,G}$ is obtained from $\bm^!_G$ by the Kostant-Whittaker reduction,
        by ~\cite{BF}.
      \end{proof}
      
      \begin{rems}  \vi The isomorphisms of the corollary have been proved
        earlier in \cite[Sect. 5.2]{BDFRT} in a totally different way.

\vii The function $\Tr$ is well defined on  the whole of $\gln$.
Therefore, the $\dd$-module $\Phi_{\sstt}=\dd_{GL_n}(\Tr)$ generated by the
function $\Tr$ is a restriction to 
    $GL_n$ of the $\dd$-module $\dd_{_{\gln}}(\Tr)$ on $\gln$.
    The first isomorphism of the proposition may be viewed as a classical anague
    of this. 
    Observe further that  matrix multiplication makes $\gln$ a monoid
    and convolution with $\dd_{_{\gln}}(\Tr)$ is essentially
    the functor of Fourier transform of $\dd$-modules on $\gln$.
        This is closely related to the fact that the $\gamma$-sheaf counterprt
        of  $\Phi_{\sstt}$ appears in the Godement-Jacquet case of symplectic duality.
        However, the roles of $G$ and $\dg$  in the Godement-Jacquet setting
      of symplectic duality   and in the isomorphism of the proposition
        are swapt, cf.  Remark ~\ref{fin rem}(i).
          \end{rems}

\section{Pointwise pure sheaves}
\label{sec2}
We keep the setting of Sect. \ref{form sec} 
 and use the notation $\t=\Lie T$, resp. $\hh(\td)=\hdt(\td)$  and
  ${\fc}=\hh(\pt)$. Thus, $\fc\cong\C[\t]$, viewed as a graded algebra such that $\t^*$ is placed in degree $2$.
         Let $Q=\C(\t)$ be the fraction field of $\fc$ and 
              $M\mto QM:=Q\o_\fc M$
              the  localization functor
              $\fc\mmod\to Q\mmod$.

              \subsection{}\label{u v} In this subsection we assume, for simplicity, that $X$
              is connected.              
              The fixed point
variety $X^T$ is a disjoint union
$X^T=\sqcup_{\tau\in\ct}\, \cx_\tau$ of connected components $\cx_\tau$,
where $\ct$ is a finite labeling set.
  Let $\eps: X^T\into X$, resp. $\eps_\tau: \cx_\tau\into X$,
  denote the imbedding.
                 For  $\cg\in D_T(X)$ there are canonical morphisms of  $A$-modules

                 \beq{loc1}
                                    \xymatrix{
                                      \hh(\eps^!\cg)     \ar@{=}[r]&          \
                                      \bigoplus_{\tau\in\ct}\ \hh(\eps_\tau^!\cg) \ar[d]
                                      \ar[r]^<>(0.5){\, \eps_!\, } &\ \hh(\cg) \ \ar[r]^<>(0.5){\,  \eps^*\, }\ar[d]
                                      &\ 
                                      \bigoplus_{\tau\in\ct}\ \hh(\eps^*_\tau\cg)  \
                                      \ar@{=}[r]\ar[d] & \ \hh(\eps^*\cg) \  \\
          Q\hh(\eps^!\cg)     \ar@{=}[r]                            & \bigoplus_{\tau\in\ct}\ Q\hh(\eps_\tau^!\cg)\
\ar[r]^<>(0.5){\, Q\eps_!\, } &\ Q\hh(\cg) \ \ar[r]^<>(0.5){\,  Q\eps^*\, }
&\ \bigoplus_{\tau\in\ct}\ Q\hh(\eps^*_\tau\cg)\ar@{=}[r]&\ Q\hh(\eps^*\cg) 
     }
     \eeq

     By the Localization theorem, the maps in the second row
     are  isomorphisms.
     If $\cg$ is $!$-pure, resp. $*$-pure, then  the $\fc$-module
     $\hh(\cg)$, as well as
 $\hh(\eps^!\cg)$,
 resp. $\hh(\eps^*\cg)$, is free, see Corollary \ref{inj cor} below.
 Hence,  the first, resp. last,  two vertical maps
 are injective.  It follows that the map $\eps_!$, resp. $\eps^*$,
 in the first row is injective.

  It follows from condition S1 of Sect. \ref{form sec} that the algebra $A:=\hh(X)$
is  concentrated in even degrees and that it 
is   free as an $\fc$-module.
Therefore, the restriction map $\eps^*: A\to \hh(X^T)=\bigoplus_{\tau\in\ct}\ \hh(\cx_\tau)$
 is an injective $\fc$-algebra map which induces a  $Q$-algebra isomorphism:
  \beq{QA}
  Q\eps^*:\ QA\, \iso\,\oplus_{\tau\in\ct}\ Q\hh(\cx_\tau).
 \eeq
 We will identify $QA$ with the direct sum in the RHS and view $A$, resp. $\hh(\cx_\tau)$,
 as a subalgebra of $QA$, resp. $Q\hh(\cx_\tau)$.
 Writing $e_\tau$ for the unit of the algebra $\hh(\cx_\tau)$, inside $QA$
 we get a decomposition  $1_A=\sum_{\tau\in\ct}\,e_\tau$ as a sum of orthogonal
 idempotents. Therefore, for any $\cg\in D_T(X)$ we obtain isomorphisms
 of $QA$-modules
\beq{loc3}
e_\tau Q\hh(\eps_\tau^!\cg)=Q\hh(\eps_\tau^!\cg)\  \xrightarrow[^{\sim}]{\  Q(\eps_\tau)_!\  }\  
e_\tau Q\hh(\cg)\  
\xrightarrow[{}^{\sim}]{\  Q\eps_\tau^*\  }\  e_\tau  Q\hh(\eps_\tau^*\cg)
=  Q\hh(\eps_\tau^*\cg).
\eeq

 Below we fix  a  closed subset $X'\sset X$ of the form
  $X'=\sqcup_{\mu\in\La'}\, X_\mu$ for some 
 $\La'\sset\La$.
 Let $X_\la$   be an open stratum in $X'$ and  put $X'_{<\la}:=X'\sminus X_\la$.
 We have a diagram
 $v: X'_{<\la}\into X'\hookleftarrow X_\la: u$,
  where $v$, resp. $u$, is a closed, resp.
  open, imbedding.
    The following lemma is standard

\begin{lem}\label{inj claim}  
 \vi  If $\cg\in \mix(X')$ is $!$-pure, resp. $*$-pure,
  then  \eqref{byv}, resp. \eqref{byu}, below  is a short
 exact sequence of free ${\fc}$-modules 
 \begin{align}
   &0\to \hh(v^!\cg)\xrightarrow{v_!} \hh(\cg)\xrightarrow{u^*} \hh(u^*\cg)\to0; \label{byv}\\
     &0\to \hh_c(u^*\cg)\xrightarrow{u_!} \hh(\cg)\xrightarrow{v^*}\hh(v^*\cg)\to 0.\label{byu}
 \end{align}

 \vii If $\cg_X\in\mix(X)$  is pure and condition S3  of Sect. \ref{form sec}
     holds then \eqref{byv} is a short
 exact sequence for $\cg:=\iota^*\cg_X$, where $\iota: X'\into X$ is the closed imbedding.
     \end{lem}

    \begin{proof} To prove (i) we induct on the number of strata in $X'$.
      If $\cg$ is $!$-pure then $\hh(u_!u^!\cg)\cong \hh(u^*\cg)[-2\dim X_\la]$
      is pure and  free over ${\fc}$.
   Further,     $\hh(v^!\cg)$
is  pure  and free over ${\fc}$, by the induction hypothesis.
    Therefore, the connecting homomorphisms
  in the long exact sequence associated with the distinguished triangle
  $v_!v^!\cg\to\cg\to u_*u^*\cg$ vanish. This implies that \eqref{byv}
  is an exact sequence.
  The proof  of \eqref{byu} is similar.

 To prove (ii)    consider the following   imbeddings
$
    \{x_\la\} \ \xrightarrow{\veps_\la}\  X_\la \  
    \xrightarrow{u} \ X' \ 
    \xrightarrow{\iota}  \ X$,
  and the corresponding restriction   maps
  \[\hh(\cg_X)\xrightarrow{i^*} \hh(\cg) \xrightarrow{u^*}
    \hh(u^*\cg) \xrightarrow{\veps_\la^*} \hh(\veps_\la^* u^*\cg)=
    \hh((\iota\ccirc u\ccirc \veps_\la)^*\cg_X).\]
It is known, cf. \cite[Lemma 3.5]{Gi}, that assumption S3
insures that the composite map $\hh(\cg_X)\to \hh((\iota\ccirc u\ccirc \veps_\la)^*\cg_X)$
is surjective. 
Therefore, the map $\hh(\cg)\to \hh(u^*\cg)$ is  surjective since the map $\veps_\la^*$
is an isomorphism.
\end{proof}

\begin{cor}\label{inj cor}
  Let $i: X'=\sqcup_\mu\,X_\mu\into  X$ be a closed  imbedding,
  and $j: X\sminus X'\into X$.
  If $\cg\in \mix(X)$ is  $!$-pure, resp. $*$-pure, then
  $\hh(i^!\cg)$, resp. $\hh(i^*\cg)$, is  pure  and free over $\fc$,
  and there is a short
    exact sequence
 \begin{align*}
   &0\to \hh(i^!\cg)\to \hh(\cg)\to \hh(j^*\cg)\to 0,\en\text{\em resp.}\\
&0\to \hh_c(j^!\cg)\to\hh(\cg)\to \hh( i^*\cg)\to 0.
 \end{align*}

 A similar result holds in the case of a closed imbedding 
 $Y'=\sqcup_\mu\,Y_\mu\into Y$.
 \end{cor}

 \begin{proof} This follows from Lemma \ref{inj claim}  by induction on the number of strata in
   $X'$.
   \end{proof}

   We now assume that the open
   stratum $X_\la$ in the diagram 
$v: X'_{<\la} \into X' \hookleftarrow X_\la: u$
of Lemma \ref{inj claim} meets $Y$,
i.e  $\la\in\Si$. Then, the set $Y_\la=Y\cap X_\la$ 
is  open in $Y\cap X'$ and $x_\la\in Y_\la$.
Let $\cx_\bullet$ be the connected 
   component of $X^T$ that contains the point $x_\la$.

   \begin{lem}
     We have $\cx_\bullet\cap X'_{<\la}=\emptyset$.
   \end{lem}
\begin{proof}      
     We have a stratification $ \cx_\bullet\cap X'=
     \sqcup_{\mu\in \Omega}\, (X_\mu)^T$ for some subset $\Omega\sset\La$.
     Since  $x_\la\in Y^T$,
     condition S2 of Sect. \ref{form sec}
     implies that $\cx_\bullet\cap X_\la=\{x_\la\}$
     is  the one point stratum of $\cx_\bullet$.
It follows that $x_\la$ is contained in the closure
of  any other stratum $(X_\mu)^T,\, \mu\in\Omega$.
On the other hand,  the set $\cx_\bullet\cap X'_{<\la}$ is closed
in $\cx_\bullet\cap X'$, hence
the set $\cx_\bullet\cap X_\la=\{x_\la\}$ must be open in  $\cx_\bullet\cap X'$. 
Therefore, $x_\la$ cannot be  contained in the closure
of   $(X_\mu)^T$ for any $\mu\in \Omega\sminus\{\la\}$.
We conclude that $\Omega\sminus\{\la\}=\emptyset$, as desired.
\end{proof}

 \begin{cor}\label{domain}
Let $\cg_{\la}\in D_T(X_\la)$, resp. $\cg_{<\la}\in D_T(X'_{<\la})$,
  and let 
 $\hh_\la$ be either $\hh(u_*\cg_{\la})$
 or $\hh(u_!\cg_{\la})$, resp. $\hh_{<\la}$ be $\hh(v_*\cg_{<\la})$.
  Put $\ct_{<\la}=\{\tau\in \ct\mid
 \cx_\tau\cap X' _{<\la} \neq\emptyset\}$. Then, we have:

 \noindent
 \vi  The idempotent
 $e_{<\la}:=\sum_{\tau\in \ct_{<\la}} e_\tau$
 acts in  $Q\hh_{<\la}$, resp. $Q\hh_\la$,
 as the identity, resp. zero.
 \vskip 3pt
 
 \noindent\vii 
 If $\hh_{<\la}$, resp. $\hh_\la$, is $\fc$-torsion free then
 $\Hom_A(\hh_\la,\,\hh_{<\la})=0$, resp.
$\Hom_A(\hh_{<\la},\, \hh_\la)=0$.
\end{cor}

\begin{proof}   The first claim in (i) is clear. To prove the second
  claim   let $\tau\in \ct_{<\la}$.
  Then $\cx_\tau\cap X'\sset \cx_\tau\cap X'_{<\la}$, by the previous lemma.
  Therefore,   the imbedding $\eps'_{\tau}: \cx_\tau\cap X'\into X'$
   factors through the imbedding $v: X'_{<\la}\into X'$.
   Using that $v^*u_!=0$, resp.  $v^!u_*=0$, we deduce that $(\eps'_\tau)^*u_!\cg_{\la}=0$,
   resp. $(\eps'_\tau)^!u_*\cg_{\la}=0$. Hence, from
\eqref{loc3} we obtain
\[e_\tau Q\hh(u_!\cg_{\la})= e_\tau Q\hh((\eps'_\tau)^*u_!\cg_{\la})=0,\en
  \text{resp.}\en
  e_\tau Q\hh(u_*\cg_{\la})=e_\tau Q\hh((\eps'_\tau)^!u_*\cg_{\la})=0.
  \]
The second
  claim in (i)  follows.
  
  To prove (ii), let
  $\phi\in\Hom_A(\hh_\la, \hh_{<\la})$. It suffices to show that
 the induced $QA$-linear map 
$Q\phi: Q\hh_\la\to Q\hh_{<\la}$ equals zero.
  For any $m\in Q\hh_\la$  using (i), we find
  $Q\phi(m)=e_{<\la}  \,Q\phi(m)=Q\phi(e_{<\la} m)=0$,
  as desired.
A similar 
argument based on the equation $(1-e_{<\la})\hh_{<\la}= 0$
shows that the image of any $A$-module map
$\phi: \hh_{<\la}\to \hh_\la$ is a torsion $\fc$-module,
and  (ii) follows.
\end{proof}

\subsection{Proof of Theorem \ref{Vthm}}
\label{pf vthm} 
Let $X=\sqcup_s\, X_s$ be a decomposition of $X$ into connected components
  and let $1_s\in H^0_T(X_s)$ be the unit of the algebra $\hh(X_s)$.
  Then, we have a decomposition $1_A=\sum_s\, 1_s$ as a sum of orthogonal idempotents.
    The action of the idempotents $1_s$
  on the cohomology shows that
  for any sheaves $\cg$ and $\cg'$ supported on different connected components of
  $X$ we have $\Hom^\hdot_{H^0_T(X)}(\hh(\cg),\,\hh(\cg'))=0$.
  Also, it is clear that $\Ext^\hdot_{D_T(X)}(\cg,\cg')=0$.
  Thus, we may (and will)  assume that $X$ is connected.
    In the setting of Lemma \ref{inj claim}, let $Y_\la=Y\cap X_\la$
    be nonempty and  consider the following  imbeddings
\beq{sqsq}
  \xymatrix{
    \      Y\cap X'_{<\la}\   \ar[d]^<>(0.5){\fii_{<\la}}\ar[rr]^<>(0.5){\bar v} && \   Y\cap X' \
    \ar[d]^<>(0.5){\fii_{\leq\la}}&&
    \     Y_\la \   \ar[d]^<>(0.5){\fii_\la}\ar[ll]_<>(0.5){\bar u}\\
   \      X'_{<\la}\    \ar[rr]^<>(0.5){ v} && \   X'\    &&
    \     X_\la\   \ar[ll]_<>(0.5){u}
    }
    \eeq

In the diagram of the following lemma the vertical maps are induced
by the functor $\hh(\td)$ and the groups involving $\bar u^*\bar\cf$ are declared to be zero if the set $Y_\la$ is empty.  We
use shorthand notation
$\Ext=\Ext^\hdot_{D_T(X')}$, resp. $\Hom=\Hom^\hdot_A$.

\begin{lem}\label{ejf} 
  Let $\bar\cf\in\mix(Y\cap X')$  be $!$-pure, resp. 
  $\ce\in\mix(X')$ be $*$-pure and such that the restriction  map $\hh(\ce)\to \hh(u^*\ce)$
  is surjective if $Y_\la$ is nonempty. Then, in the above setting
  there is a commutative diagram 
  \[{\small
    \xymatrix{
      0\ar[r]& \Ext(v^*\ce, (\fii_{<\la})_*{\bar v}^!\bar\cf)\ar[r]^<>(0.5){\bar\al}
      \ar[d]^<>(0.5){h_1}& \Ext(\ce,(\fii_{\leq\la})_*\bar\cf)
 \ar[r]^<>(0.5){\bar\be}\ar[d]^<>(0.5){h_2}&   \Ext(u^*\ce,  (\fii_\la)_*\bar u^*\bar\cf)\ar[r]\ar[d]^<>(0.5){h_3}&0\quad{\,}\\
0\ar[r] & \Hom(\hh(v^*\ce),\, \hh((\fii_{<\la})_*{\bar v}^!\bar\cf))\ar[r]^<>(0.5){\al}&
 \Hom(\hh(\ce),\, \hh((\fii_{\leq\la})_*\bar\cf))\ar[r]^<>(0.5){\beta} &
 \Hom(\hh(u^*\ce), \, \hh((\fii_\la)_*\bar u^*\bar\cf))&
}
}\quad
\]
of morphisms of graded $A$-modules, such that the following holds:

\begin{enumerate}
\item The top row 
  is a short exact sequence  of pure and free  $\fc$-modules
  (possibly zero).

   \item The map $\al$ is injective;

  \item We have $\im\al=\Ker \be$.
    \end{enumerate}
\end{lem}

\begin{rem} Similar statements may be found 
   in \cite{BY}, \cite{Gi}, and \cite{CMNO}, Sect. 4.7.2,
   but the proof below, based on Corollary \ref{domain},
   is quite different from arguments  used in {\em loc cit}.
  In particular, 
  condition S2 from Sect. \ref{form sec}
  is essential for the proof.
\end{rem}

\begin{proof} 
    To simplify notation, we put $\cf=(\fii_{\leq\la})_*\bar\cf$.
  Thus,   $v^!\cf= (\fii_{<\la})_*{\bar v}^!\bar\cf$ by base change.  By adjunction,  we have
 the following two chains of isomorphisms
\begin{align}
  \Ext(v_*v^*\ce,  \cf)\cong\Ext(v^*\ce,  v^!\cf)&\cong
                                                   \Ext(v^*\ce,  (\fii_{<\la})_*{\bar v}^!\bar\cf)\label{ext2}\\
  &\cong \hh(\BD(\fii_{<\la}^*v^*\ce)\so \bar\cf)
    \cong \hh((\fii_{<\la}^!v^!\BD\ce)\so \bar\cf),\en \ \text{resp.}                                                   
 \nonumber\\
    \Ext(u^*\ce, (\fii_\la)_*\bar u^*\bar\cf)\cong \Ext(\fii_\la^*u^*\ce,\bar u^*\bar\cf)
      &\cong      \hh((\fii_\la^!u^*\BD\ce)\so \bar u^*\bar\cf)\nonumber\\
    &\cong
      \hh((\bar u^*\fii_{\leq\la}^!\BD\ce)\so \bar u^*\bar\cf)
  \cong
  \hh(\bar u^*((\fii_{\leq\la}^!\BD\ce)\so \bar\cf)).\nonumber
\end{align}
Since 
the sheaves $(\fii_{<\la}^!v^!\BD\ce)\so\bar\cf$ and
$(\fii_{\leq\la}^!\BD\ce)\so\bar\cf$ are $!$-pure, we conclude that
the cohomology group on the right of each chain of isomorphisms,
hence all the $\Ext$-groups above,  are pure and free over $\fc$.

To construct the maps $\bar\al$ and $\bar\be$
we apply
  the functor $\Ext(\td,\cf)$ to the distinguished triangle
  $u_!u^*\ce\to \ce\to v_*v^*\ce$. This gives a long exact sequence
  \beq{uuvv}
  \ldots\xrightarrow{\ \partial \ }
  \Ext^\hdot(v_*v^*\ce,\cf)\xrightarrow{\  \tilde\al \  }
      \Ext^\hdot(\ce,\cf) \xrightarrow{\  \tilde\be \  } \Ext^\hdot(u_!u^*\ce,\cf)
    \xrightarrow{\ \partial \ }\Ext^{\hdot+1}(v_*v^*\ce,\cf)\xrightarrow{\  \tilde\al \  }\ldots
  \eeq
  The first, resp. third,  $\Ext$-group in \eqref{uuvv} is one of the
  $\Ext$-groups that appear in \eqref{ext2}.
  These groups being  pure and free over $\fc$,
  it follows that so is the group $\Ext(\ce,\cf)$; furthermore,
    the connecting
          homomorphisms $\partial$ in  \eqref{uuvv}  vanish.
          The resulting short exact sequence gives the top row
                   of the diagram of the lemma, and  (1) holds by construction.

  We define the map $\al$  in the second row of the diagram  of the lemma
  to be the map that sends
  $f\in \Hom(\hh(v^*\ce), \hh((\fii_{<\la})_*{\bar v}^!\bar\cf))$ 
 to the composition
 \[
   \al(f):\    \hh(\ce)\xrightarrow{v^*}
   \hh(v^*\ce)\xrightarrow{f}
   \hh((\fii_{<\la})_*{\bar v}^!\bar\cf))\iso\hh(\bar v^!\bar\cf)
   \xrightarrow{\bar v_*}\hh(\bar\cf)\iso \hh(\cf).
\]
Since $\ce$ is $*$-pure, resp. $\bar\cf$ is $!$-pure,
the map $v^*$, resp. $\bar v_*$, in the composition  is surjective, resp. injective, by Lemma
\ref{inj claim}. Hence, $\al$ is injective and (2) holds.

Next, we construct the map $\be$. If the set $Y_\la$ is empty
then $\be$ is the zero map.
If  $Y_\la$ is nonempty then the map
$u^*: \hh(\ce)\to \hh(u^*\ce)$ is surjective by the assumptions of the lemma, hence
this map induces an  isomorphism
$\hh(\ce)/\hh(v^!\ce)\iso\hh(u^*\ce)$, by Lemma \ref{inj claim}(ii).
Let
$g\in\Hom_A(\hh(\ce),\,\hh(\cf))$.
The composition
  \[\hh(v^!\ce)\xrightarrow{v_*}
  \hh(\ce)\xrightarrow{g} \hh(\cf)
  \xrightarrow{u^*}
 \hh(u^*\cf)\cong \hh(\bar u^*\bar\cf)
 \]
is  the zero map by Corollary \ref{domain}.
It follows that
the map $u^*\circ g$ has a  factorization 
\[
  \hh(\ce)\to \hh(\ce)/\hh(v^!\ce)\iso \hh(u^*\ce)\xrightarrow{\be(g)}
  \hh((\fii_\la)_*\bar u^*\bar\cf)\iso\hh(u^*\cf)
\]
for a unique map $\be(g)$.
The assignment $g\mto\be(g)$ provides the desired
map $\be$.
It is easy to check that  the maps  thus defined make the diagram  of the lemma
commute.

It remains to prove (3). If $Y_\la=\emptyset$
we have $\bar u^*=0$ and the map $\bar v_*$ is clearly an isomorphism, so
the bottom line is exact, proving (3). We now assume $Y_\la$ is nonempty.
Then, the equation $\be\ccirc\al=0$ holds by construction; explicitly,
we have $\be(\al(f))=\bar u^*\ccirc \bar v_*\ccirc \fii_{\leq\la}^*\ccirc f\ccirc v^*\ccirc u_!=0$,
since $\bar u^*\ccirc \bar v_*=0$. 
To prove an opposite inclusion $\Ker(\be)\sset\im(\al)$,
let $g: \hh(\ce)\to \hh(\bar\cf)$
be such that
$\be((\fii_{\leq\la})_*\ccirc g)=0$, that is, such that
$\bar u^*\ccirc g\ccirc u_!=0$.
By
\eqref{byv}, this means that $\im(\bar u^*\ccirc g)\sset \im(\bar v_*)$, so
we have 
  $g\ccirc u_!=\bar v_*\ccirc g_1$ 
  for some $A$-module map $g_1$ in the diagram:
  \[
  \xymatrix{
0\ar[r]&    \hh_c(u^*\ce) \ \ar[r]^<>(0.5){u_!}  \ar@{.>}[dr]_<>(0.5){g_1}&
    \ \hh(\ce)  \ \ar[drrr]^<>(0.3){g}\ar[rrr]^<>(0.5){v^*} &&& \
    \hh(v^*\ce)\  \ar@{.>}[dlll]^<>(0.1){g_3} \ar@{.>}[d]^<>(0.5){g_2}\ar[r]&0 & \\
&0\ar[r]  &\hh(\bar v^!\bar\cf) \  \ar[rrr]^<>(0.5){\bar v_*} &&&
\  \hh(\bar\cf) \ \ar[r]^<>(0.5){\bar u^*} & \  \hh(\bar u^*\bar\cf) \ar[r]&0 
  }
  \]
    The map $g_1$ must be zero by Corollary \ref{domain},
  since $\hh(\bar v^!\bar\cf)$ is a free $\fc$-module.
  We conclude that $g\ccirc u_!=\bar v_*\ccirc g_1=0$, 
so the map $g$ sends the image of $u_!$
to zero. By \eqref{byu},  it follows  that $g$ factors as
$g=g_2\ccirc v^*$.
Applying Corollary \ref{domain} again, we deduce that $\bar u^*\ccirc g_2=0$.
Hence, $g_2$  factors as $g_2=\bar v_*\ccirc g_3$.
We conclude that
$g=\bar v_*\ccirc g_2\ccirc v^*$. Thus, we obtain
$(\fii_{\leq\la})_*\ccirc g=v_*\ccirc(\fii_{<\la})_*\ccirc g_2\ccirc v^*=
\al((\fii_{<\la})_*\ccirc g_2)\in \im(\al)$,
as desired.
\end{proof}

\begin{proof}[Proof of Theorem \ref{Vthm}]
 We induct  on the number of strata in $X$.
If this number is 1 then $X=\pt$ and the statement is clear.
If $X$ has more than one stratum,  choose
a subset $X'=X_{<\la}\sqcup X_\la\sset
X$ as above, let $i: X'\into X$, resp. $\bar i: Y\cap X'\into Y$, denote
the closed imbedding and put  $\ce:=i^*\ce_X$, resp. $\bar\cf:=\bar i^!\cf_Y$.
Then, we are in the setting of  Lemma \ref{ejf},   since
the restriction map $u^*: \hh(\ce)\to \hh(u^*\ce)$
is surjective by Lemma \ref{inj claim}(ii).
By the induction hypothesis,  we may assume
that the map $h_1$ in the diagram of  Lemma \ref{ejf} is an isomorphism.
If   the set $Y\cap X_\la$ is empty then the imbedding $\bar v$ in the diagram
of the lemma is a bijection and the term in each row that involves $\bar u^*\cf$
vanishes. It follows that 
the 
maps $\al$ and $\bar\al$ are isomorphisms, so the map $h_2$ 
is an isomorphism.
In the case where $Y\cap X_\la$ is nonempty  the map $h_3$
in the diagram is  an isomorphism since the sheaves
$u^*\ce$ and $(\fii_\la)_*\bar u^*\cf$ are  direct sums of
shifts of constant sheaves.
Using Lemma \ref{ejf},  we deduce 
by diagram chase that the map $h_2$ is an isomorphism,
completing the induction step.
\end{proof}

\subsection{$!$-tensor product}\label{tp}
Let $A$ be  an arbitrary commutative $\fc$-algebra. For an $A$-module
$M$ we let $M^\vee:=\Hom_\fc(M,\fc)$ and equip this $\fc$-module
with an $A$-module structure by $a(f): m\mto f(am)$, where $f\in M^\vee,\,a\in A,\,m\in M$.
For $A$-modules $M,N$, the tensor product
$M\o_\fc N$ has the natural structure of an $A\o_\fc A$-module, and
we define an $A$-module as follows
\beq{olem1}
M\so_A N\cong
\{u\in M\o_\fc N\mid (a\o1)u=(1\o a)u\;\ \forall a\in A\}.
\eeq

In the rest of this section we assume
that $A=\oplus_{i\geq 0} \,A_i$ is a commutative nonnegatively graded
$\fc$-algebra  which is finitely generated as an $\fc$-module
and such that $A_0=\C$.
Let $A\grfmod$ be the category of  $\Z$-graded
   $A$-modules which are finitely generated and free
   as $\fc$-modules.
   A finitely generated graded $A$-module is an object of $A\grfmod$ if and only if it is             
   a maximal Cohen-Macaulay  $A$-module.
   In particular, $A$ itself  is a Cohen-Macaulay algebra.
     On $A\grfmod$ one has an exact contravariant duality 
     $M\mto M^\vee:=\Hom_\fc(M,\fc)$ such that the canonical map
     $M\to (M^\vee)^\vee$ is an isomorphism.
     The $A$-module $A^\vee$ is the  relative dualizing module.
     It follows that for any $M\in A\grfmod$
there is a canonical isomorphism $M^\vee\cong \Hom_A(M, A^\vee)$.

              \begin{lem}\label{olem} For $A$ as above, the operation $\td\so_A\td$ gives $A\grfmod$
              a symmetric monoidal structure
and for $M,N\in A\grfmod$,
                  there are canonical isomorphisms of $A$-modules:
              \begin{align}
  & M\so_A N\cong (M^\vee\o_A N^\vee)^\vee\label{olem2}\\
   &\Hom_A(M,N) \cong        M^\vee\,\so_A\, N.\label{MMNN}
              \end{align}
            \end{lem}

\begin{proof}
  For any $M,N\in A\grfmod$ one has a canonical isomorphism  
 $M^\vee\o_\fc N\iso \Hom_\fc(M,N)$  of  $A \o_\fc A$-modules where
 the action of $a\o b\in A\o_\fc A$ on $f\in \Hom_\fc(M,N)$ is defined by 
 $(a\o b)f: m\mto a f(bm)$.
 From this isomorphism
    we obtain canonical isomorphisms
    \[
         \Hom_A(M,N) \cong \{f\in \Hom_\fc(M,N)\mid
(a\o1)f=(1\o a)f\;\ \forall a\in A\}
   \cong                    M^\vee\,\so_A\, N.
    \]

  This  proves  \eqref{MMNN}. Replacing $N$ by $N^\vee$ in the above, we obtain
        \begin{align*}
          M^\vee\,\so_A\, N^\vee &\cong  \Hom_A(M,N^\vee)\cong
          \Hom_A(M,\,\Hom_A(N,A^\vee))\\
          &\cong \Hom_A(M \o_A N, A^\vee) \cong (M \o_A N)^\vee.
            \end{align*}
        Replacing $M$ by $M^\vee$, resp. $N$ by $N^\vee$, yields   \eqref{olem2}. 
        \end{proof}

        \begin{rem}\label{oo}
\vi   There is a canonical composite 
        morphism $M\so_A N\to M\o_\fc N\to M\o_AN$;  the composite is
        an isomorphism provided one of the modules is free over $A$.

        \vii Let $I\sset A$ be the kernel of an $\fc$-algebra homomorphism
        $A\to\fc$.
        Then, there is an $\fc$-stable direct sum decomposition
        $A\cong \fc\!\cdot\! 1\oplus I$, and for any $A$-module $M$
        one has
        \[\Hom_A(A/I, M) \cong  M\so_A (A/I) \cong M^I:=\{m\in M\mid am=0\;\ \forall a\in I\}.
        \]

        \viii Let $A$ be a flat commutative $\fc$-algebra  of finite relative injective dimension
        and $A^\vee=\RHom_\fc(A,\fc)\in D^b(A\mmod)$ the relative dualizing complex.
        Then, there is a duality  $M\mto M^\vee:=\RHom_A(M, A^\vee)\cong\RHom_\fc(M,\fc)$
        on $D^b(A\mmod)$ and
        one may use the RHS of \eqref{olem1}  to define a symmetric monoidal
        structure on $D^b(A\mmod)$.
\end{rem}

Let $X=\sqcup_\la\, X_\la$ be as in Sect. \ref{form sec}  and let $\BD(\td)$
     denote the Verdier duality functor on $\mix(X)$.
     Recall that the constant sheaf $\k_X$, resp.  the
     dualizing complex $\om_X=\BD\k_X$, is $*$-pure, resp. $!$-pure,
     and
     $H^{\BM}_T(X):=\hh(\om_X)$ is the equivariant Borel-Moore homology.

     Part (i) of the following result  has been observed earlier by Brion, \cite{Br}.

     \begin{prop}\label{bri} \vi  The $\fc$-algebra $\hh(X)$ is Cohen-Macaulay
           and there is a canonical isomorphism
 $H^{\BM}_T(X)\cong  \hh(X)^\vee$.

 \vii         If   $\cg\in \mix(X)$  is either $!$-pure or $*$-pure  then    
    $\hh(\cg)\in \hh(X)\grfmod$    and
 Verdier duality yields a canonical isomorphism
 $\hh(\BD\cg)\cong (\hh(\cg))^\vee$ of graded $\hh(X)$-modules.\qed
\end{prop}
\begin{proof} Part (i) follows from (ii).
  To prove (ii)  let $p_n: X_n\to B_n,\, n\geq1$, be a sequence of finite
  dimensional approximations of the Borel construction
  $X\times_T ET\to BT$, e.g. for $T=\gm$ one can take $B_n={\mathbb P}^n$.   Thus,
    $BT$ is a direct limit of
  a chain $\ldots\xrightarrow{i_{n}}B_n\xrightarrow{i_{n+1}} B_{n+1}\xrightarrow{i_{n+2}}\ldots$,
  of closed imbeddings and the maps $p_n$ are smooth projective morphisms with fiber $X$,
  so  $(p_n)_*=(p_n)_!$.
  Let
  $\cg_n$, resp. $(\BD\cg)_n$, be either $!$-pure or $*$-pure sheaf on $X_n$ that corresponds to
  $\cg$, resp. $\BD\cg$. The collection of sheaves $\cg_n,\,n\geq1$, comes
  equipped with     isomorphisms
  $(p_n)_*\cg_n\cong i_{n+1}^*(p_{n+1})_*\cg_{n+1}$,
  and
    the  inverse system $\ldots \leftarrow H^k((p_n)_*\cg_n) \leftarrow
  H^k((p_{n+1})_*\cg_{n+1}) \leftarrow\ldots$, stabilizes for each  integer $k$.
  The  stable limit is,  by definition,
  the equivariant cohomology group $\hh^k(\cg)$.
We know by Corollary \ref{inj cor} applied in the case of a trivial torus
that the nonequivariant cohomology of $\cg$ and  $\BD\cg$
  are both pure.
  Since $B_n$ is simply connected, it follows from the Leray spectral sequence that
  for every $n$ the
  objects $(p_n)_*\cg_n,\,(p_n)_*(\BD\cg)_n\in D^{mix}(B_n)$ are  pure, hence quasi-isomorphic
   to a direct sum of shifts of the constant sheaf $\k_{B_n}$.
    We have $(p_n)_*(\BD\cg)_n\cong \BD((p_n)_*\cg_n)$
  since $p_n$ is  proper.  Thus,  Verdier duality
  yields   
  \[H^\hdot((p_n)_*(\BD\cg_n))\,\cong\,
  H^\hdot(\BD((p_n)_*\cg_n))\,\cong\, \Hom^\hdot_{H^\hdot(B_n)}(H^\hdot((p_n)_*\cg_n),\,
  H^\hdot(B_n)).
\]
The desired isomorphism $\hh(\BD\cg)\cong (\hh(\cg))^\vee$ follows from this 
by taking an inverse limit.
\end{proof}

\begin{cor}\label{o cor}
For  any   $\cf\in D_{T}(Y)$ and $\ce_1,\ldots, \ce_n\in D_T(X)$, 
     there is a natural morphism
\[
   \hh(\cf\,\so\, \fii^!(\ce_1\so\ldots\so\ce_n))
  \to\hh(\cf)\so_{\hh(X)}\hh(\ce_1)\so_{\hh(X)}\ldots\so_{\hh(X)}\hh(\ce_n).
 \]

 If $\cf$ is    $!$-pure  and all $\ce_1,\ldots, \ce_n$ are pure then the above morphism
 is an  isomorphism.
\end{cor}

\begin{proof} Let  $\cg=\cf\boxtimes\ce_1\boxtimes\ldots\boxtimes\ce_n$
  and $i: Y\into Y\times X^n,\,
  y\mto (y,\fii(y), \ldots,\fii(y))$, a closed imbeddng.
  Thus, we have $\hh(\cf\,\so\, \fii^!(\ce_1\so\ldots\so\ce_n))=\hh(i^!\cg)$.
  The image of the adjunction
    $\hh(i_!i^!\cg) \to \hh(\cg)$
is annihilated by the action of the kernel $\Ker\subset \hh(Y\times X^n)$ of the restriction 
map $i^*: \hh(Y\times X^n)\to \hh(Y)$. Therefore, this image is contained
in a subspace of $\hh(Y)$ formed by the elements annihilated by $\Ker$.
This subspace is 
$\hh(\cf)\so_{\hh(X)}\hh(\ce_1)\so_{\hh(X)}\ldots\so_{\hh(X)}\hh(\ce_n)$,
by the definition of $\so$.

Assume next that $\cf$ is $!$-pure, resp. $\ce=\ce_1$ is pure, and $n=1$.
Then, the morphism in the statement may be factored as 
  the following composition:
     \begin{multline*}
       \hh^\hdot((\fii^!\ce)\so \cf)\iso
       \hh^\hdot(\BD(\BD\ce)\,\so\, \fii_*\cf)
       \iso
         \Ext^\hdot_{D(X)}(\BD\ce,\, \fii_*\cf)\\
       \iso \Hom^\hdot_{\hh(X)}(\hh(\BD\ce),\,\hh(\cf))
      \iso\hh(\BD\ce)^\vee\so\hh(\cf)\iso
       \hh(\ce)\so_{\hh(X)}  \hh(\cf).
       \end{multline*}
Here, the third, resp. fourth and fifth,  isomorphism holds 
       by Theorem \ref{Vthm}, resp.       part (ii)  of
       Proposition \ref{bri}, and \eqref{MMNN}.
       The general case of an arbitrary $n\geq 1$ may be reduced
    to   the case $n=2$ by induction on $n$ using the associativity of the
  monoidal structure $\so_{\hh(X)}$ on the category $\hh(X)\grfmod$.
       \end{proof}

\begin{cor}\label{o cor1}
  In the setting of Sect. \ref{form sec}, let $\Theta\sset \Si$ be a subset such that 
  $Y_\Theta=\sqcup_{\nu\in\Theta}\, Y_\nu$ is a  locally closed
  subvariety of $X$ and  $\fii_\Theta: Y_\Theta\into X$,
  resp.   $\jm_\mu: X_\mu\into X$, the imbedding.
    Then,  for a pure $\ce\in \mix(X)$ the adjunction $(\fii_\Theta)_!\fii_\Theta^!\to\Id$,
   resp. $(\jm_\mu)_!\jm_\mu^!\to\Id$, 
induces an isomorphism
\begin{align}
  \hh_c(Y_\Theta,\fii_\Theta^!\ce)\ &\iso\  \Hom^\hdot_{\hh(X)}(\hh_c(Y_\Theta),\,
  \hh(\ce)),\en\text{\em resp.}\label{om f}\\
    \hh(\jm_\mu^!\ce) \ &\iso \Hom^\hdot_{\hh(X)}(\hh(X_\mu),\,
    \hh(\ce)),\quad\forall \mu\in\Si.\label{fiii i}
\end{align}
\end{cor}

\begin{proof}
  We apply Corollary   \ref{o cor} in the case  $n=1$.
  Taking  $\cf=\fii_*\om_Y\in \mix(Y)$  yields \eqref{om f}.
     Similarly, to prove \eqref{fiii i} one takes  $Y=Y_\Theta=X_\mu$ and 
     $\cf=(\jm_\mu)_*\k_{X_\mu}$.
  The desired isomorphism then follows  using that
    $\hh(\fii_!\fii^!\ce)\cong\hh_c(\fii^!\ce)\cong \hh(\fii^!\ce)[-2\dim X_\mu]$, resp.
  $\hh_c(X_\mu)\cong \hh(X_\mu)[-2\dim X_\mu]$.
\end{proof}

  \section{Reminder on derived Satake}
\label{rem sec}

\subsection{Generalities}\label{gen}
Fix a maximal torus $T\sset G$ and an Iwahori subgroup
 $I\sset G(O)$  that contains $T$.

The affine Grassmannian $\Gr=\Gr_G$ can be presented as a direct limit
$\Gr=\underset{\stackrel{\too}{{n}}}\lim\, X_{n}$ of $G(O)$-stable projective subvarieties
$X_{n}\sset\Gr$.
For each ${n}$, the $I$-orbits $X_\la=I.\la,\,\la\in\BX_*(T)$, contained in $X_{n}$ form a $T$-stable stratification
$X_{n}=\sqcup_{\la\in \La_{n}} X_\la$,
where $\La_{n}\sset \BX_*(T)$ is a finite subset. 
Let $i_{n}: X_{n} \into X_{{n}+1}$ denote the closed imbedding.
Thus, one has the category $D_I(\Gr)=\underset{\too}\lim\,D_I(X_{n})$
and $\underset{\longleftarrow}\lim\,\Ind D_I(X_n)$, an inverse limit of the categories
$D_I(X_{n})$ along $!$-restrictions. An object
of $\Ind D_I(\Gr)$
is an inverse limit $\cf=\underset{\leftarrow}\lim\, \cf_{n}$
of objects
$\cf_{n}\in D_I(X_{n})$ equipped with isomorphisms
$\kappa_{n}: i_{n}^!\cf_{{n}+1}\iso \cf_{n}$ which satisfy natural compatibilities,
cf. eg. \cite[Sect. 2(i)]{BFN2}.
In particular, one has the dualizing sheaf $\om_\Gr=
\underset{\leftarrow}\lim\, \om_{X_{n}}\in \Ind D_I(\Gr)$.
There is also the $G(O)$-equivariant counterpart
$\Ind\sat$ defined in a similar fashion.

Let $\t=\Lie T$ and  $\hh(\td)=H_T^\hdot(\td)\cong H^\hdot_I(\td)$, so
$\hh(\pt)\cong\C[\t]$, where $\t^*$ is placed in degree $2$.
The constant sheaf $\k_\Gr$ is not an object of $\Ind D_I(\Gr)$
  since it is not an inverse limit along $!$-restrictions; instead,
  one has natural isomorphisms $i^*_{n}\k_{X_{{n}+1}}\iso \k_{X_{{n}}}$.
  The compatible cell decompositions of the $X_{n}$'s  insure
that the restriction maps $i^*_{n}: \hh(X_{{n}+1})\to \hh(X_{n})$ are surjective;
furtheremore, for each $m$, the inverse system of cohomology
$\ldots \leftarrow \hh^m(X_n)\leftarrow \hh^m(X_{n+1})\leftarrow \ldots$ stabilizes. Let
$\hh^m(\Gr)$ be this stable limit and
 $\hh(\Gr):=\oplus_m\,\hh^m(\Gr)$.  Then, for any ${n}$ the restriction
map $\hh(\Gr)\to \hh(X_{n})$ is a surjective map of graded $\C[\t]$-algebras, so
we may (and will) view $\hh(X_{n})$-modules
as $\hh(\Gr)$-modules.
For $\cf=\underset{\leftarrow}\lim\, \cf_{n}\in \Ind D_I(\Gr)$,
one has an
$\hh(X_{n})$-module structure on $\hh(\cf_{n})$.
The maps $\hh^m(\cf_k) \xleftarrow[\cong]{\kappa_k} \hh^m(i_k^!\cf_{k+1})
\xrightarrow{(i_k)_!}  \hh^m(\cf_{k+1})$ are $\hh(X_{n})$-module maps for all $k<{n}$.
This way,  the  groups $\hh^m(\cf_{n}),\, {n}\geq1$,   form a direct system 
and we define $\hh^m(\cf):=\underset{\rightarrow}\lim\, \hh^m(\cf_{n})$.
The above shows that the total cohomology $\hh(\cf)=\oplus_m\,\hh^m(\cf)$
acquires the structure of a graded $\hh(\Gr)$-module.

In the case of a finite ground field $\BF_q$ and a split reductive group
$G$ over $\BF_q$, one has the affine Grassmannian over $\BF_q$,
which is a direct limit of projective subvarieties $X_n$, as above, defined over $\BF_q$.
Let  $\Gr$ be a
base change to $\BF=\overline{\BF_q}$ and
$D_I(\Gr)$, resp. $\sat$, the corresponding category of $\bar{\mathbb Q}_\ell$-adic sheaves.
This category is equipped with a pullback  functor $Fr^*$
along the geometric Frobenius.
An object $\cf\in
\Ind D_I(\Gr)$ is called $!$-pure if  $\cf=\underset{\leftarrow}\lim\, \cf_{n}$
where each $\cf_n$ has a lift to a $!$-pure object
of $D^{mix}_I(X_{n})$.
On defines pure objects in a similar fashion.

\subsection{}\label{Jjsec}
We choose  a maximal compact subgroup $G_c\subset G$.
We write $\sfh(\td):=H^\hdot_{G(O)}(\td)$ and let  $W=N(T)/T$ denote the Weyl group.
For  $\cf\in D_{G(O)}(\Gr)$,  there are canonical
 isomorphisms
 \beq{GTH}
 \sfh(\cf)\cong H^\hdot_{G_c}(\text{Obl}^{G(O)}_{G_c}\cf) ,\quad
  H^\hdot_T(\text{Obl}^{G(O)}_T\cf)\cong \C[\t]\o_\sfc \sfh(\cf),\quad
                      \sfh(\cf)\cong \big(H^\hdot_T(\text{Obl}^{G(O)}_T\cf)\big)^W.
            \eeq
            Let  $\sfc:=\C[\g]^G\cong\C[\t]^W$, resp. $\fcc=\fcc_\dg=\Spec\sfc\cong\t/W
            \cong\dgl^*/\!/\dg$.
           We have $\sfc\cong \sfh(\pt)\cong H^\hdot_{G_c}(\pt)$. This allows
  to transfer various results for
  $T$-equivariant cohomology to results for $G(O)$-equivariant cohomology.

  Let $\Omega_{pol} G_c\subset G(K)$ be the subgroup of based polynomial maps
  $S^1\to G_c$. The algebra $H^\hdot_{G_c}(\Omega G_c)$
 has the structure of a commutative and cocommutative
 graded Hopf algebra over $\sfc$ with coproduct
 $\De: H^\hdot_{G_c}(\Omega G_c)\to
H^\hdot_{G_c}(\Omega G_c) \,\o_\sfc \,H^\hdot_{G_c}(\Omega G_c)$
 induced
 by pointwise multiplication  $\Omega G_c\times \Omega G_c\to \Omega G_c$.
 Recall  that the cotangent sheaf 
     $T^*_{\fcc}$  is a free $\oo_{\fcc}$-module with  basis $df,\, f\in\C[\fcc]$,
     and there is an  isomorphism $\Lie J\cong T^*_{\fcc}$.
     Let  $\fj=\fj_\dg:=\Ga(\fcc,\Lie J)\cong \Ga(\fcc, T^*_\fcc)$.
     Thus, $\fj$ is a free $\sfc$-module of rank $\dim\t$.
     It is known that there is a Hopf algebra isomorphism, \cite{BF}, ~\cite{BFM}, and \cite{YZ}.
\beq{symJ}
H^\hdot_{G_c}(\Omega G_c) \cong \sym_\sfc\fj
\eeq
    such that 
the space of primitive elements
of the Hopf algebra $H^\hdot_{G_c}(\Omega G_c)$  corresponds to $\fj$.
Further, it is known that the composite  $\Omega_{pol} G_c\into G(K) \onto \Gr$
is a  homeomorphism,
cf. \cite[Sect. 4]{Nad}.
Hence, we obtain
$\sfh(\Gr)\cong H^\hdot_{G_c}(\Gr) \cong H^\hdot_{G_c}(\Omega G_c)$, so
  the Hopf algebra structure  on $H^\hdot_{G_c}(\Omega G_c)\cong \sym_\sfc\fj$
may be transported to $\sfh(\Gr)$ via the isomorphisms.
Dually, there is  a commutative and cocomutative Hopf
algebra structure on $H^\BM_{G_c}(\Omega G_c)$ which can be transported to 
$H^\BM_{G(O)}(\Gr)$.
Thus, $\Spec H^\BM_{G(O)}(\Gr)$ acquires the structure
of a commutative group scheme over $\fcc$.

Let $\star$ be the convolution monoidal structure on $\sat$.
  For any $\cf_1,\cf_2\in\sat$ there are
  isomorphisms
  \beq{TG}
  R\Gamma_{G(O)}(\Gr,\cf_1)\,\lo_\sfc\, R\Gamma_{G(O)}(\Gr,\cf_2)\iso
 R\Gamma_{G(O)}(\Gr, \,\cf_1\boxtimes\cf_2)\iso 
 R\Gamma_{G(O)}(\Gr, \cf_1\star\cf_2),
 \eeq
 in $D(\sfh\mmod)$,
 where the first isomorphism is the derived Kunneth formula.
  The composite of these maps, to be also denoted $\star$,
 equips the functor
$R\Gamma_{G(O)}(\Gr, \td): \sat \to D(\sfc\mmod)$ with a
monoidal structure. If at least one of the sheaves $\cf_i,\,i=1,2$,  is $G(O)$-equivariantly formal
then \eqref{TG}
induces an isomorphism
\beq{Hform}
\star:\, \sfh(\cf_1)\o_\sfc\sfh(\cf_2)\iso\sfh(\cf_1\star\cf_2).
\eeq

The  $\sfh(\Gr)$-action  on cohomology 
              is compatible with the coproduct on $\sfh(\Gr)$ in the sense 
              that for all $h\in\sfh(\Gr)$               the following diagram
 commutes
\beq{der}
   \xymatrix{
     \sfh(\cf_1)\o_\sfc \sfh(\cf_2) \ar[rr]^<>(0.5){\star}_<>(0.5){\cong}
     \ar[d]^<>(0.5){\De(h)}
              && \sfh(\cf_1\star\cf_2)\ar[d]^<>(0.5){h}\\
       \sfh(\cf_1)\o_\sfc \sfh(\cf_2)\ar[rr]^<>(0.5){\star}_<>(0.5){\cong}&& \sfh(\cf_1\star\cf_2)
     }
     \eeq

     The natural action of $\sfh(\Gr)$ on $\sfh(\cf)$
     gives,
     thanks to the isomorphism $\sfh(\Gr)\cong\sym_\sfc\fj$,
     a  $\fj$-module structure on $\sfh(\cf)$.
 It was  shown by Yun and Zhu, \cite[Lemma 3.1]{YZ},  that  the  $\fj$-action  on $\sfh(\cf)$
 can be exponentiated to a $J$-action. Specifically, write
 $\cf=\underset{\longleftarrow}\lim\,\cf_i$. 
  Fix $i$ and let $X\subset\Gr$ be a $G(O)$-stable projective subvariety that contains $\supp\cf_i$.
    Since $X$ is a union of Iwahori orbits,  $H^\hdot_{G(O)}(X)$ is a free $\sfc$-module so
 $H^{G(O)}_\idot(X)=H^\hdot_{G(O)}(X)^\vee$.  Let $\{h^i\}$
 and $\{h_i\}$ be dual bases of 
 $H^\hdot_{G(O)}(X)$ and $H^{G(O)}_\idot(\Gr)$, respectively, with respect to the
perfect pairing $H^\hdot_{G(O)}(X) \times H^{G(O)}_\idot(\Gr)\to \sfc$.
Define a map $\sfh(\cf_i)\to H_\idot^{G(O)}(X)\o_\sfc \sfh(\cf_i)$ 
by the formula
$m\mto \sum_i\,h_i \o (h^im)$. This map is independent of the choice of bases and
it was shown in \cite{YZ} that the composite
$\sfh(\cf_i)\to H_\idot^{G(O)}(X)\o_\sfc \sfh(\cf_i)\to H^{\BM}_{G(O)}(\Gr)\o_\sfc \sfh(\cf_i)$
gives a coaction of the coalgebra $\k[J]$ on $\sfh(\cf_i)$.
Furthermore, the differential of the resulting $J$-action agrees with the $\fj$-action
considered earlier.

             Any pure object of $\sat$ is (up to Tate twists of its mixed lift)
             isomorphic to a finite direct sum of objects of the form $\IC(\ogr{\la})[m]$
           for some $\la\in\dom$ and $m\in \Z$.
We say that $\ce\in\Ind\sat$ is pure if it is isomorphic to                
 a countable direct sum
 $\oplus_i\, \IC(\ogr{\la_i})[m_i]$ where $\la_i\in\dom$ and $m_i\in\Z$.

    Recall  the lax monoidal functor  $\td\so\td$ on $\sat$.
                   For  ind-objects $\cf=\underset{\leftarrow}\lim\, \cf_{n}$ and
$\cg
=\underset{\leftarrow}\lim\, \cg_{n}$, there is a well defined object
$\cf\so\cg=\underset{\leftarrow}\lim\, \cf_{n}\so\cg_{n}\in \Ind D_{G(O)}(\Gr)$.

\begin{lem}\label{Jj} If $\cf\in\Ind\sat$ is $!$-pure then $\sfh(\cf)$ is pure
  and we have:

              \vi   The cohomology $\sfh(\cf)$  is a direct limit of a sequence
            of  finite rank free $\sfc$-submodules of $\sfh(\cf)$;
            in particular,  the $\sfc$-submodule $\sfh(\cf)$ is faithfully flat.

            \vii If $\ce\in\Ind\sat$ is pure  then there is
             an isomorphism
              $\sfh(\ce\so\cf)\iso\sfh(\ce)\o_{\sfc}^J\sfh(\cf)$.
              Moreover, if  $\ce$ and $\cf$ are commutative ring objects
              then this is a ring isomorphism.
              \end{lem}
     \begin{proof}  
       Part (i) follows from Corollary \ref{inj cor} by taking $W$-invariants,
       cf. \eqref{GTH}.
       The proof of (ii) is based on the observation that for any
       $M\in \qcoh^{J}(\fcc)$ which is  flat over $\fcc$
  the natural inclusion
        $\Ga(\fcc, M^J)\sset\Ga(\fcc, M)^{\fj}$ is an equality.
To prove this, one uses that     the fibers of the restriction
        of the group scheme $J\to\fcc$ to a Zariski open dense subset
        $U\sset\fcc$ are  tori. This implies that
        for any section $m\in \Ga(\fcc, M)^{\fj}$,  the
        restriction of  $m$ to $U$ is fixed by the group scheme $J|_U$,
        since tori are connected.
        The sheaf $M$ being flat, it follows that $m$ is itself  fixed by $J$,
        as desired.
        
        Now, let $\ce,\cf$ be as in the statement.
        Thus, $\ce=\oplus_i\, , \IC(\ogr{\la_i})[m_i]$,
        resp. $\cf=\underset{\stackrel{\leftarrow}{i}}{\lim}\,\cf_i$,
        where $X_i\sset\Gr$ is a $G(O)$-stable projective variety that contains
        $\ogr{\la_k}$ for all $k\leq i$ and $\cf_i$ is a $!$-pure sheaf supported on $X_i$.
        By Sect. \ref{gen} we have
        $\sfh(\cf)
        =\underset{\to}{\lim}\,\sfh(\cf_i)$ where $\sfh(\cf_i)$ are    finite rank free $\sfc$-modules.
                We apply Corollary \ref{o cor} for $\cf=\cf_i$, resp. $\ce= \IC(\ogr{\la_i})$.
       Since the algebra $\sfh(\Gr)\cong\sym_\sfc\fj$ is generated by $\fj$,
        using \eqref{GTH} and taking $W$-invariants of $T$-equivariant cohomology        we deduce 
        \begin{align*}
          \sfh(\ce\so\cf)=\underset{\stackrel{\longrightarrow}{i}}{\lim}\,
        \Big(\bigoplus_{k\leq i}\, &\sfh(\IC(\ogr{\la_k})[m_k]\so\cf_i)\Big)\cong
       \underset{\stackrel{\longrightarrow}{i}}{\lim}\,
       \Big(\bigoplus_{k\leq i}\, \sfh(\ogr{\la_k})[m_k]\so_{\sfh(X_i)} \sfh(\cf_i)\Big)\\
       &\cong \sfh(\ce)\,\so_{\sfh(\Gr)} \,\sfh(\cf)\cong \sfh(\ce)\,\o_\sfc^\fj \,\sfh(\cf)\cong
       \sfh(\ce)\,\o_\sfc^J\,\sfh(\cf),
       \end{align*}
       where the last isomorphism follows    from the first paragraph of the proof
       since $\sfh(\ce)$ and $\sfh(\cf)$ are flat $\sfc$-modules.
       The  statement about ring structures   is a direct consequence
        of the lax monoidality of the functors $\td\so\td$ and $\sfh$.
          \end{proof}

\subsection{Geometric Satake equivalence and the functor $\Phi$}
\label{appl sat sec}    
 
Recall that $\dom\sset \BX^*(\dt)$ denotes the dominant Weyl chamber.
 The abelian Satake category $\per$
  is known to be a semisimple category with simple objects 
 $\IC_\la=\IC(\ogr{\la}),\,\la\in\dom$; furthermore, this category is
 stable under convolution, \cite{MV}.
  All objects of
 $\per$ are  pure, hence equivariantly formal. It follows that
  the lax monoidal functor $\sfh$
  restricts to a monoidal functor $\sfh: \per\to \sfc\proj$,
  where $\sfc\proj$ is the category of finitely generated graded projective,
  hence free $\sfc$-modules.

  We now equip the convolution product
   on $\per$ with a (super)symmetry constraint,
    bypassing  the Beilinson-Drinfeld Grassmannian and fusion,
     \cite{MV}, \cite{BD}.
    To this end, for any $\cf_1,\cf_2\in\per$ we consider the following chain
    of isomorphisms:
    \begin{align}
      \Hom_{\perv_{\gc(O)}(\Gr)}(\cf_1\!\star\! \cf_2, \cf_2\!\star\! \cf_1) &\iso
 \Ext^0_{D_{\gc(O)}(\Gr)}(\cf_1\!\star\! \cf_2, \cf_2\!\star\! \cf_1)\nonumber\\                                  
&\iso   
  \Hom^0_{\sfh(\Gr)}\big(\sfh(\cf_1\!\star\! \cf_2), \sfh(\cf_2\!\star \!\cf_1)\big)
  \label{symm}
      \\   &\iso\Hom^0_{\sfh(\Gr)}\big(\sfh(\cf_1)\o_\sfc \sfh(\cf_2),\,\sfh(\cf_2)\o_\sfc\sfh(\cf_1)\big),
     \nonumber
    \end{align}
    where the first  isomorphism holds since the category $\per$ is the
    heart of the perverse t-structure on $\sat$
    and the second isomorphism holds by Theorem \ref{Vthm}.
 The cocommutativity of the coproduct $\De$ on the Hopf algebra $\sfh(\Gr)$ provides a canonical
  isomorphism
  $\phi_{\cf_1,\cf_2}: \sfh(\cf_1)\o_\sfc\sfh(\cf_2) \iso \sfh(\cf_2)\o_\sfc\sfh(\cf_1)$
  of $\sfh(\Gr)$-modules.
      We define the desired super-symmetry constraint
      $\varphi_{\cf_1,\cf_2}: \cf_1\star \cf_2\iso \cf_2\star \cf_1$
      to be the preimage of  $\phi_{\cf_1,\cf_2}$
      under the composite isomorphism in \eqref{symm}.
      The signs appearing in the super-symmetry constraint are removed either as
      it has been traditionally done in
       \cite{BD}, \cite{MV}, or using shearing as in \cite{BZSV}.
            All compatibilities required   for the resulting symmetry constraint
      are immediate consequences of the axioms
      of  Hopf algebra.

Next, we recall the construction of the monoidal   equivalence
$\Phi: D_{G(O)}(\Gr)\iso D^{\dg}(\sg)$ introduced in
\cite[Sect. 5]{BFN2}. This functor is related 
to the derived Satake equivalence $\Psi: \sat\iso
D^{\dg}_{\text{perf}}(S\dgl)$ 
of Bezrukavnikov and Finkelberg by the formula
$\Phi=\chev\ccirc \Psi\inv$,
where $\chev^*$ is the auto-equivalence induced by a Chevalley
   involution $\chev: \dg\to\dg$.  
Let $\bone=\IC(\Gr_0)$ denote
      the skyscraper sheaf at the base point of $\Gr$.
We may view $\C[\dg]$
    as a ring object of $\Ind\Rep \dg$  via the $\dg$-action on itself by
    {\em left} translations. 
    Therefore, the regular sheaf  $\rr:=\phi\inv(\C[\dg])=
    \oplus_{\la\in\dom}\,\IC(\ogr{\la})\o V_\la^*$ has
    the structure of a ring object in $\Ind\per$. Thus, $\rr$ is equipped
    with a morphism $m: \rr\star\rr\to\rr$
    and the unit $\bone\to\rr$ subject to natural compatibilities.
     The $\dg$-action on the $V_\la^*$'s, i.e. the action on $\C[\dg]$ by {\em right} translations,
    gives an action of $\dg$ on $\rr$ by ring object automorphisms.

    
    Following \cite[Sect. 5(vii)]{BFN2}, we define the     functor
    \[\Phi: \sat\to D_{\text{perf}}^\dg(\sfc),\en
      \cf\mto \rgo(\Gr, \rr\so\cf).
      \]
      This functor has a monoidal structure constructed as follows.
      First of all, the   lax monoidal structure on  $\so$     provides,
      for each pair $\cf_1,\cf_2\in \sat$, a morphism
    \[
      (\rr\so \cf_1)\star(\rr\so\cf_2)\to (\rr\star\rr)\so(\cf_1\star\cf_2)
      \xrightarrow{m} \rr\so (\cf_1\star\cf_2).
      \]
           Applying  $R\Ga_{G(O)}(\Gr,\td)$ to  the composite
           we get, since  $R\Ga_{G(O)}(\Gr,\td)$ is a monoidal  functor,
           a lax monoidal structure
           $\Phi(\cf_1)\,\lo_\sfc \,\Phi(\cf_2)\,\to \, \Phi(\cf_1\star\cf_2)$ on $\Phi$.
    It is not difficult to show that the morphism
  $\Phi(\bone)\,\lo_\sfc \,\Phi(\bone)\,\to \, \Phi(\bone)$
  can be
  upgraded to a dg algebra structure on $\sfr:=\Phi(\bone)$,
  cf. Sect. \ref{pf sec1}. Similarly,
  the morphisms $\Phi(\bone)\,\lo_\sfc \,\Phi(\cf)\,\to \, \Phi(\cf)$ provide an
   upgrading of $\Phi$ to a functor
   $\Phi: \sat\to D_{\text{perf}}^\dg(\sfr)$.
    
   Let $S\dgl=\sg$.
   Next, mimicing arguments in \cite[Sect. 6.5]{BF} based on purity,
   one shows that the dg algebra
   $\sfr$
   is formal, i.e. quasi-isomorphic to the graded $\sfc$-algebra $H^\hdot(\sfr)$,
   cf. Sect.
      The derived Satake equivalence $\Psi: \sat\iso D^{\dg}_{perf}(S\dgl)$ of \cite{BF} sends
   $\bone$ to $S\dgl$, resp. $\rr$ to $S\dgl\o\C[\dg]$.
   We deduce
   \begin{align*}
   H^\hdot(\sfr)\cong  \sfh(\bone\so\rr)\cong
   \Ext_\sat(\bone,\rr)\cong\Hom_{D_{\text{perf}}^\dg(S\dgl)}(S\dgl,\,S\dgl\o\C[\dg])
                  \cong (S\dgl\o\C[\dg])^\dg\cong S\dgl.
                    \end{align*}
                    Thus,  the functor $\Phi$ may be identified with
                    a functor $\sat\to D_{\text{perf}}^\dg(S\dgl)$.
      Furthermore, by \cite[Lemma 5.13]{BFN2}, this functor is an equivalence and 
   for any $\cf\in\per$ one has
   $\Phi(\cf)\cong \chev^*(S\dgl\o \phi(\cf))$.
   For $\cf_1,\cf_2\in\per$, writing $\Psi(\cf_i)=V_i$, we compute
   \begin{align*}
     \Phi(\cf_1)\o_{\sym\dgl}\Phi(\cf_2) & =
                                           \chev^*(S\dgl\o V_1)\o_{\sym\dgl}\chev^*(S\dgl\o V_2)\\
     &\cong
     \chev^*(S\dgl\o V_1\o V_2)\cong \Phi(\cf_1\star\cf_2).
   \end{align*}
   Since the objects of $\per$ generate the derived Satake
   category, it follows that the lax monoidal structure on $\Phi$ is actually monoidal.
   
 \subsection{Functoriality for the universal centralizer}
 \label{Jsec} The goal of this subsection is to prove the following
 result.
 
 \begin{prop}\label{JHG} A morphism $\rho: H\to G$ of connected reductive groups
   induces a 
 morphism $\fcc_\dh\times_{\fcc_\dg } J_\dg$ $\to J_\dh$ of group schemes over $\fcc_\dh$.

 If $\rho$ is injective then the group scheme $J_\dh\to \fcc_\dh$
 has  a natural $N_G(H)/H$-equivariant structure.
\end{prop}

We will give two proofs of the proposition. The first proof
relies on the 
  description of the universal centralizer $J_\dg$ in terms of the functor of
  points given  \cite{DG}, \cite{BFM}, cf. also \cite[Proposition 2.2.3]{Ng1}.
The second proof, which is well known, relies on geometric Satake.

First, we recall  the construction of the functor of
  points.
   Let
  $\text{Sch}_{\fcc_\dg}$, resp. $\text{Ab}$, be the category of schemes over
  $\fcc_\dg=\dgl^*/\!/\dg$, resp. 
 abelian groups. Let $\Map_{\fcc_\dg}(\t_\dg^*\times_{\fcc_\dg} \td,
 T_\dg)^{W_\dg}:
 \text{Sch}_{\fcc_\dg}\to \text{Ab}$ be the functor
 that sends a  test scheme $S\in \text{Sch}_{\fcc_\dg}$  to  the abelian group
   of $W_\dg$-equivariant morphisms $\t_\dg^*\times_{\fcc_\dg}  S\to T_\dg$,
 of schemes over $\t_\dg^*$. 
 It is known that this functor is represented by an
 abelian group scheme $J^1_\dg$
 smooth over $\fcc_\dg$;
 furthermore, the universal centralizer $J_\dg$ is isomorphic to
 an open subgroup scheme of $J_\dg^1$ defined 
 by the condition that the group $\pi_0(Z_\dg(G))$ maps surjectively onto 
 $\pi_0(F)$ for  each closed fiber $F$ of $J_\dg\to \fcc_\dg$, cf. \cite{Ng1}, \cite{Ng2}. 
Here, $\pi_0(\td)$  denotes the set of connected components.

 Given a subgroup $M\sset G$ and subsets $S, S'$ of $G$
 we write $Z_M(S)=\{g\in M\mid gs=sg\ \forall s\in S\}$, resp.
 $N_M(S)=\{g\in M\mid gSg\inv\sset S\}$ and
 $N_M(S, S')=N_M(S)\cap N_M(S')$.

 Let $H$ be a  reductive subgroup of $ G$,
 and  $T\sset G$ a maximal torus such that
 $T_H= H\cap T$ is  a maximal torus of $H$.
 The group  $L:=Z_G(T_H)$ is a Levi subgroup of $G$.
 Write
 $W=N_G(T)/T$, resp. $W_H=N_H(T_H)/T_H$ and $W_L=N_L(T)/T$,  for the
 respective Weyl groups.
  It is easy to show that the group $W_{L;G}:=N_G(L,T)/T\sset W$
  equals the normalizer of $W_L$ in $W$.

  The following result is similar to \cite[Lemme 10.1]{Ng2}.
  
  \begin{lem}\label{weyl lem} Let   $H\sset G$ be a connected reductive subgroup.
    Then,   there is an isomorphism  $N_G(L)/L\cong W_{L;G}/W_L$ and
    a natural group imbedding 
    $W_H\into W_{L;G}/W_L$. Further, there is
    a 
    short exact sequence
   \[
     1 \to W_H  \to N_G(T_H,T)/T_H  \xrightarrow{q} N_G(H)/H  \to 1,
     \]
     which has a  splitting  provided by
     a group homomorphism $s: N_G(H)/H \to N_G(T_H,T)/T_H$.
 \end{lem}
 \begin{proof} 
   To prove the first statement,  observe that
   the  kernel of the composite $N_G(L,T)\into N_G(L)\onto N_G(L)/L$
   equals $N_G(L,T)\cap L=N_L(T)$. Thus, we get an injective homomorphism
   $W_{L;G}/W_L\cong N_G(L,T)/N_L(T)\into  N_G(L)/L$.
      For any $g\in N_G(L)$ the torus
   $gTg\inv$ is a maximal torus of $L$, hence one can find $\ell\in L$ such that
   $gTg\inv=\ell T\ell\inv$. Hence $\ell\inv g\in N_G(L,T)$.
   It follows that the
   homomorphism $W_{L;G}/W_L\to N_G(L)/L$ is an isomorphism.
   Next, using that
   $N_G(T_H)\sset N_G(Z_G(T_H))=N_G(L)$ we deduce inclusions
   $N_H(T_H)\sset N_G(T_H, H)\sset N_G(L, H)\sset N_G(L)$.
   The kernel of the composite  $N_H(T_H)\into N_G(L)\to N_G(L)/L$ equals
   $N_H(T_H)\cap L$ $=Z_H(T_H)=T_H$. Therefore, the composite
   induces an imbedding $W_H=N_H(T_H)/T_H\into N_G(L)/L$, proving the second statement.

   Similarly, the kernel of the composite $N_G(T_H, H)\into N_G(H)
   \onto N_G(H)/H$ equals $N_H(T_H)$,
   so we get an imbedding $N_G(T_H, H)/N_H(T_H)\into N_G(H)/H$.
   For any $g\in N_G(H)$ the torus $gT_Hg\inv$ is a maximal torus of $H$.
   Arguing as above, we deduce that the imbedding is a bijection. Thus, we obtain isomorphisms
   \[
     (N_G(T_H, H)/T_H)/W_H\cong N_G(T_H, H)/N_H(T_H)\iso N_G(H)/H.
   \]
   This yields the short exact sequence of the lemma.

   To complete the proof observe that the action of
   $N_G(T_H, H)$ in $\BX^*(T_H)$ takes the set of roots of
   $H$ to itself and takes Weyl chambers to Weyl chambers,
   since it sends any Borel $B_H\supseteq T_H$ of $H$ to another such Borel.
   Fix a Weyl chamber $C\sset \BX^*(T_H)$  and let
   $N^+_G(T_H, H)=\{g\in N_G(T_H, H)\mid g(C)\sset C\}$.
   Then, $N^+_G(T_H, H)\cap N_H(T_H)=\{1\}$ and
      the composite
   $N^+_G(T_H, H)\into N_G(T_H, H)\onto N_G(T_H, H)/N_H(T_H)$
   is an isomorphism.
   Inverting this isomorphism, we get a chain of
   homomorphisms
   \[
     N_G(H)/H\iso N_G(T_H, H)/N_H(T_H)\iso N^+_G(T_H, H) \into
     N_G(T_H, H)\onto N_G(T_H, H)/T_H.
   \]
   The composite provides the desired splitting of the short exact sequence.
 \end{proof}

   \begin{proof}[First proof of Proposition \ref{JHG}]
Let $\rho: H\to G$  be a not necessarily injective morphism
 of reductive groups. 
Choose maximal tori
 $T_H\sset H$ and $T\sset G$
so that $\bar T_H:=\rho(T_H)\sset T$. 
There is a canonical map of Weyl groups
   $W_H=N_H(T_H)/T_H\to W_{\rho(H)}=N_{\rho(H)}(\bar T_H)/\bar T_H$.
   We consider  the Levi subgroup $L=Z_G(\bar T_H)$ of $G$
   and the dual Levi  $\check L\sset \dg$.
      We have 
   the identification  $W_{\check L}=W_L$, resp.  $W_\dg=W_G$ and $W_\dh=W_H$. Thus, 
   we identify $W_{L;G}$ with a subgroup $W_{\check L;\dg}\sset W_\dg$.
   In  the setting of Lemma \ref{weyl lem} we may (and will) replace $H$ by $\rho(H)$
   and
   let $\rho_W: W_\dh\to W_{{\check L},\dg}/W_{\check L}$ be the composite
   of   the following maps
      \beq{weyl}
            W_\dh=W_H \to W_{\rho(H)}
      \xrightarrow{\ \text{Lemma \ref{weyl lem}} \ }
      W_{L;G}/W_L=W_{{\check L},\dg}/W_{\check L}.
   \eeq
     By construction, for all $w\in W_\dh$ and $t\in T_H$, one has
     $\rho(w(t))=\rho_W(w)(\rho(t))$. We have the dual map
     $\rho^*_\BX: \BX^*(T)\to \BX^*(T_H)$, resp. $\check\rho:  T_\dg\to T_\dh$
     and $\rho^*: \t^*_\dh\to\t^*_\dg$. Dualizing, 
   we deduce the equation
   $\rho^*_\BX(\rho_W(w)(\la))=w(\rho^*_\BX(\la))$
   for all  $\la\in \BX^*(T)$.
   Hence, similar equations hold for $\check\rho$ and $\rho^*$.

   Given a  test scheme $S\in\text{Sch}_{\fcc_\dg}$ and
 $f\in\Map(\t_\dg^*\times_{\fcc_\dg} S, T_\dg)$ 
 we form a  chain of maps
 \[\t_\dh^*\times_{\fcc_\dh} S\xrightarrow{\,\rho^*\,} \t_\dg^*\times_{\fcc_\dg} S
   \xrightarrow{\,f\,}
   T_\dg\xrightarrow{\,\check\rho\,} T_\dh.
   \]
If the map $f$ is $W_\dg$-equivariant
   then  it follows from the above that
   the composite $\check\rho\ccirc f$
   is $W_{\check L;\dg}$-equivariant.  Hence, for any
   $\la\in \t_\dh^*,\, s\in S$, and $w\in W_\dh$, we find
   \[(\check\rho\ccirc f\ccirc \rho^*)(w(\la)\times s)=
     (\check\rho\ccirc f)(\rho^*(\rho_W(w)(\la)\times s))=
     \rho_W(w)\big((\check\rho\ccirc f\ccirc \rho^*)(\la\times s)\big).
     \]
      We conclude that the assignment $f\mto \check\rho\ccirc f\ccirc \rho^*$ yields
            a morphism of functors
   $\Map(\t_\dg^*\times_{\fcc_\dg} (\td), T_\dg)^{W_\dg}$  $\to
   \Map(\t_\dh^*\times_{\fcc_\dh}(\td), T_\dh)^{W_\dh}$, hence a morphism 
   $\fcc_\dh\times_{\fcc_\dg } J^1_\dg\to J_\dh^1$ of group schemes over $\fcc_\dh$.
    We must show
   that this morphism 
  sends  $\fcc_\dh\times_{\fcc_\dg } J_\dg$ to $J_\dh$. 
By the description of $J$ given above,
the map $\pi_0(Z(\dg))\to \pi_0(F)$ is surjective,
for every fiber $F$ of $\fcc_\dh\times_{\fcc_\dg } J_\dg\to\fcc_\dh$.
  The action of $\pi_0(Z(\dg))$ in $F$ agrees with
  the restriction to the image of 
  the natural homomorphism $\pi_0(Z(\dg))\to \pi_0(Z(\dh))$
  of the $\pi_0(Z(\dh))$-action 
  in the fiber of $J^1_\dh$ that contains the image of $F$.
  The first statement of the proposition follows from this.

To prove the second statement of the proposition we may
identify $H$ with $\rho(H)$.
 Then, the exact sequence of Lemma \ref{weyl} shows that
 the natural action of the group $N_G(T_H,T)$ on $\t^*_\dh=\t_H$
 induces an   action of $N_G(H)/H$ in $\fcc_\dh$. Similarly, we obtain
 an $N_G(H)/H$-action on the functor $\Map(\t_\dh^*\times_{\fcc_\dh}(\td), T_\dh)^{W_\dh}$.
 It is straightforward to check that this
 induces an action  on the group scheme $J_\dh\to\fcc_\dh$.
\end{proof}

 \begin{proof}[Second proof of Proposition \ref{JHG}]      A morphism $\rho: H\to G$ of
     reductive groups induces      an  ind-proper
     map $\rho_\Gr: \Gr_H\to \Gr_G$.  Since $\rho^!_\Gr\om_{\Gr_G}\cong\om _{\Gr_H}$,
     by adjunction one gets
     an algebra morphism $H^\hdot_{H(O)}(\om_{\Gr_H})=
     H^\hdot_{H(O)}(\rho^\dagger\om_{\Gr_G})$
     $\to
     \C[\fcc_\dh]\otimes_{\C[\fcc_\dg]} H^\hdot_{G(O)}(\om_{\Gr_G})$.
   Thus, we obtain an algebra map
 \[\C[J_\dh] \cong H^{\BM}_{H(O)}(\Gr_H)\too
 \C[\fcc_\dh\times_{\fcc_\dg} J_\dg]\cong  \C[\fcc_\dh]\o_{\C[\fcc_\dg]}\,H^{\BM}_{G(O)}(\Gr_G).
   \]
   One can check that this map respects the bialgebra structures and
   the induced morphism
      $\rho^*J_\dg \to J_\dh$ of group schemes over
      $\fcc_\dh$ is equal to the morphism from Corollary \ref{JHG}.
      Similarly, the action of $N_G(H)/H$
on $H^{\BM}_{H(O)}(\Gr_H)$ induced by the
   natural $N_G(H)$-action 
   on $\Gr_H$ agrees with  the 
   action of  $N_G(H)/H$  on
      $J_\dh$  from  Corollary \ref{JHG}.
\end{proof}
\section{Proofs of main results}
    \subsection{Proof of Theorem \ref{sat thm}}
   \label{pf sec1}
   First we sketch the proof of the formality statement.
   To this end, for each $\la\in\dom$ put
   $\rr_{\leq\la}=\oplus_{\mu\leq\la}\, \IC(\ogr{\mu})\o V^*_\mu$.
   In \cite[Sect.\,6.5]{BF} the authors constructed
   an upgrading of the algebra
   $\Ext^\hdot_{D_{G(O)}(\ogr{\la})}(\rr_{\leq\la},\,\rr_{\leq\la})$
   to a dg algebra. Using that these $\Ext$'s  are pure
   it was shown in {\em loc cit} that the constructed dg algebra is formal.
   The derived Satake equivalence is deduced from this by passing
   to a limit $D_{G(O)}(\Gr)=\underset{\too}\lim\,D_{G(O)}(\ogr{\la})$.

   The technology of $\infty$-categories allows to package arguments in {\em loc cit}
   as follows. 
   Fix a prime power $q$ and an Artin stack $\cy_1$ over $\BF_q$ with finite
   stabilizers. For each $n\geq 1$ let $\cy_n$ denote the stack obtained from
   $\cy_1$ by base change 
   to $\BF_{q^n}$. There is a (stable presentable) $\infty$-category
   $\ssh_c(\cy_n)$ of constructible $\overline{{\mathbb Q}_\ell}$-sheaves
   on $\cy_n$ defined in \cite[Sect. 4.1]{HL}, cf. also \cite{GL}. This category has a full subcategory
   $\ssh_{m,c}(\cy_n)$
   of mixed constructible sheaves,  defined in  \cite[Sect. 4.2]{HL}.
   The category  $\ssh_{m,c}(\cy_n)$ comes equipped with an auto-equivalence
   $Fr^*_{\BF_{q^n}/\BF_q}$ given by the action of geometric Frobenius.
   Finally,  in  \cite[Sect. 4.4]{HL} the authors introduce a dg  $\infty$-category
   $\ssh_{gr,c}(\cy_n)$ of {\em graded} constructible sheaves which  is
  obtained from $\ssh_{m,c}(\cy_n)$
  by an appropriately defined  `semisimplification' of the Frobenius action.
  The category $\ssh_{gr,c}(\cy_n)$ comes equipped with an auto-equivalence
  $\cf\to \cf\langle1\rangle$, the Tate twist.
  In the special  case of the stack  $\pt_n:=\Spec \BF_{q^n}$, the category $\ssh_{gr,c}(\pt_n)$
  is, by definition, the category $\vect^{gr}_c$ of perfect chain complexes of graded
  $\overline{{\mathbb Q}_\ell}$-vector spaces.
  Tensor product of complexes gives 
  $\vect^{gr}_c$ the structure of a symmetric monoidal category and 
    $\ssh_{gr,c}(\cy_n)$ has, by construction,  the structure of a $\vect^{gr}_c$-module category,
  for any $\cy_n$. This allows to upgrade $\Hom$'s in $\ssh_{gr,c}(\cy_n)$
  to $\Hom$'s  enriched over $\vect^{gr}_c$. Thus, for any objects
  $\cf_1,\cf_2\in \ssh_{gr,c}(\cy_n)$ there is an internal Hom-object
  $\chom(\cf_1,\cf_2)\in\vect^{gr}_c$. 

  In \cite[Sect. 5]{HL} the category  $\ssh_{gr,c}(\cy_n)$ was equipped with a perverse t-structure
  and with a {\em weight structure}
  in the sense of \cite{Bon}. The latter structure comes from the weight
  filtration on mixed sheaves, as in \cite{BBD}.
  Associated with each of these two structures there is
  an $\infty$-subcategory $\ssh_{gr,c}(\cy_n)^{\heartsuit_t}$, resp. $\ssh_{gr,c}(\cy_n)^{\heartsuit_w}$,
  of  $\ssh_{gr,c}(\cy_n)$,  the   heart. An object $\cf$ of $\ssh_{gr,c}(\cy_n)$ is said
  to be pure of weight $m$ if $\cf[-m]\in \ssh_{gr,c}(\cy_n)^{\heartsuit_w}$.
  There is a natural lift of any $\IC$-sheaf of geometric origin to a pure object
  of $\ssh_{gr,c}(\cy_n)$, \cite[Sect. 5]{HL}.
  Pure objects of $\vect^{gr}_c$ are perfect complexes  $V$ such that the $i$-th cohomology 
  $H^i(V)$ is concentrated in grade degree $i$. In that case, one has
  $V\cong\oplus_i\, H^i(V)\langle-i\rangle[-i]$.

  Now, let  $G_{\BF_q}$ be  a split reductive group over $\BF_q$
  and $\Gr_{\BF_q}$ an associated affine Grassmannian. Then,
  for each $\la\in \dom$ we consider the stack  $\cy_{\leq\la}:=G_{\BF_q}(O)\back\ogr{\la}$
  and  for each  $n$ let $\cy_{\leq\la,n}:=(\cy_{\leq\la})_n$.
  The perverse sheaf $\rr_{\leq\la}$ has a lift
  to a pure object $\rr^{gr}_{\leq\la}\in\ssh_{gr,c}(\cy_{\leq\la,n})$
  and the algebra $\Ext^\hdot_{D_{G(O)}(\ogr{\la})}(\rr_{\leq\la},\,\rr_{\leq\la})$
  has an upgrading
  $\chom^{gr}(\rr_{\leq\la},\,\rr_{\leq\la}):=\oplus_{i,j}\,\chom(\rr_{\leq\la},\,\rr_{\leq\la}\langle i\rangle[j])$,
  which is a ring object of the category $\vect^{gr}$.
  The purity of $\Ext^\hdot_{D_{G(O)}(\ogr{\la})}(\rr_{\leq\la},\,\rr_{\leq\la})$
  translates into the statement that $\chom^{gr}(\rr_{\leq\la},\,\rr_{\leq\la})$
  is a pure object of $\vect^{gr}$, hence this object is isomorphic
  to 
  $\oplus_i\, H^i(\chom^{gr}(\rr_{\leq\la},\,\rr_{\leq\la})\langle-i\rangle[-i]$.
  This isomorphism is equivalent to the formality result proved in \cite[Sect.\,6.5]{BF}.

  A similar argument applies to $!$-pure objects of  $D^{mix}_{G_{\BF_{q}}(O)}(\ogr{\la})$
  as follows.
  In general, for any stack $\cy_n$, the forgetful functor
  $\ssh_{m,c}(\cy_n)\to D^{mix}(\cy_n)$  is essentially surjective
  by the definition of the category $\ssh_{m,c}(\cy_n)$, \cite[Sect. 4.4]{HL}.
  Therefore, any $!$-pure object $\cf\in D^{mix}(\cy_n)$ has a lift to an object of $\ssh_{m,c}(\cy_{\leq\la,n})$.
  The image of the
  lift under the `semisimplification of Frobenius'  functor
  $\ssh_{m,c}(\cy_{\leq\la,n})\to \ssh_{gr,c}(\cy_{\leq\la,n})$, see \cite[Sect. 1.1]{HL},
  yields an object $\cf^{gr}\in \ssh_{gr,c}(\cy_{\leq\la,n})$.
  Since $\cf$ is $!$-pure, the group
  $\Ext^\hdot_{D_{G_{\BF_{q}}(O)}(\ogr{\la})}(\rr_{\leq\la},\,\cf)\cong
  H^\hdot_{G_{\BF_{q}}(O)}(\BD\rr_{\leq\la}\so\cf)$ is pure.
  Arguing as above, we deduce that the internal Hom-object
  $\chom^{gr}(\rr_{\leq\la},\,\cf^{gr}):=\oplus_{i,j}\,\chom(\rr_{\leq\la},\,\cf^{gr}\langle i\rangle[j])
  \in\vect^{gr}$ is pure and, therefore,  isomorphic to the direct sum
  $\oplus_i\, H^i(\chom^{gr}(\rr_{\leq\la},\,\cf^{gr})\langle-i\rangle[-i]$.
  It follows that $\chom^{gr}(\rr_{\leq\la},\,\cf^{gr})$, viewed as a module object over
  the ring object $\chom^{gr}(\rr_{\leq\la},\,\rr_{\leq\la})$,
  is isomorphic to $\Ext^\hdot_{D_{G_{\BF_{q}}(O)}(\ogr{\la})}(\rr_{\leq\la},\,\cf)$,
  viewed as a graded module over the graded algebra
  $\Ext^\hdot_{D_{G_{\BF_{q}}(O)}(\ogr{\la})}(\rr_{\leq\la},\,\rr_{\leq\la})$.
  The desired formality of the dg module $\Phi(\cf)\in D^{\dg}_{\text{perf}}(\sg)$
  follows from this using that the functor $\Phi$ is isomorphic
  to the composition of the functor $\RHom(\rr,\td)$ and the functor
  $\chev$ induced by a Chevalley automorphism, cf. Sect.  \ref{appl sat sec}.
  
To prove the remaining  statements of Theorem \ref{sat thm},  
 recall the flat  morphism $p: \WH\to\fcc$, cf. Sect. \ref{sat int}.
 Since  $\WH$ is an affine variety, for any $M\in \qcoh^J(\fcc)$
 one has $\C[\WH]\o_\sfc \Ga(M)\cong\Ga(p^*M):=\Ga(\WH, p^*M)$.
 If, moreover, $M$ is  flat over $\fcc$ then, using the proof of
 Lemma \ref{Jj}, we deduce
  \[\C[\WH]\o_\sfc^J \Ga(M)\cong
    \Ga(p^*M)^{\fj}\cong
    \Ga((p^*M)^J)=\Ga(\greg, \wh(M))=\Ga(\dgl^*,\,\jmath_*\wh(M)).
   \]

   We know that $\Phi(\rr)=\sg\o\C[\dg]$ by \cite[Sect.\,5(vii)]{BFN2}, hence
   $\sfh(\Phi(\rr)\cong \C[\WH]$, the Kostant-Whittaker reduction of $\sg\o\C[\dg]$,
   by \cite[Theorem 2]{BF}. Since the sheaf $\cf$ is $!$-pure, the  $\sfc$-module
   $\sfh(\cf)$ is free and 
 from Lemma \ref{Jj} we obtain  
\beq{rrcf}
  \HF(\cf)\cong \sfh(\rr\so\cf)\cong \sfh(\rr)\o_{\sfc}^J\sfh(\cf).
 \eeq
 Hence, taking $M=\chh_{\gc(O)}(\cf)$ in the displayed formula above,
 we deduce an isomorphism $\HF(\cf)$ $\cong \Ga(\dgl^*,\,\jmath_*\wh(\chh_{\gc(O)}(\cf)))$,
 and the desired isomorphism
   $\chh\Phi(\cf)\cong\jmath_*\wh(M)$ follows.
 \qed

 \subsection{Proof of Theorem \ref{fin}}\label{filt state}
 Recall the notation $\hh(\td)=H^\hdot_T(\td)$,
 resp. $\sfh(\td)=H^\hdot_{G(O)}(\td)$.
 The following proposition, which is part of the  statement
  of       Theorem \ref{fin}, will be proved in  Section \ref{gr pf}.

  \begin{prop}\label{part} Let $\cf\in\Ind D_{\gc(O)}(\Gr_G)$ be a $!$-pure commutative
  ring object. If
  the algebra $\hh(\veps^\dagger\cf)_+$  is finitely generated
  then so are the algebras $\hh(\veps^\dagger\cf)^W$ and  $\sfh(\cf)$.\qed
\end{prop}

We now prove Theorem \ref{fin}
assuming the proposition.
First, we show that if the algebra  $\hh(\veps^\dagger\cf)_+$  is finitely generated
then so is
      the  algebra $\HF(\cf)$.
            Since $\HF(\cf)=\sfh(\rr\so \cf)$,  by Proposition \ref{part}  it suffices to show that
      the algebra $\hh(\veps^!(\rr\so \cf))_+$  is finitely generated.
  To this end, write
  $\hh(\veps^\da\rr)=\oplus_{\la\in\BX}\,\hh(\veps^\da_\la\rr)$,
  resp. $\hh(\veps^\da\cf)=\oplus_{\la\in\BX}\,\hh(\veps^\da_\la\cf)$.
  By \cite{GK}, there is an $\BX$-graded algebra
  isomorphism
  \beq{gkd}
 \hh(\veps^\da\rr) =\bigoplus_{\nu\in\BX}\, \hh(\veps_\nu^\da\rr)\; \cong\;
 \C[T^*(\dg/\du)]=\bigoplus_{\nu\in\BX}\, \C[T^*(\dg/\du)]_\nu.
 \eeq
 Since $\veps^\da(\rr\so\cf)\cong(\veps^\da\rr)\so(\veps^\da\cf)$ we obtain
   isomorphisms
  \begin{align*}
        \hh(\veps^\da(\rr\so\cf))
    \cong \hh((\veps^\da\rr)\so(\veps^\da\cf))
    &\cong    \oplus_{\la\in\BX}\ \hh(\veps^\da_\la\rr)\,\so_\fc\,\hh(\veps^\da_\la\cf)\\
    &\cong \oplus_{\la\in\BX}\ \hh(\veps^\da_\la\rr) \,\o_\fc \,\hh(\veps^\da_\la\cf)\\
&\cong \oplus_{\la\in\BX}\ \C[T^*(\dg/\du)]_\la\,\o_\fc \,\hh(\veps^\da_\la\cf),
 \end{align*}
of $\BX$-graded $\fc$-algebras,
where the  isomorphism in the second line holds by Remark \ref{oo}.

  We define a $\dt$-action on $\C[T^*(\dg/\du)]$,
  resp. $\hh(\veps^\da\cf)$, by letting $\dt$ act
  in $\C[T^*(\dg/\du)]_\nu$, resp.  $\hh(\veps_\nu^\da\cf)$, via the character
  $\nu$, resp. an inverse character $-\nu$.
  Let $\dt$ act in $\C[T^*(\dg/\du)]\o_\fc \hh(\veps^\da\cf)$ via the
  tensor product action. Then, from the isomorphisms above we obtain
  the following $\BX$-graded, resp. $\dom$-graded, algebra isomorphism:
    \begin{align}
  \hh(\veps^\da(\rr\so\cf))&\ccong \big(\C[T^*(\dg/\du)] \o_\fc
  \hh(\veps^\da\cf)\big)^\dt,\en\text{resp.}\label{rgk}\\
      (\hh(\veps^\da(\rr\so\cf)))_+
    &\ccong \big(\C[T^*(\dg/\du)]_+ \o_\fc \hh(\veps^\da\cf)_+\big)^\dt.
 \nonumber
      \end{align}

The algebra $\C[T^*(\dg/\du)]_+$ 
  is finitely generated by \cite{GR}, \cite{Bro}, and 
  the algebra $\hh(\veps^\da\cf)_+$
  is finitely generated by the assumptions of the theorem.
  It follows that $\C[T^*(\dg/\du)]_+ \o_\fc \hh(\veps^\da\cf)_+$,
  hence its subalgebra of $\dt$-invariants, is finitely generated.
  Thus, $\HF(\cf)$  is finitely generated, as desired.

  To prove the remaing statements of the theorem,
  put $X=\Spec \sfh^{\text{even}}(\cf)$ and let  $\mu: X\to\fcc$  be the structure morphism.
    Thus, $X$ is an affine  irreducible and normal  $J$-scheme and   we must show that
   the scheme $\WH(X)=(\WH\times_\fcc X)/J$,
  a twist of $X$ by
  the $\vpi^*J$-torsor $\pi: \WH\to\greg$,
   is quasi-affine, irreducible and normal.
  Using that the torsor $\pi$  has an \'etale local section, cf. eg.
   \cite[Lemma 3.2.3]{GK}, we deduce that the morphism 
  $\mu\times \pi:      \WH(X)\to \greg$ is \'etale locally isomorphic to the
  first projection $pr: \greg\times_\fcc X\to\greg$.
  This projection is an affine morphism since the scheme $X$ is affine.
  Since  $\greg$ is quasi-affine we deduce that $\WH(X)$ is quasi-affine.
    The variety $X$ is  irreducible and normal by the assumptions of the theorem.
 The morphism $\vpi$ being smooth, 
 it follows by smooth base change
 that $\WH(X)$ is  irreducible and normal.
 
 To complete the proof we use that the functors $\sfh$ and
 $\wh: \qcoh^J(\fcc)\to \qcoh^G(\greg)$ are monoidal.
 Hence,  Theorem \ref{sat thm} yields algebra isomorphisms
    $\hrm^{\text{even}}\Phi(\cf)\cong\Ga(\wh(\oo_X))\cong\C[\WH(X)]$.
  We obtain an isomorphim $\Spec \hrm^{\text{even}}\Phi(\cf)\cong
  \overline{\WH(X)}$. By the first part of the proof we know that
  $\Spec \hrm^{\text{even}}\Phi(\cf)$ is a scheme of finite type.
  Thus, it follows from the above that $\overline{\WH(X)}$
  is an irreducible and normal algebraic variety, and we are done.
  \qed

  \subsection{Functoriality}\label{iwahori sec}
  In this section we will prove Theorem \ref{vincent}, which requires some preparations.

Let $\shh$ be a reductive subgroup
of  a connected reductive  group ${G}$.
Let  $T_\shh\sset B_\shh\sset \shh$
be a maximal torus and a Borel subgroup of $\shh$.
We choose (as we may) a maximal torus and  a Borel subgroup
$T\sset B\sset {G}$
such that $T_\shh \sset T$, resp. $B_\shh\sset B$.
Let $\BX_*(T_\shh)\sset \BX_*(T)$ be the corresponding
cocharacter lattices.
Our choices insure that
 $T_\shh=\shh\cap T$, resp.
 $B_\shh=\shh\cap B$.
 Let  $L=Z_{G}(T_H)$ be the centralizer of $T_H$,
  a Levi subgroup of
  ${G}$ such that $T_H$ is contained in the center of $L$.
   Let $I=\{\gamma\in{G}(O)\mid \gamma(0)\in B\}$
be the Iwahori subgroup of ${G}(K)$ associated with $B$.
Then, the group
$I_\shh:=\shh(O)\cap I$ is  the Iwahori subgroup of $\shh(K)$
associated with $B_\shh=\shh\cap B$,
 and we have $\BX_*(T_\shh)=H(K)\cap \BX_*(T)$.

We identify
$\Gr_\shh$ with the image of the closed
imbedding  $\Gr_{\shh}\into \Gr_{{G}}$. 
Let $\BX_*(T)\iso (\Gr_{G})^{T}$, $
\la\mto x_\la$, be the canonical bijection and
$I.x_\la$,
resp. $I_\shh.x_\la$,  the  Iwahori orbit of $x_\la$.

\begin{lem}\label{iwahori} \vi For any $\la\in\BX_*(T)$ one has
  \[
    \Gr_\shh \cap I.x_\la=\begin{cases}   I_\shh.x_\la &\text{\em if}\ \en \la\in\BX_*(T_H)\\
      \emptyset&\text{\em if}\ \en \la\in\BX_*(T)\sminus\BX_*(T_H).
    \end{cases}
    \]

    \vii Conditions {\em S1-S3} of Sect. \ref{form sec}  hold
    for     $X=\Gr_G$, viewed as a  $T_H$-variety equipped with stratification
    $\Gr_G=\sqcup_{\la\in \BX_*(T_G)}\, I.x_\la$, and the subvariety $Y=\Gr_H$.
    \end{lem}
  
  \begin{proof} Let $\la\in \BX_*(T)$. Then $\Gr_\shh \cap  I.x_\la$ is an
      $I_\shh$-stable subvariety of $\Gr_\shh$, hence this subvariety
      is a union of $I_\shh$-orbits. Let $I_\shh.x_\mu$, where $\mu\in\BX_*(T_H)$,
      be such an orbit. Then, we have
      $x_\mu\in I_\shh.x_\mu\sset I.x_\la$. 
      This forces
      $\mu=\la$. It follows that  $\la\in\BX_*(T_H)$ and $\Gr_\shh \cap  I.x_\la=I_\shh.x_\la$,
      proving  (i). 

      Next, we  verify condition S1.  Let $I^u$ be the (pro)unipotent radical
      of $I$ and $I^u_\la$ the stabilizer of
      $x_\la, \la\in\BX_*(T)$ in $I^u$.
      The adjoint action of $T$ makes  the vector space
      $V_\la:=\Lie(I^u/I^u_\la)$    a finite dimensional 
      representation of $T_H$ and the map
      $\Lie(I^u) \to I.x_\la,\,v\mto \exp(v).x_\la$,  induces a
     $T$-equivariant, hence $T_H$-equivariant,  
     isomorphism $V_\la\iso I.x_\la$.
          Observe further that the group $\shh(O)\cap I^u$ equals the (pro)unipotent radical
          $I^u_H$ of the group  $I_\shh$.  Therefore, for $x_\la,\,\la\in \BX_*(T_H)$,   the stabilizer
          $I^u_{H,\la}$ of
      $x_\la$ in $I^u_H$  equals $\shh(O)\cap I^u_\la$.
      This gives  a $T_H$-equivariant imbedding
      $\Lie(I^u_H/I^u_{H,\la})\into \Lie(I^u/I^u_\la)=V_\la$ such that
      the isomorphism $V_\la\iso I.x_\la$ restricts to
      an isomorphism  $\Lie(I^u_H/I^u_{H,\la})
      \iso      I_\shh.x_\la$.
      Thus, an inverse  $f_\la: I.x_\la\to V_\la$ of
      the isomorphism  $V_\la\iso I.x_\la$
      satisfies condition S1.

      To check condition S2, observe  that     $(\Gr_{G})^{T_H}=\Gr_L$
      and any element of $\BX_*(T_H)$ is central in $L(K)$
      since the torus  $T_H$ is contained in the
      connected center of $L$.
      We know that $x_\la\in \Gr_H$ if and only if $\la\in\BX_*(T_H)$,
      by (i). In that case we find
             \[
    (I.x_\la)^{T_H}=    \Gr_L\cap I.x_\la=I_L.x_\la\cong I_L.x_\la.L(O)/L(O)=
     x_\la.L(O)/L(O) \cong
     \{x_\la\}.
   \]
   This  proves S2.      
       It  is well known that condition S3 holds for  affine Grassmannians.
  \end{proof}
         \medskip

  It follows from Lemma \ref{iwahori} that  results of Section \ref{sec2}
  are applicable in the above setting.
  Write $\rho:\Gr_\shh\into \Gr_{G}$
 for the
 imbedding.
  Then, Corollary \ref{o cor} implies that if
 $\cf\in D^{\text{mix}}_{I_H}(\Gr_{H})$ is    $!$-pure  and
 $\ce\in D^{\text{mix}}_{I_G}(\Gr_{G})$ is pure then
     there is a natural isomorphism
\beq{cor map}
   H^\hdot_{T_H}(\cf\,\so\, \rho^\dagger\ce)
  \iso H^\hdot_{T_H}(\cf)\,\so_{H^\hdot_{T_H}(\Gr_G)}\,H^\hdot_{T_H}(\ce).
 \eeq

We now turn to  Satake categories. 
We have $H^\hdot_{H(O)}(\om_{\Gr_H})=H^\BM_{H(O)}(\Gr_H)\cong \k[J_\dh]$.
Hence, taking $W_{H}$-invariants  in \eqref{cor map} in the case  $\cf=\om_{\Gr_H}$
and using \eqref{GTH},
by repeating the proof of Lemma \ref{Jj} we obtain the following result
 \begin{cor}\label{iwahori prop} If $\rho: H\into G$ is an imbedding of connected reductive groups
   then for any pure  $\ce\in\Ind \sat$
    there is a natural  isomorphism
    \beq{cor m2}
   H^\hdot_{H(O)}(\rho^\dagger\ce)\  \iso\
    \k[J_\dh]\o_{\k[\fcc_\dg]}^{J_\dg}
      H^\hdot_{G(O)}(\ce).
          \eeq
    If $\ce$ is a ring object then this  is a ring isomorphism.\qed
  \end{cor}

\medskip\noindent
\textsf{\em Proof of Theorem \ref{vincent}.}
 \textit{Step 1:\,  Proof of isomorphism \eqref{vin2}.}\en 
            We factor the morphism $\rho$ as a composition
      $H\xrightarrow{\varrho} H\times G\xrightarrow{pr_2} G$. 
      Then, we have $pr_2^!\rr_G=\om_{\Gr_H}\boxtimes\rr_G$. This sheaf is a $!$-pure
       ring object and by the K\"unneth formula we have an isomorphism
       $H^\hdot_{H(O)\times G(O)}(pr_2^!\rr_G)\cong H^\BM_{H(O)}(\Gr_H) \o H^\hdot_{G(O)}(\rr_G)$.
       One has an algebra
       isomorphism $H^\BM_{H(O)}(\Gr_H)\cong \C[J_\dh ]$, resp.
         $H^\hdot_{G(O)}(\rr_G)\cong \k[\WH]:=\k[T^*_\psi(\dg/U_\dg)]$, by \cite{BF}.
                             Observe further that one has 
         \beq{whJ}
           \WH(J_{\dh\times\dg})\ =\ 
           J_{\dh\times\dg}\times_{\fcc_{\dh\times\dg}}^{pr_2^*J_\dg}  (pr_2^*\WH)\ \cong\ 
           J_\dh\times \WH.
           \eeq
         We deduce that  $H^\hdot_{H(O)\times G(O)}(pr_2^!\rr_G)\cong
         \k[J_\dh\times \WH]\cong\k[\WH(J_{\dh\times\dg})]$.
         This implies  isomorphism \eqref{vin2}   in the case of  the morphism
         $pr_2: H\times G\to G$.

        Thus, we are reduced to the case where our
        morphism $\rho: H\to G$
        is an imbedding. In this case
           isomorphism \eqref{vin2} is a direct consequence of 
            \eqref{cor m2} for $\ce=\rr_G$.
            \smallskip
            
\textit{Step 2:\,  Proof of part (ii).}\en    To prove isomorphism \eqref{vin1}, we compute
      \begin{align*}
     \HF(\rho^\dagger\rr_G) \   \cong \ H^\hdot_{H(O)}(\rr_{H}\so\rho^\dagger\rr_G)&\ \cong\ 
   H^\hdot_{{H(O)}}(\vrho^\dagger(\rr_{H\times G}))\\
                                                                                   &\ \cong\
                                                                                     \C[J_\dh \times_{\fcc_\dh}^{\vrho^*J_{\dh\times\dg}} \vrho^*(\WW_{\dh\times\dg})]\ccong\k[\WHG],
    \end{align*}
    where the third isomorphism follows from \eqref{vin2} and the last isomorphism is \eqref{vin0}.

     The algebras $\C[J_\dh ]$ and  $\k[\WH]$ are  finitely generated
     integrally closed domain.
         We deduce that so is 
         the algebra $\k[\WH(J_{\dh\times\dg})]\cong \k[J_\dh]\otimes \k[\WH]$, by \eqref{whJ}.
         To prove that the         algebra
    $\HF_\shh(\rho^\da\rr_G)$ is finitely generated              we consider the following
      commutative diagram  of closed imbeddings:
    \[
      \xymatrix{
        (\Gr_\shh)^{T_\shh}   \ar@{=}[r]&  \BX_*(T_\shh)     \ar[r]^<>(0.5){i}\ar[d]^<>(0.5){\veps_\shh}
        &\BX_*(T)\ar[d]^<>(0.5){\veps_{G}}\\
     &   \Gr_\shh\ar[r]^<>(0.5){\rho} & \Gr_{G}&
      }
    \]

   By Theorem \ref{fin}(i),
     it suffices to show that the algebra
    $H^\hdot_{T_H}(\veps_\shh^\da\rho^\da\rr_{G})_+$ is finitely generated.
            Using that $\veps_\shh^\da\rho^\da=i^\da\veps_{G}^\da$ from \eqref{gkd}, we deduce
    \[
      H^\hdot_{T_H}(\veps_\shh^\da\rho^\da\rr_{G})_+=H^\hdot_{T_H}(i^\da\veps_{G}^\da\rr_{G})_+=
      \oplus_{\la\in \BX(T_\shh)_+}\,\C[T^*(\dg/{U_\dg})]_\la,
    \]
    where $\BX(T_\shh)_+$ is the dominant Weyl chamber in $\BX(T_\shh)$.
        Since $\BX(T_\shh)_+$ is a finitely generated semigroup
    it follows from \cite{Bro} (and the proof in \cite{GR})
    that the algebra
    $\oplus_{\la\in\BX(T_\shh)_+}\,\C[T^*(\dg/{U_\dg})]_\la$
    is finitely generated, as desired.
     \qed \smallskip

\textit{Step 3:\,  Proof of part (i).}\en 
 Write $S\dgl=\sg$. 
 It is clear that there are isomorphisms
 \[
   (\td)\lo_{S\dgl}(S\dgl\otimes\C[\dg])
   \cong(\td)\o_{S\dgl}(S\dgl\otimes\C[\dg])\cong (\td)\o\C[\dg]\en\text{and}\en
   (M\o\C[\dg])^\dg\cong M,
 \]
 for any $\dg$-representation $M$.
 Also,  $\rho^\dagger\rr_G$ is $!$-pure, so $\Phi(\rho^\dagger\rr_G)$ is formal
 so from Step(2)  we obtain an isomorphism
 $\Phi_H(\rho^\dagger\rr_G)\cong\k[\WHG]$. Using that
 $\Phi_G(\rr_G)=S\dgl\otimes\C[\dg]$, by \cite[Sect. 5]{BFN2},  from the above isomorphisms
 we find
    \begin{align}
      \Phi_H(\rho^\dagger\rr_G)&\cong
      \C[\WHG]   \cong (\C[\WHG]\o\C[\dg])^\dg\label{phir}\\
      &\cong      (\C[\WHG]\lo_{S\dgl} (S\dgl\otimes\C[\dg]))^\dg
        \cong\big(\C[\WHG]\lo_{S\dgl}\Phi_G(\rr_G)\big)^\dg.
        \nonumber
      \end{align}

       For $V\in \Rep\dg$, let $Fr(V)=S\dgl\o V$, a free $S\dgl$-module
       equipped with the diagonal $\dg$-action. These objects 
are compact generators
        of the category $D_{\text{perf}}^\dh(S\dgl)$ and the functor
$F:=\Phi_H\ccirc\rho^\dagger\ccirc \Phi_G\inv$
sends $Fr(V)$ to a finitely generated $\dh$-equivariant
$\sym\dhl$-module (with zero differential).
Furthermore, this functor is continuous.
It follows easily that $F$  is isomorphic
to the functor  $(B\lo_{S\dgl}(\td))^\dg$, where
$B=F(S\dgl\otimes\C[\dg])$ is an
$\dh\times\dg$-equivariant dg $(S\dhl, S\dgl)$-bimodule.
Hence, using \eqref{phir} we deduce isomorphisms
\begin{align*}
  \C[\WHG]   \cong \Phi_H(\rho^\dagger\rr_G)
&\cong
 \big(\C[\WHG]\lo_{S\dgl}\,\Phi_G\inv(S\dgl\otimes\C[\dg])\big)^\dg\\
 & =F(S\dgl\otimes\C[\dg])\cong
  \big(B\lo_{S\dgl}(S\dgl\otimes\C[\dg])\big)^\dg\cong (B\o \C[\dg])^\dg\cong B.
 \end{align*}
We conclude that the functor $F$  is isomorphic
to the functor $(\C[\WHG] \lo_{S\dgl}(\td))^\dg$, and (i) follows.
\qed \smallskip
            
\textit{Step 4:\,  Proof of part (iii).}\en
 From   \eqref{phir}    we deduce the following chain of  isomorphisms         
                    \begin{align*}
                      (\C[\WW_{{G_1}\to {G_2}}]\,\lo_{S{\check{\mathfrak g}_2}}\,\C[\WW_{G_2\to G_3}])^{\check{G}_2}
                      &\cong         \big(\C[\WW_{{G_1}\to {G_2}}]\,
            \lo_{S{\check{\mathfrak g}_2}}\,\Phi_{G_2}(\rho_{23}^\dagger\rr_{G_3})\big)^{\check{G}_2}\\
     &\cong\Phi_{G_1}(\rho_{12}^\dagger\vrho_{23}^\dagger\rr_{G_3})
            \cong \Phi_{G_1}(\rho_{13}^\dagger\rr_{G_3})\\
&\cong(\C[\WW_{{G_1}\to {G_3}}]\,\lo_{S\dgl_3}\, \Phi_{G_3}(\rr_{G_3}))^{\dg_3} \\          
                      &\cong\big(\C[\WW_{{G_1}\to {G_3}}]\,\lo_{S\dgl_3}\, (\S\dgl_3\o\k[\dg_3)\big)^{\dg_3}\
                        \cong\C[\WW_{{G_1}\to {G_3}}].
          \end{align*}
This proves the isomorphism   $\th_{\text{alg}}$ in (iii).

 Next we observe that taking ${\dg_3}$-invariants, resp.
          ${\check{G}_2}$-invariants, is an exact functor on $D_{\text{perf}}^{\check{G}_2}(S\dgl_3)$,
          resp. $D_{\text{perf}}^{\check{G}_2}(S{\check{\mathfrak g}_2})$.
          It follows that for any $M\in D_{\text{perf}}^{\check{G}_2}(S\dgl_3)$ there are canonical isomorphisms
     \begin{align*}
       (\C[\WW_{{G_1}\to {G_2}}]\,\lo_{S{\check{\mathfrak g}_2}}\,(\C[\WW_{G_2\to G_3}]\,\lo_{S\dgl_3}\,M)^{\dg_3})^{\check{G}_2}    &\cong
         (\C[\WW_{{G_1}\to {G_2}}]\,\lo_{S{\check{\mathfrak g}_2}}\,\C[\WW_{G_2\to G_3}]\,\lo_{S\dgl_3}\,M)^{{\dg_3}\times{\check{G}_2}}\\                                                                       &\cong ((\C[\WW_{G_1\to {G_2}}]\,\lo_{S{\check{\mathfrak g}_2}}\,\C[\WW_{G_2\to G_3}])^{\check{G}_2}\,\lo_{S\dgl_3}\,M)^{\dg_3}.
     \end{align*}

 This completes the proof of (iii), hence the proof of the theorem.
   \qed

\subsection{Proof of Theorem \ref{c thm}}\label{c pf}
It was proved in \cite{BFN1}  that
$H^{\text{odd}}_{\gc(O)}(\cz_{\bm, G})=0$ and   ${\mathcal C}_{\bm,\gc}$
is an irreducible normal variety. Moreover, the moment map
 ${\mathcal C}_{\bm,\gc}\to\fcc_\dg$ for the Hamiltonian $J_\dg$-action is a
 flat morphism with Lagrangian fibers and this morphism
 is a $J_\dg$-torsor
 away from a certain discriminant divisor in $\fcc=\fcc_\dg$.
 Further, the formula in \cite[Theorem 4.1]{BFN1} shows that the algebra
$H^\hdot_{T(O)}(\cz_{\bm, T})_+$ is  finitely generated and
it is immediate from the definition of the scheme $\cz_{\bm, G}$ that the fibers of the map
 $q_G: \cz_{\bm, G}\to \Gr_G$  are (infinite dimesional) affine spaces.
 This implies that the the cohomology of the
 $!$-stalk of $\aa_{\bm, G}$ at any $T$-fixed point $\la\in \BX\sset \Gr_G$
 is a pure,  rank one free $\sym\t$-module
 (note that $*$-stalks of  $\aa_{\bm,\gc}$ are not well defined).
 These results have been extended 
 in \cite{BDFRT} to arbitrary symplectic representations which satisfy the
 anomaly cancelation condition.
Thus, part (1)  of the theorem now follows from Proposition  \ref{bm pure}.

Next, we  prove (3).
As has been mentioned in the proof of Theorem \ref{fin}, see Sect. \ref{filt state},
the morphism $\WH({\mathcal C}_{\bm,\gc})\to \greg$
is \'etale locally isomorphic to the
morphism $\greg\times_{\fcc} {\mathcal C}_{\bm,\gc}\to\greg$.
  Hence, using  \eqref{gwd} it  follows   by \'etale descent that  a regular function
  on $\WH({\mathcal C}_{\bm,\gc})$ may be identified
  with a function $f\in\C[\WH]$ such that
  the rational function $\boldsymbol{\tau}_{\WH,\rho}^*(f)$
  is a regular function on $\WH$. Thus, we have $\C[\WH({\mathcal C}_{\bm,\gc})]\cong
  \C[\dc{(\WH)}]$, as desired.

  We now prove (2) by
  adapting the strategy   used by Bellamy \cite{Bel} in his proof that any
  Coulomb branch has
  symplectic singularities. Specifically, we construct  a variety
  $Y$ with symplectic singularities and a generically
  \'etale morphism $\bm^!_G\to Y$. 
  To this end, recall 
  the ring object  $\aa_{\bm,G}$  for our Coulomb branch
  and
  let $\ccc:={\mathcal C}_{\bm, T}$ be the Coulomb branch of $\bm$, viewed as a representation
  of the maximal torus $T$.
  By \cite{BFN2}, one has $\k[\ccc]\cong H^\hdot_{T(O)}(\veps^\dagger\aa_{\bm,G})$,
  where 
  $\veps: \Gr_T\into \Gr_G$ is the imbedding. Hence, the first isomorphism
  in \eqref{rgk} applied to $\cf=\aa_{\bm,G}$ yields an algebra isomorphism
  $H^\hdot_{T(O)}(\veps^\dagger(\rr_G\so\aa_{\bm,G}))\cong
  (\k[T^*(\dg/U_\dg)]\o_{\k[\dtl^*]} \k[\ccc])^\dt$; furthermore, the algebra
  on the right
  is finitely generated since $\k[T^*(\dg/U_\dg)]$ and $\k[\ccc]$
  are finitely generated algebras, by \cite{GR} and \cite{BFN1}, respectively.
  We conclude that
  $\big({\overline{T^*(\dg/U_\dg)} \,\times_{\dtl^*}\,\ccc}\big)/\!/\dt:=
  \Spec\big((\k[T^*(\dg/U_\dg)]\o_{\k[\dtl^*]} \k[\ccc])^\dt\big)$
  is a scheme of finite type and there is an isomorphism
  \[\Spec H^\hdot_{T(O)}(\veps^\dagger(\rr_G\so\aa_{\bm,G}))\cong    
    \big({\overline{T^*(\dg/U_\dg)} \,\times_{\dtl^*}\,\ccc}\big)/\!/\dt.
  \]

  By adjunction, we have a canonical map
  $H^\hdot_{T(O)}(\veps^\dagger(\rr_G\so\aa_{\bm,G}))\to
  H^\hdot_{T(O)}(\rr_G\so\aa_{\bm,G})$. Taking $W$-invariants yields a
  map $H^\hdot_{T(O)}(\veps^\dagger(\rr_G\so\aa_{\bm,G}))^W\to
  H^\hdot_{G(O)}(\rr_G\so\aa_{\bm,G})$. This map is an  injective algebra homomorphism
  by  Lemma \ref{ring} of the next section. Furthermore,
  the induced map $\Spec H^\hdot_{G(O)}(\rr_G\so\aa_{\bm,G})\to
  \big(\Spec H^\hdot_{T(O)}(\veps^\dagger(\rr_G\so\aa_{\bm,G}))\big)/W$
  is a birational isomorphism, by the
  Localization theorem.
  Thus, thanks to \cite[Lemma 2.1]{Bel} and \cite[Proposition 2.4]{Be}, 
  to prove that $\bm^!_G$ has symplectic singularities it suffices
  to show  that the scheme
  $\Spec H^\hdot_{T(O)}(\veps^\dagger(\rr_G\so\aa_{\bm,G}))/W$ 
 is an algebraic variety with   symplectic singularities.

To this end, we consider
 morphisms $T^*(\dg/U_\dg)\to \wt{\dgl^*}\to\dgl^*$,
  where the second morphism is the Grothendieck-Springer map
  and the first morphism is a $\dt$-torsor.
  Let $\wt{\greg}$, resp. $\cx$, be the preimage of the open set $\greg\sset\dgl^*$
  of regular elements
  in $\wt{\dgl^*}$, resp. $T^*(\dg/U_\dg)$.
  It is known that there is an isomorphism $\wt{\greg}\iso \dtl^*\times_\fcc\,\greg$
  and   the complement of $\cx$ in $\overline{T^*(\dg/U_\dg)}$
  has codimension $\geq 2$. It follows that
  $\wt{\greg}$, and hence $\cx$, is a smooth quasi-affine  variety
  and    one has
  $\overline{T^*(\dg/U_\dg)} =\bar\cx$.

    Let $\mu_{\bar\cx}: \bar\cx\to\dtl^*$,
  resp. $\mu_{\cx}:  \cx\to\dtl^*$  and $\mu_\ccc: \ccc\to\dtl^*$,
  be the moment map for the respective
  Hamiltonian action of $\dt$.
  It is known that the morphisms $\mu_{\bar\cx}$ and $\mu_\ccc$ are flat, 
  see \cite{GR}, \cite{BFN1}, and $\mu_\cx$
  is smooth.
  Thus, $\cx\times_{\dtl^*}\ccc$ is a $\dt$-scheme flat over $\dtl^*$.
  Associated with this scheme
  and the $\dt$-torsor $\cx\to \wt{\greg}$,
  there is a scheme $\cx\times_{\dtl^*}^\dt\ccc$  over  $\wt{\greg}$ that fits into 
  a cartesian square in the following diagram,
  were the maps indicated by `$\dt$' are $\dt$-torsors:
  \[
    \xymatrix{
      \ccc\ar[d]_<>(0.5){\mu_\ccc}&     \cx\times_{\dtl^*}\ccc
      \ar[l]_<>(0.5){pr_2}\ar[d]_<>(0.5){pr_1}\ar[r]^<>(0.5){\dt}\ar@{}[dr]|{\square}
      & \cx\times_{\dtl^*}^\dt\ccc 
      \ar[d]_<>(0.5){p}&\\
 \t&     \cx\ar[l]_<>(0.5){\mu_\cx}\ar[r]^<>(0.5){\dt}& \wt{\greg} \ar[r]^<>(0.5){\cong} &
      \greg\times_\fcc\dtl^*
    }
    \]

  Write $\ccr$ for the regular locus of
  $\ccc$. The morphism $\mu_\cx$ being  smooth and quasi-affine,
  so is the second projection  $\cx\times_{\dtl^*}\ccr\to \ccr$,
  by smooth base change.
    We deduce that $\cx\times_{\dtl^*}\ccr$ is a smooth quasi-affine
  variety.
  Observe further  that
   the differential of the map
   $\mu_{\cx}- \mu_\ccc: \cx\times\ccc\to \dtl^*$
   is surjective at any point of  $\cx\times_{\dtl^*}\ccr$,
   since  the (anti-) diagonal action of $\dt$ on  $\cx\times_{\dtl^*}\ccc$
  is free. Hence,
   one may view $\cx\times_{\dtl^*}^\dt\ccr$  as  a Hamiltonian reduction of
  $\cx\times\ccr$ with respect to the
  anti-diagonal $\dt$-action on $\cx\times\ccr$. Thus, the symplectic form
  on $\cx\times\ccr$ induces a  symplectic form on
  $\cx\times_{\dtl^*}^\dt\ccr$.

  The composite $\cx\times_{\dtl^*} \ccr \into \bar\cx\times_{\dtl^*} \ccc
  \to (\bar\cx\,\times_{\dtl^*}\,\ccc)/\!/\dt$ factors through a map
  $\cx\times_{\dtl^*}^\dt\ccr\to (\bar\cx\,\times_{\dtl^*}\,\ccc)/\!/\dt$,
  by the universal property of geometric quotients.
  It is clear that the induced pullback of the regular functions
  gives an algebra map which is equal to the composition of
   the following chain of algebra isomorphisms:
    \[
    \k\big[(\bar\cx\,\times_{\dtl^*}\,\ccc)/\!/\dt\big]\,\iso\,
    \big(\k[\bar\cx]\o_{\k[\dtl^*]}\k[\ccc]\big)^\dt\,\iso\,
    (\k[\cx]\o_{\k[\dtl^*]}\k[\ccr])^\dt\,\iso\,
    \k[\cx\times_{\dtl^*}\ccr]^\dt\,\iso\,
    \k[\cx\times_{\dtl^*}^\dt\ccr].
    \]
   Here, in the second isomorphism we have used that
   $\k[\cx]=\k[\bar\cx]$ and  $\k[\ccr]=\k[\ccc]$, since the variety $\ccc$ is irreducible and normal,
 \cite{BFN1},   so
  the complement of $\ccr$ in $\ccc$ has codimension ${\geq 2}$.
    The algebra isomorphisms above imply that the map
 $\cx\times_{\dtl^*}^\dt\ccr\to (\bar\cx\,\times_{\dtl^*}\,\ccc)/\!/\dt$
 induces an isomorphism $\overline{\cx\times_{\dtl^*}^\dt\ccr}\iso 
 (\bar\cx\,\times_{\dtl^*}\,\ccc)/\!/\dt$.
  It follows that $\cx\times_{\dtl^*}^\dt\ccr$ is a smooth Zariski dense subvariety of
 the scheme $(\bar\cx\,\times_{\dtl^*}\,\ccc)/\!/\dt$. Since the latter
 scheme has finite type, we deduce that $(\bar\cx\,\times_{\dtl^*}\,\ccc)/\!/\dt$
 is a reduced and irreducible, normal variety; moreover,
 the complement of $\cx\times_{\dtl^*}^\dt\ccr$
  has codimension ${\geq 2}$ in $(\bar\cx \,\times_{\dtl^*}\,\ccc)/\!/\dt$.

  To complete the proof 
  it remains to show, thanks to \cite[Theorem 6]{Nam},
  that $(\bar\cx \,\times_{\dtl^*}\,\ccc)/\!/\dt$ is a Cohen-Macaulay
  variety.
  \footnote{I am grateful to Tom Gannon for informing me about this result of Namikawa.}
  To this end, we use that
  the variety $\overline{T^*(\dg/U_\dg)}=\bar\cx$, resp. $\ccc$, has symplectic
  singularities, by \cite{Ga}, resp. \cite{Bel}.
  Hence,  $\bar\cx\times \ccc$  has rational singularities.
  Therefore,
   $(\bar\cx\times \ccc)/\!/\dt$  has rational singularities,  by \cite{Bou};
      in particular, $(\bar\cx\times \ccc)/\!/\dt$ is a  Cohen-Macaulay variety.
      From this, we will deduce that $(\bar\cx \,\times_{\dtl^*}\,\ccc)/\!/\dt$ is  Cohen-Macaulay
      as follows.
   Let  $\bar\mu: (\bar\cx\times \ccc)/\!/\dt\to \dtl^*\times\dtl^*$
    be the morphism 
    induced by the map $\mu_{\bar\cx}\times\mu_\ccc$, and let
    ${\mathcal I}_\De\subset \k[\dtl^*\times\dtl^*]$
  be the ideal of  the diagonal  $\dtl^*_\De\sset \dtl^*\times\dtl^*$.
  Then, by definition one has
  $\k[\bar\mu\inv(\dtl^*_\De)]=
  \k[\bar\cx\times \ccc]^\dt\big/\k[\bar\cx\times \ccc]^\dt\cdot
    \bar\mu^*({\mathcal I}_\De)$, where
     $\bar\mu\inv(\dtl^*_\De)$  is the {\em scheme-theoretic}
    preimage of   $\dtl^*_\De$.
      On the other hand, by Hilbert
    there is an algebra isomorphism
  \[     
   (\k[\bar\cx\times \ccc]/I_\De)^\dt\ccong
    \k[\bar\cx\times \ccc]^\dt\big/\k[\bar\cx\times \ccc]^\dt\cdot
    \bar\mu^*({\mathcal I}_\De),
    \]
  where $I_\De$ denotes
 the ideal of $\k[\bar\cx\times \ccc]$
  generated by $(\mu_{\bar\cx}\times\mu_\ccc)^*({\mathcal I}_\De)$.
  We obtain an isomorphism
  $(\bar\cx\,\times_{\dtl^*}\,\ccc)/\!/\dt=
  \Spec\big((\k[\bar\cx\times \ccc]/I_\De)^\dt\big)\ccong
        \bar\mu\inv(\dtl^*_\De)$.    
      Now, $\k[\bar\cx\times \ccc]^\dt$ is a flat
      $\k[\dtl^*\times\dtl^*]$-module,
      since it is a direct summand of the $\k[\dtl^*\times\dtl^*]$-module
      $\k[\bar\cx\times \ccc]$, which we know is flat.
      Hence, the equations of the diagonal
      $\dtl^*_\De\sset \dtl^*\times\dtl^*$ form a regular sequence in
      the Cohen-Macaulay algebra
      $\k[\bar\cx\times \ccc]^\dt=\k[(\bar\cx\,\times\ccc)/\!/\dt]$.
      This implies that the scheme $\bar\mu\inv(\dtl^*_\De)\cong(\bar\cx\,\times_{\dtl^*}\,\ccc)/\!/\dt$ is  also
      Cohen-Macaulay, and we are done.\qed
  
   \section{Finite generation}
         \label{gr pf}
       Recall that $W$ is the Weyl group and
       $\dom$ is the dominant Weyl chamber in   $\BX=\BX_*(\tv)$.
       For $\la\in \dom$, let $W_\la\sset W$ be the stabilizer, resp.
       $W.\la$ the $W$-orbit, of $\la$.
       Let $\Si\sset \BX$ be the set of simple coroots and
              `$\leq$'  the standard partial order on $\dom$
              such that $\la\leq\mu$ if and only if $\mu-\la$ is a sum of simple
              coroots.

       \subsection{Algebraic preliminaries}
       \label{alg pre}
       We begin with an elementary observation:
\beq{WW}
W_{\la+\mu}=W_\la\cap W_\mu\qquad\forall \la,\mu\in\dom.
\eeq
To see this, consider the height function
$ht: \BX\to\Z,\, \sum_{\al\in\Si}\, n_\al  \al\mto \sum_{\al\in\Si} n_\al$.
For any
$\nu\in \dom$ and  $w\in W\sminus W_\nu$, one has $ht(w(\nu))<
ht(\nu)$.
Therefore, for $\la,\mu\in\dom$, we have
\begin{align}
w\not\in W_{\la+\mu}\en&\Leftrightarrow\en
ht(w(\la+\mu)) <
                         ht(\la+\mu)\en\label{lamu}\\
  &\Leftrightarrow\en
  ht(w(\la)) < ht(\la)\en\text{or}\en ht(w(\mu))<ht(\mu)
\en\Leftrightarrow\en
  w\not\in W_{\la} \en\text{or}\en w\not\in W_{\mu}.\hfill\qquad\qed
  \nonumber
\end{align}

Let $B=B^{\text{even}} \oplus B^{\text{odd}}=
\oplus_{\la\in \dom}\,B_\la^{\text{even}} \oplus B^{\text{odd}}_\la$ be an $\dom$-graded supercommutative
algebra 
  and  $M=\oplus_{\la\in \dom}\,M_\la$
  an  $\dom$-graded
  $B$-module. 
  We say that $B$, resp. $M$, is {\em partially equivariant} if for each $\la\in \dom$
  there is a linear $W_\la$-action in $B_\la$, resp. $M_\la$,
  such that $w(b_\la b_\mu)=w(b_\la) w(b_\mu)$,
  resp. $w(b_\la m_\mu)=w(b_\la) w(m_\mu)$,  for all $w\in W_{\la+\mu},\,
  b_\la\in B_\la,\, b_\mu \in B_\mu$, resp. $m_\mu\in M_\mu$.
  These equations make sense thanks \eqref{WW}.
  We define  
  \[B\nana:=\oplus_{\la\in\dom}\, (B_\la)^{W_\la},\en\text{resp.}\en
    M\nana:=\oplus_{\la\in\dom}\, (M_\la)^{W_\la}.
    \]
       It follows from \eqref{WW} that one has an inclusion
       $(B_\la)^{W_\la}\,(M_\mu)^{W_\mu}\sset (M_{\la+\mu})^{W_{\la+\mu}}$.
       We deduce (taking $M=B$) that $B\nana$ is  a subalgebra of $B$,
       resp. $M\nana$ is an $B\nana$-submodule of $M$.

  \begin{prop}\label{sisi} 
If the algebra $B$ is  finitely generated then $B\nana$
  is a  finitely generated  algebra and $B$    is a finitely generated 
  $B\nana$-module.
\end{prop}

    The proof of  the proposision mimics the classic argument due to Hilbert as follows.
    Let $BM\nana\sset M$ be the $B$-submodule generated by $M\nana$.
    We start with 
  \begin{claim}\label{claim4} If $m_i\in (M_{\mu_i})^{W_{\mu_i}},\,i=1,\ldots,n$, generate
  $B M\nana$ as a $B$-module then these elements
  generate $M\nana$ as an  $B\nana$-module.
\end{claim}

\begin{proof}[Proof of Claim]
  Let $\nu\in\dom$  and $\si_\nu: M\to M^{W_\nu}$ be the averaging operator
  $m\mto \frac{1}{|W_\nu|}\sum_{w\in W_\nu}\, w(m)$.
  Since the elements $m_i$ generate $BM\nana$,
 any element $m_\nu \in M_\nu \cap (B M\nana)$
  can be written in the form 
  $m_\nu:=b_1 m_1+\ldots b_n m_n\in M_\nu$ for some
  $\la_1,\ldots,\la_n\in\dom$
  such that $\la_i+\mu_i=\nu$  for all $i$ and some $b_i\in B_{\la_i},\,i=1,\ldots,n$.
  Observe that \eqref{WW} implies that  for all $i$ the group $W_\nu$ is
 contained in   $W_{\mu_i}$, so we have $\si_\nu(m_i)=m_i$ for all $i$.
   In the special case where  $m_\nu\in (M_\nu)^{W_\nu}=M_\nu\cap M\nana$, we find
 \[
  m_\nu= \si_\nu(m_\nu)=
   \si_\nu(b_1) m_1+\ldots \si_\nu(b_n) m_n\in B\nana m_1+\ldots+B\nana m_n,
 \]
  and our claim follows.
\end{proof}

\begin{proof}[Proof of Proposition \ref{sisi}]
 Assume that the algebra $B$ is  finitely generated.
 Let $I:=\oplus_{\la\in\dom\sminus\{0\}}\, B_\la$,
  resp. $B\nana\cap I=I\nana=\oplus_{\la\in\dom\sminus\{0\}}\, (B_\la)^{W_\la}$,
  be  the augmentation ideal of $B$, resp. $B\nana$.
     To prove that  $B\nana$ is a finitely generated algebra, it suffices
     to show that $I\nana$
     is  a finitely generated ideal (this statement holds for supercommutative algebras
     with the same proof as in the case of commutative algebras).
Since $B$ is   noetherian  the ideal $B.I\nana$, of $B$, is
generated by finitely many elements $m_1,\ldots,m_n$,
which we may assume  without loss of generality  
to be elements of $I\nana$. Applying 
Claim \ref{claim4}  in the case where $M=I$
we deduce that the elements $m_1,\ldots,m_n$
  generate the ideal $I\nana$,
  as desired.

  It remains to prove that $B$ is finitely generated as a $B\nana$-module.
  To this end, let $\C W$ be the group algebra of $W$ viewed as the
  regular representation of
  $W$. We make 
    $\C W \o B=\oplus_{\la\in\dom}\, 
  \C W \o B_\la$ a partially equivariant $B$-module by letting
  $B$ act on $\C W \o B$ by $_1: w\o b_2\mto w\o b_1b_2$ and
  letting $W_\la$ act in the direct summand $\C W \o B_\la$ by
  $w': w\o b\mto w(w')\inv\o w'(b)$.
  Thus, $\C W \o B$ is a free $B$-module of rank $|W|$
  and $(\C W \o B)\nana :=\oplus_{\la\in\dom}\, (\C W \o B_\la)^{W_\la}$  is 
  a $B\nana$-submodule of that module.
  The $B$-submodule of $\C W \o B$ generated by
  $(\C W \o B)\nana$  is finitely generated,
  since $B$ is noetherian. Therefore, $(\C W \o B)\nana$
  is a finitely generated $B\nana$-module, by 
  Claim \ref{claim4}. 
  The vector space $\oplus_{\la\in\dom}\, (\C W_\la\o B_\la)^{W_\la}
  \cong \oplus_{\la\in\dom}\, B_\la$ is 
  a  $B\nana$-submodule of $(\C W \o B_+)\nana$ isomorphic to $B$.
  Thus, this  submodule is finitely generated, since we proved that
  $B\nana$ is finitely generated, hence noetherian.
  \end{proof}

  \begin{rem} Let $\C\BX_+$ be
    the semigroup algebra of $\BX_+$.
    The $\BX_+$-grading gives a $T$-action on $B$ and this action extends to an 
    action of the  algebraic monoid $\bar T=\Spec(\C\BX_+)$.
    \erem

 Next, let $A= \oplus_{\la\in \BX}\, A_\la$, resp. $M= \oplus_{\la\in \BX}\, M_\la$,
  be an $\BX$-graded algebra, resp.
   $A$-module, 
  equipped with a  $W$-action by algebra, resp. $A$-module, automorphisms
    such that $w(A_\la)\sset A_{w(\la)}$, resp. $w(M_\la)\sset M_{w(\la)}$,
  for all $w\in W$.
  Then, $A_+=\oplus_{\la\in \dom}\, A_\la$, resp. $M_+=\oplus_{\la\in \dom}\, M_\la$, is a
  partially equivariant algebra, resp. $A_+$-module.
  Thus, one has the subalgebra $A_+\nana:=(A_+)\nana\sset A_+$, 
  resp. $A_+\nana$-submodule $M_+\nana:=(M_+)\nana\sset M_+$.

  \begin{cor}\label{tom} Assume that the semigroup $\dom$ is free.
    If $A$ is an integrally closed domain then so is $A_+$; in this case,
    $\Spec A_+ \to \Spec A_+\nana$ is the normalization map.
  \end{cor}
  \begin{proof}\footnote{I am grateful to Tom Gannon for providing this proof.}
    The first statement is a consequence of  the following more general claim:
    let $X$ be a free abelian group with basis $\omega_1,\ldots\omega_\ell$
    and 
   $A=\oplus_{\la\in X} A_\la$  an $X$-graded  integrally closed domain. 
    Then the algebra $A_+:=\oplus_{\la\in X_+} A_\la$, where $X_+$ is the free semigroup
    generated by  $\omega_1,\ldots\omega_\ell$, is  integrally closed.
    Using induction on $\ell$ the claim may be reduced  to the case $\ell=1$.
    Thus, we may (and will) assume that $X=\Z$, so $A=\oplus_{n\in\Z}\, A_n$.
   
    Since $A$  is integrally closed,    it suffices to prove that if $b\in \text{Frac}(A_+)\cap A$
    satisfies an equation $f(b)=0$ for some monic polynomial
    $f(t) = t^d + f_{d-1}t^{d-1} + \ldots + f_0 \in A_+[t]$, then $b\in A_+$.
        We may write
$ b = p + q$  for $ q \in A_+ $ and
$ p = p_{-m} + p_{-m+1}  + \ldots+ p_{-1} $,  where $ p_i \in  A_i$.
If $b \not\in A_+$ we may assume  without loss of generality   that
 $p_{-m}\neq0$.
Expanding  the equation $f(p+q)=0$ we find that
$p^d + dp^{d-1}q + \ldots+ q^d + f_{d-1}p^{d-1} + \ldots+ f_{d-1}q^{d-1} +\ldots+ f_0 = 0$. 
The homogeneous component of the LHS of degree  $-md$
equals  $(p_{-m})^d$. Thus, the equation implies that $ p_{-m}=0$, a contradiction.

To complete the proof, recall that $A_+$ is a finitely generated $A_+\nana$-module,
by Proposition \ref{sisi}. Hence, it remains to show that the inclusion $\text{Frac}(A_+\nana)\into
\text{Frac}(A_+)$ is an equality. 
To this end,
let $a=\sum_{\nu\in S}\, a_\nu\in A_+$, where $a_\nu\in A_\nu$ and $S$ is a finite subset of $\dom$.
Pick  a nonzero element
$q\in A_\la$ such that $\la\in\dom$ is  sufficiently general so that $W_\la=\{1\}$.
By \eqref{WW}, for all $\nu\in S$ we have  $W_{\nu+\la}=\{1\}$.
Therefore, $q$ and all elements $a_\nu\cdot q$ are in $A_+\nana$.
Hence, we have $a=\sum_{\nu\in S}\, \frac{a_\nu \cdot q}{q}\in \text{Frac}(A_+\nana)$.
We conclude that
$A_+\sset \text{Frac}(A_+\nana)$, which implies that $\text{Frac}(A_+\nana)=
\text{Frac}(A_+)$.
\end{proof}

In the above setting,  for each $\la\in\dom$ let
  $(A_\la)^{W_\la} \to A^W$, resp.
  $(M_\la)^{W_\la}\to M^W$, be
    the  symmetrization map  $x\mto \frac{1}{|W|}\sum_{w\in W}\,w(x)$.
It is clear that these maps yield a bijection $\si: A\nana\iso A^W$,
resp. $\si: M\nana\iso M^W$.

    We will need the following simple result that
 has been essentially observed in \cite{BFN1}.
\begin{lem}\label{sisi2}
    For any $a\in (A_\la)^{W_\la}$ and $m\in (M_\mu)^{W_\mu}$, where
  $\la,\mu\in\dom$, one has
  $\si(a m)= \si(a)\si(m)+\sum_\nu\, \si(m_\nu)$,
 for some $\nu\in\dom$ such that $\nu<\la+\mu$ and 
  $m_\nu\in   (M_\nu)^{W_\nu}$.
    \end{lem}
    \begin{proof}   Write
      $W\times W=(W_\la\times W_\mu) \sqcup {\mathscr S}$ where
      ${\mathscr S}=(W\sminus W_\la)\times W\, \bigcup
      \,(W\times (W\sminus W_\mu)$. 
      From the definition of $\si$ one easily finds
\[  
 \si(a)\si(m)-
  \si(am) \ =\frac{1}{|W|^2}\sum_{(w,y)\in {\mathscr S}}\, w(a)\, y(m).
    \]
    The sum on the right is in the image of the map $\si$ since
    $\si(b)\si(m)- \si(bm)\in M^W=\im(\si)$.
        The desired statement now follows from the observation
    that for any $(w,y)\in {\mathscr S}$, by  \eqref{lamu} , one has
    $w(\la)+y(\mu)< \la+\mu$.
  \end{proof}

  For any $\la\in\dom$ let
  $\BX_{\pre\la}:=\bigsqcup_{\{\mu\in\BX_+\,\mid\,\mu\leq \la\}}\, W.\mu$.
      We write  $\nu\pre \la$, resp.  $\nu\pree \la$,
      if $\nu \in \BX_{\pre\la}$, resp. $\nu\in \BX_{\pre\la}\sminus W.\la$.
            Given an $\BX$-graded algebra $A$ with a $W$-action, as above,
      we define an ascending   filtration  $A_{\pre\la},\,\la\in \dom$, on $A$,
      by the formula $A_{\pre\la}=\bigoplus_{\{\mu\in\BX\,\mid\,\mu\pre\la\}}\, A_\mu$.
      The space $A_{\pre\la}$ is clearly $W$-stable. It follows from
      Lemma \ref{sisi2} that for all $\la,\mu\in\dom$
      one has $(A_{\pre\la})^W\cdot (A_{\pre\mu})^W\sset (A_{\pre\la+\mu})^W$;
      moreover, the map $\si$ induces an isomorphism
      $A\nana\iso \oplus_{\la\in\dom}\, (A_{\pre\la})^W/(A_{\pree\la})^W$,
      of $\dom$-graded algebras.      
      This, combined with Proposition \ref{sisi},  yields the following result.

      \begin{cor}\label{Wfin} 
        If the algebra $A_+$ is finitely generated then so is the algebra $A^W$.
        \qed
      \end{cor}
      
       \subsection{Filtration by dominant weights} \label{gr sec}
   In this subsection we will
   exploit an idea from \cite{BFN1} to prove  Proposition \ref{ppp} below,
   which is a key step of  the proof of  
   Proposition  \ref{part}.
      We will use simplified notation
      $\sfh(\td)=H^\hdot_{\gc(O)}(\td)$, resp.
      $\hh(\td)=H^\hdot_{\tv(O)}(\td)$
     and write $\t=\Lie \tv$, resp. $\fc=\C[\t]=\hh(\pt)$ and
      $\sfc=\fc^W=\sfh(\pt)$.

      Recall  the
   imbedding $\veps: \BX=\Gr_\tv\into \Gr=\Gr_G$
   and the restriction functor $\veps^\dagger: D_{\gc(O)}(\Gr)\to D_{\tv(O)}(\Gr_\tv)$.
     For any $\cf\in\sat$  there is  a natural $W$-action on
     $\hh(\veps^\dagger\cf)$ such that $w(s h)=w(s) w(h)$ for all
     $s\in\fc$ and $h\in \hh(\veps^\dagger\cf)$.
     If $\cf$ is a ring object then so is $\veps^\dagger\cf$ and the group $W$ acts on the ring
     $\hh(\veps^\dagger\cf)$ by  ring automorphisms. Thus,
     $\hh(\veps^\dagger\cf)^W$ is a $\sfc$-subalgebra of $\hh(\veps^\dagger\cf)$.

        \begin{lem}\label{ring}  For any  commutative ring object $\cf\in D_{\gc(O)}(\Gr)$ 
          the adjunction $\veps_!: \hh(\veps^\dagger\cf)^W\to\sfh(\cf)$
          is a $\sfc$-algebra homomorphism.
\end{lem}
    \begin{proof} 
The argument below is similar to the one used in the proof
      of \cite[Lemma 5.10]{BFN1}.
      To simplify notation, for $\cf\in\sat$
      we write   $\bar\cf=\text{Obl}^{\gc(O)}_{\tv(O)}\cf\in D_{\tv(O)}(\Gr)$.

      We have  the following natural maps
      $\tv(K)\times_{\tv(O)} \Gr           \xrightarrow{\tilde\veps}
      \gc(K)\times_{\gc(O)}\Gr\xrightarrow{a}\Gr$.
      Let $\tilde a=a\ccirc \tilde\veps$ denote the composite
      and given $\ce\in D_{\tv(O)}(\Gr_\tv)$, resp. $\cf\in\sat$ and $\cg \in D_{\tv(O)}(\Gr)$,
      let $\ce\,\widetilde\boxtimes_{T(O)}\,\cg\in \tv(K)\times_{\tv(O)} \Gr $,
      resp.  $\cg\widetilde\boxtimes_{G(O)}\cf\in\gc(K)\times_{\gc(O)}\Gr$,
        denote the descent of $\ce\boxtimes\cg$, resp. $\cg\boxtimes\cf$.
      One makes $D_{\tv(O)}(\Gr)$  a left, resp. right, module category
      over the monoidal category $D_{\tv(O)}(\Gr_\tv)$, resp.  $D_{G(O)}(\Gr)$, 
      by letting $\ce\in D_{\tv(O)}(\Gr_\tv)$, resp.  $\cf\in \sat$,
      act on $D_{\tv(O)}(\Gr)$ by the formula
      $\ce: \cg\mto\ce\star_\ell\cg:=\tilde a_!(\ce\,\widetilde\boxtimes_{T(O)}\,\cg)$,
      resp.  $\cf: \cg\mto \cg\star_r\cf=a_!(\cg\widetilde\boxtimes_{G(O)}\cf)$.

      For $\ce,\cf\in\sat$, using the natural isomorphism $(\veps^!\bar\ce)\,
      \widetilde\boxtimes_{T(O)}\,\bar\cf
      \cong\tilde\veps^!(\bar\ce\,\widetilde\boxtimes_{G(O)}\,\cf)$
      we obtain a chain of 
      morphisms
      \[(\veps^\dagger\ce)\star_\ell\bar\cf=
        \tilde a_!((\veps^!\bar\ce)\,\widetilde\boxtimes_{T(O)}\,\bar\cf)\iso
        a_!\tilde\veps_!\tilde\veps^!(\bar\ce\,\widetilde\boxtimes_{G(O)}\,\cf) 
        \xrightarrow{} a_!(\bar\ce\,\widetilde\boxtimes_{G(O)}\,\cf)
        =\bar\ce\star_r\cf=\overline{\ce\star\cf}.
      \]
          Restricting    the  map
      $\hh(\veps^\dagger\ce)\o_\fc \hh(\bar\cf)
      \to \hh(\overline{\ce\star\cf})$ induced by the composite morphism
            to $W$-invariants yields a map
      $*_\ell: \hh(\veps^\dagger\ce)^W\o_\sfc \sfh(\cf)\to \sfh(\ce\star\cf)$.
      Tracing through the construction one obtains the following equation:
            \beq{asb}
      h *_\ell  h'=\veps_!(h)\star h',\qquad \forall h\in \hh(\veps^\dagger\ce)^W,\en
      h'\in \sfh(\cf).
      \eeq

      To complete the proof we recall that the unit of the ring object $\cf$
      is a morphism $\bone\to\cf$ subject to a certain compatibility.
           Let $1$      be 
            the image of the unit of the algebra $\sfc$ under the induced
            map $\sfc=\sfh(\bone)\to \sfh(\cf)$.
           We use the same symbol   $1$     for the element
            $1\o 1\in \fc\o_\sfc \sfh(\cf)=\hh(\bar\cf)$.
            It is not difficult to see that  for $h\in \hh(\veps^\dagger\cf)$, in $\hh(\bar\cf)$,
            one has $h*_\ell 1=\veps_!(h)$.
            Thus, using \eqref{asb} we obtain 
            \[\veps_!(h\star h')=(h\star h')*_\ell1=h*_\ell (h'*_\ell 1)=
              h *_\ell \veps_!(h')=\veps_!(h)\star \veps_!(h'),\qquad
              \forall\, h,h'\in \hh(\veps^\dagger\cf)^W.
              \qedhere
              \]
\end{proof}
\medskip

      For $\la\in\dom$, let ${i}_{\leq\la}$, resp. $i_{<\la}$, denote the  closed imbedding
      of $\ogr{\la}$, resp. $\ogr{\la}\sminus \Gr_\la$, into $\Gr$.
 If $\cf\in\sat$ is $!$-pure then it follows from
 Corollary \ref{inj cor} that for all $\la\in\dom$
 the maps $({i}_{\leq\la})_!: \sfh({i}^!_{\leq\la}\cf)\to \sfh(\cf)$
 are injective. Thus, the  images of the maps  $({i}_{\leq\la})_!$
  give an ascending  filtration of  $\sfh(\cf)$.
  From now on, we will identify $\sfh({i}^!_{\leq\la}\cf)$ with its image.

    \begin{prop}\label{ppp} For any  $!$-pure commutative ring object  $\cf\in \sat$
      the filtration $\sfh({i}^!_{\leq\la}\cf),\,\la\in\dom$, on the ring  $\sfh(\cf)$
      is multiplicative and
    there is an isomorphism of $\dom$-graded $\sfc$-algebras:
    \[\oplus_{\la\in\dom}\, \sfh({i}^!_{\leq\la}\cf)/\sfh({i}^!_{<\la}\cf)
      \ccong \big((H(\veps^\dagger\cf))_+\big)\nana.
    \]
  \end{prop}
\smallskip

  To prove the proposition, fix  $\la\in\dom$, identify it with the
  corresponding $T$-fixed point
of $(\Gr_\la)^T$, and let
 $\gc(O)_\la\sset \gc(O)$ be  the stabilizer of that point in $\gc(O)$.
 Then, $P_\la= \gc \cap \gc(O)_\la$ is a parabolic subgroup of $\gc$,
the imbedding $\gc/P_\la\into \gc(O)/\gc(O)_\la\cong \Gr_\la$ is a
 $\gc$-equivariant homotopy
 equivalence and there are
  $\sfc$-algebra isomorphisms
  $\sfh(\Gr_\la)\cong \sfh(\gc/P_\la)\cong \fc^{W_\la}$.

      Recall the partial order $\pre$ on $\dom$ defined at the end of
      Sect. \ref{alg pre}.          The imbedding  $\veps$ restricts
      to an   imbedding
      $\eps_{\pre\la}: \BX_{\pre\la}\into \ogr{\la}$,
      resp. $\eps_{\pree\la}: \BX_{\pree\la}\into \ogr{\la}$,
      with image $(\ogr{\la})^\tv$, resp. $(\ogr{\la}\sminus\Gr_\la)^\tv$.

      Let  $\cll\in D_{\gc(O)}(\ogr{\la})$ be a $!$-pure object and
         $\cll_{<\la}$, resp. $\cll_\la$,   the $!$-restriction of $\cll$
along the closed, resp. open, imbedding
      $\ogr{\la}\sminus\Gr_\la \into \ogr{\la}$, resp. $\Gr_\la\into\ogr{\la}$.
      We have a direct sum decomposition $\hh(\eps_{\pre\la}^!\cll)
      =\oplus_{\nu\pre\la} \, \hh_\nu(\cll)$, where $\hh_\nu(\cll)$ is the
      $\tv$-equivariant cohomology
      of the stalk of $\eps_{\pre\la}^!\cll$ at $\nu$.
         There are natural $\sfc$-linear maps 
\beq{sisa1}
  \hh_\la(\cll)^{W_\la}
  \,\xrightarrow{\cong}\,
  \big(\oplus_{\nu\in W.\la}\, \hh_\nu(\cll)\big)^W
  \into \hh(\veps_{\pre\la}^\dagger\cll)^W
    \xrightarrow{\ (\veps_{\pre\la})_!\ }
  \sfh(\cll),
\eeq
where the first map is the  symmetrization map
$h\mto \{w(h)\in \hh_{w(\la)}(\cll)\}_{w\in W/W_\la}$.
Let  $\si_{\pre\la}$ be the composite of the first two maps in \eqref{sisa1}.

       By
          Lemma \ref{inj claim}
           there is a short exact sequence
 \beq{ij gr}
 0\to \sfh(\cll_{<\la})\xrightarrow{}
 \sfh(\cll)\xrightarrow{\ \text{res}\ }
\sfh(\cll_\la)\to 0
\eeq
We obtain the following chain of  $\sfc$-linear maps
\beq{sisa}
  \hh_\la(\cll)^{W_\la}
  \xrightarrow[\cong]{\ \bar\si \ }
  \hh(\eps_{\pre\la}^\dagger\cll)^W/\hh(\eps_{\pree\la}^\dagger\cll)^W
  \xrightarrow{(\bar\eps_{\pre\la})_!}
  \sfh(\cll)/\sfh(\cll_{<\la}) 
  \xrightarrow[\cong]{\ \overline{\text{res}}\ } \sfh(\cll_\la).
    \eeq
    where $(\bar\eps_{\pre\la})_!$ is the map induced by $\eps_{\pre\la})_!$
    and
   $\bar\si$ is  the composite of  $\si_{\pre\la}$ 
  and the quotient map. 
 The natural action of $\sfh(\Gr_\la)\cong\fc^{W_\la}$ gives $\sfh(\cll_\la)$
 the structure of an $\fc^{W_\la}$-module.

 Let  $Q=\text{Frac}(\fc)$
  be the  fraction field 
of $\fc$.
      Given an $\fc$-module, resp. $\sfc$-module,
      $M$,
      we put   $QM=Q\o_\fc M$, resp.
      $Q^W M:=Q^W\o_{\sfc} M$
      and $Q^{W_\la}M=Q^{W_\la} \o_\sfc M$.
      Similarly, given  a morphism $f: M\to N$ of $\fc$-modules,
      we have a morphism $Qf: QM\to QN$ of $Q$-modules, etc.
  Bellow, all localization maps like $M\to QM$ are injective, thanks to
 the $!$-purity of $\cll$.

 We extend the map $\si_{\pre\la}$, resp. $(\eps_{\pre\la})_!$,
 $(\bar\eps_{\pre\la})_!$ and $\overline{\text{res}}$,
 to a $Q^W$-linear map $Q^W\si_{\pre\la}$, resp. $Q^W(\eps_{\pre\la})_!$,
 $Q^W(\bar\eps_{\pre\la})_!$ and $Q^W\overline{\text{res}}$.
  The composite 
$Q^W(\overline{\text{res}}\ccirc (\bar\eps_{\pre\la})_!):
Q^W\hh(\eps_{\pre\la}^\dagger\cll)^W/Q^W\hh(\eps_{\pree\la}^\dagger\cll)^W
\to Q^W\sfh(\cll_\la)$
  is an  isomorphism of $Q^W$-modules, by  the Localization theorem.
 Let $q_\la$ be the composite map 
in   \eqref{sisa}. This map is
$\fc^{W_\la}$-linear and it follows from the above that
the map  $Q^{W_\la} q_\la: Q^{W_\la}\hh_\la(\cll)^{W_\la}\to
  Q^{W_\la}\sfh(\cll_\la)$  is an isomorphism
   of $Q^{W_\la}$-modules.

 Let  $\De\pos\in\BX^*(T)$  be the set   of positive roots associated 
 with our choice $\Si\sset \BX_*(T)$ of the set of simple coroots.
 Following \cite{BFN1}, for $\la\in\BX_+$ we define
\[{\mathfrak d}_\la:=\prod_{\al\in\De\pos}\ \al^{\langle \la,\al\rangle} \in \fc^{W_\la}.
  \]

  \begin{lem}\label{i claim}  The image of the map
       $\,\hh_\la(\cll)^{W_\la} \to
    Q\sfh(\cll_\la),\
    h\mto \mbox{$\frac{1}{{\mathfrak d}_\la}$}(Q\bar\si(q_\la(h))
 $,
    equals the
    subspace $\sfh(\cll_\la)\subset Q\sfh(\cll_\la)$; furthermore, the resulting
    map
yields an   isomorphism of $\fc^{W_\la}$-modules:
    \[\mbox{$\frac{1}{{\mathfrak d}_\la}$}\bar\si\ccirc q_\la:\, \hh_\la(\cll)^{W_\la} \iso
     \sfh(\cll_\la).
  \]
  \end{lem}

  \begin{proof}
      Since $\cll$ is $!$-pure, in $D_{\gc(O)}(\Gr_\la)$ there is an isomorphism
       $\cll_\la\cong \oplus_m\, {\mathcal H}^m\cll_\la[-m]$; moreover,
      ${\mathcal H}^m\cll_\la$
      are geometrically constant sheaves on $\Gr_\la$.
  Therefore, it suffices to prove the lemma  in the case where $\cll_\la$ is a constant sheaf.
    This is, essentially, the content of the proof
of
\cite[Proposition 6.2]{BFN1}.  In that proof,   ${\mathfrak d}_\la$
appears
as an equivariant Euler class. 
\end{proof}

\begin{proof}[Proof of Proposition \ref{ppp}] 
    For any $\cg_1,\cg_2\in \sat$ and $\la,\mu\in\dom$, the following diagram commutes
  \[
   \xymatrix{
    \   \sfh(i_{\leq\la}^!\cg_1)\o_\sfc \sfh(i_{\leq\mu}^!\cg_2) \ 
    \ar[d]^<>(0.5){\star}
          \ar[rr]^<>(0.5){({i}_{\leq\la})_!\o ({i}_{\leq\mu})_!}
     &&
       \    \sfh(\cg_1)\o_\sfc \sfh(\cg_2) \ \ar[d]^<>(0.5){\star}\\ 
 \  \sfh(i_{\leq\la+\mu}^!(\cg_1\star\cg_2)) \ 
\ar[rr]^<>(0.5){({i}_{\leq\la+\mu})_!}
&& \ \sfh(\cg_1\star\cg_2)
   }
 \]

 Now let $\cf$ be the ring object from  the statement of the proposition.
From the commutative diagram above we deduce that the filtration $\sfh(i_\la^!\cf)$
 on $\sfh(\cf)$ is multiplicative. 
 Further, 
 for each $\la\in \dom$ we have an isomorphism
  $\mbox{$\frac{1}{{\mathfrak d}_\la}$}\bar\si\ccirc q_\la:\, \hh(\veps^\dagger_\la\cf)^{W_\la} \iso
  \sfh(i_\la^!\cf)$,
  by Lemma  \ref{i claim} applied in the case  $\cll=i^!_{\leq\la}\cf$.
     Therefore, the assignment $h\mto \zeta(h):=
     \overline{\text{res}}\inv(\mbox{$\frac{1}{{\mathfrak d}_\la}$}\bar\si\ccirc q_\la)(h)$  gives a $\sfc$-module isomorphism
     $\hh(\veps^\dagger_\la\cf)^{W_\la}$ $\iso  \sfh(i^!_{\leq\la}\cf)/\sfh(i^!_{<\la}\cf)$.
     Assembling these isomorphisms for all $\la$ we obtain an isomorphism
     of $\dom$-graded $\sfc$-modules
     \[\zeta:\  \hh(\veps^\dagger)_+\nana=
     \bplus_{\la\in\dom}\, \hh(\veps^\dagger_\la\cf)^{W_\la}
     \ \iso \  \oplus_{\la\in\dom}\, \sfh(i^!_{\leq\la}\cf)/\sfh(i^!_{<\la}\cf).
     \]

     It remains to prove that the map $\zeta$ is a ring homomorphism.
     To this end, we observe that the assignment
     $h\mto {\mathfrak d}_\nu \cdot h$ gives an injective map
     $\hh(\veps^\dagger_\nu\cf)\to \hh(\veps^\dagger_\nu\cf)$,
     resp. $\sfh(i^!_{\la}\cf)\to \sfh(i^!_{\la}\cf)$,
     since this cohomology group is a free $\fc^{W_\la}$-module.
For $h\in \hh(\veps^\dagger_\nu\cf)$, by construction, we have
     \[{\mathfrak d}_\nu \cdot \zeta(h)=\zeta({\mathfrak d}_\nu \cdot h)
       =\bar\eps_{\pre\nu}(\bar\si(h)).
     \]
     Observe, that we have  ${\mathfrak d}_{\la+\mu}={\mathfrak d}_\la\cdot
     {\mathfrak d}_\mu$ for all $\la,\mu\in\dom$.
         Thus, it suffices to show that
     for all $a\in \hh(\veps^\dagger_\la\cf)^{W_\la}$ and $b\in \hh(\veps^\dagger_\mu\cf)^{W_\mu}$,
     one has
     \[\bar\eps_{\pre\la+\mu}(\bar\si(ab))=\zeta({\mathfrak d}_{\la+\mu}\cdot ab)\ \stackrel{?}= \
       \zeta({\mathfrak d}_\la \cdot a)\cdot
      \zeta({\mathfrak d}_\mu\cdot  b)=  \bar\eps_{\pre\la}(\bar\si(a))\cdot
       \bar\eps_{\pre\mu}(\bar\si(b)).
              \]

              Explicitly, this means that we must show that the element
              \[
                \veps_!(\si(ab))-\veps_!(\si(a))\cdot\veps_!(\si(b))
                              \]
                is contained in $\sfh(i_{<\la+\mu}^!\cf)$.
                Using  Lemma     \ref{ring} and the notation of  Lemma     \ref{sisi}, we find
                \[
                  \veps_!(\si(ab))-\veps_!(\si(a))\cdot\veps_!(\si(b))=
                  \veps_!(\si(ab)-\si(a)\si(b)) =
                  \sum_{\{\nu\in\dom\,\mid\, \nu<\la+\mu\}}\,\veps_!(\si(m_\nu)),
                \]
                for some
                  $m_\nu \in \hh(\veps^\dagger_\nu\cf)$.                The map   $\veps_!$ sends
                $(\oplus_{\{w(\nu),\, w\in W\}}\,\hh(\veps^\dagger_{w(\nu)}\cf)^W$
                to $\sfh(i_{\leq \nu}^!\cf)$. We conclude that
                the sum in the right hand side  of the displayed formula is contained
                in $\sfh(i_{<\la+\mu}^!\cf)$. This proves our claim, and
                Proposition  \ref{ppp}  follows.
                \end{proof}

    \begin{proof}[Proof of Proposition   Proposition  \ref{part}]
     Assume that   the algebra $\hh(\veps^\dagger\cf)_+$  is finitely generated.
    Then, the algebra $\hh(\veps^\dagger\cf)^W$, resp. $\oplus_{\la\in\dom}\, \sfh({i}^!_{\leq\la}\cf)/\sfh({i}^!_{<\la}\cf)$,
    is finitely generated by Corollary \ref{Wfin}, resp.
    Proposition \ref{ppp}. Since the spaces $\sfh({i}^!_{\leq\la}\cf),\,\la\in\dom$,
    form a multiplicative
    filtration on $\sfh(\cf)$, we deduce that the algebra
    $\sfh(\cf)$ is also finitely generated.
    \end{proof}

\small{
\bibliographystyle{plain}

\begin{thebibliography}{BFNM}
\bibitem[Ach]{Ach} P. Achar, {\em Equivariant mixed Hodge modules.} 2013.
\bibitem[A]{A} T.  Arakawa,  {\em
    Chiral algebras of class $\mathcal S$ and Moore-Tachikawa symplectic varieties.}
  arXiv:1811.01577.
\bibitem[Be]{Be} A. Beauville, {\em Symplectic singularities.}
  Invent. Math. {\textbf {139}} (2000), 541-549.
\bibitem[BBD]{BBD} A. Beilinson, J. Bernstein, P. Deligne,  {\em 
Faisceaux pervers.}
Ast\'erisque, 100 (1982), 5--171.
\bibitem[BD]{BD} \bysame, V. Drinfeld, {\em Quantization of Hitchin's integrable system.}
(1991).

\bibitem[Bel]{Bel} G. Bellamy, {\em
Coulomb Branches have symplectic singularities.}  
Lett. Math. Phys. 113 (2023), no. 5, Paper No. 104.


\bibitem[BF]{BF} R. Bezrukavnikov,
 M.~Finkelberg, \emph{Equivariant Satake category and
Kostant--Whittaker reduction}, Mosc. Math. J. 8 (2008),  39-72.
\bibitem[BFM]{BFM}\bysame, \bysame, I. Mirkovic,
  {\em  Equivariant homology and K-theory of affine Grassmannians and Toda lattices.} Compos. Math. 141 (2005),  746-768.
\bibitem[Bon]{Bon} M. Bondarko, {\em 
  Weight structures and `weights' on the hearts of t-structures.}
Homology Homotopy Appl. 14 (2012),  239-261.

\bibitem[Bou]{Bou} J.-F.  Boutot, {\em Singularit\'es rationnelles et quotients
    par les groupes r\'eductifs.}  Invent. Math.  88 (1987), 65-68. 
\bibitem[BK1]{BK1} A. Braverman, D.Kazhdan, {\em
$\gamma$-functions of Representations and Lifting.}

  \bibitem[BK2]{BK2} \bysame,\bysame {\em  $\gamma$-sheaves
    on reductive groups.}  Studies in memory of Issai Schur, Progr. Math., 210 (2003), 27-47.

\bibitem[BY]{BY} \bysame, Z. Yun, {\em
On Koszul duality for Kac-Moody groups.}
Represent. Theory 17 (2013), 1–98.

\bibitem[BZSV]{BZSV} D. Ben-Zvi,  Y. Sakellaridis, A. Venkatesh,
  {\em Relative Langlands Duality.}  	arXiv:2409.04677.


\bibitem[BFN1]{BFN1}\bysame, \bysame, H. ~Nakajima, \emph{Towards a mathematical
  definition
of Coulomb branches of 3-dimensional $N=4$ gauge theories} II.
Adv. Theor. Math. Phys. 22 (2018),  1071-1147.

\bibitem[BFN2]{BFN2}\bysame, \bysame,\bysame,
  {\em Ring objects in the equivariant derived Satake category arising from
    Coulomb branches.}    Adv. Theor. Math. Phys. 23 (2019), 253-344.


    
\bibitem[BDFRT]{BDFRT} \bysame, G. Dhillon, M. Finkelberg, S. Raskin, R. Travkin,
  {\em
    Coulomb branches of noncotangent type.}
  arXiv:2201.09475.

\bibitem[Br]{Br}  M. Brion, {\em
    Poincar\'e Duality and Equivariant Cohomology.}
    Michigan Math. J. 48 (2000), 77-92.
  \bibitem[Bro]{Bro} A. Broer,
    {\em Line bundles on the cotangent bundle of the ﬂag variety.}  Invent. Math. 113
(1993), 1-20.
\bibitem[Ch]{Ch} T.-H. Chen, {\em Notes on Drinfeld-Lafforgue varieties, ring objects
    and relative duality}.  Preprint 2025.

\bibitem[CMNO]{CMNO} \bysame,  M. Macerato, D. Nadler,  J. O'Brien,
  {\em Quaternionic Satake equivalence.}  	arXiv:2207.04078
  \bibitem[Dev]{Dev}  S. Devalapurkar, {\em 
      $ku$-theoretic spectral decompositions for spheres and projective spaces.}
    arXiv:2402.03995.

  \bibitem[DG]{DG} R. Donagi, D.  Gaitsgory, {\em The Gerbe of Higgs bundles.}
        Transform. Groups 7 (2002) 109-153.

  
      \bibitem[FL]{FL} B. Fu,  J. Liu, {\em The affine closure of cotangent bundles of horospherical spaces.}
        arXiv:2502.06383.

            \bibitem[GL]{GL}  D. Gaitsgory, J. Lurie,
        {\em Weil's conjecture for function fields. Vol. 1.}
Ann. of Math. Stud., 199
Princeton University Press, 2019.

\bibitem[Ga]{Ga} T. Gannon, {\em The cotangent bundle of $G/U_P$ and Kostant-Whittaker
      descent.}  arXiv:2407.16844. IMRN  (2025).

    \bibitem[GG]{GG} \bysame, V. Ginzburg, {\em Quantization of the universal centralizer and
    very central $D$-modules.}
  arXiv:2409.18054
  
 \bibitem[GWe]{GWe} \bysame, B. Webster, {\em Functoriality for Coulomb branches}.
  	arXiv:2501.09962.


  \bibitem[GW]{GW} \bysame, H. Williams,
  {\em Differential operators on the base affine space of $SL_n$ and quantized Coulomb branches.}
 arXiv:2312.10278.  




  
  
\bibitem[Gi]{Gi} V. Ginzburg,  {\em Perverse sheaves and ${\mathbb C}^*$-actions}.  
  J. Amer. Math. Soc. 4  (1991), 483--490.
  

  






  

  
 \bibitem[GK]{GK} \bysame, D. Kazhdan, {\em 
 Differential operators on $G/U$ and the Gelfand-Graev action.} 
Adv. Math. 403 (2022), Paper No. 108368.

\bibitem[GR]{GR} \bysame, S.~Riche, {\em Differential operators on $G/U$ and the affine Grassmannian.}
  J. Inst. Math. Jussieu 14 (2015),  ~{493-~575}.

\bibitem[HL]{HL} Q. Ho, P. Li, {\em Revisiting mixed geometry.}
  arXiv:2202.04833.
  
\bibitem[KRS]{KRS}  A. Kapustin, L. Rozansky, N. Saulina,
  {\em Three-dimensional topological field theory and symplectic algebraic geometry I.}
  Nuclear Phys. B 816 (2009),  295-355.
\bibitem[La]{La} V. Lafforgue, {\em A question related to Dixmier algebras.}
  Unpublished notes, 2008.
\bibitem[Ma]{Ma} M. Macerato, {\em Levi-equivariant restriction of spherical perverse sheaves.}
  arXiv:2309.07279.

\bibitem[MV]{MV} 
   I. Mirkovi\'c, K. Vilonen,
   {\em Geometric Langlands duality and representations of algebraic groups over commutative rings.}    Ann. of Math. (2) {\textbf {166}} (2007),  95–143. 
\bibitem[MT]{MT} G.  Moore, Y. Tachikawa, {\em
On 2d TQFTs whose values are holomorphic symplectic varieties.}
 String-Math 2011, 191–207,
Proc. Sympos. Pure Math., 85,  2012.


\bibitem[Nad]{Nad} D. Nadler, {\em
    Matsuki correspondence for the afﬁne Grassmannian.} Duke Math. J. 124 (2004), 421-457.
\bibitem[Na]{Na} H. Nakajima, {\em
Towards a mathematical definition of Coulomb branches of 3-dimensional $N = 4$
gauge theories}, I, Adv. Theor. Math. Phys. 20 (2016),  595-669.

\bibitem[Na2]{Na2} \bysame,
  {\em  $S$-dual of Hamiltonian $G$-spaces and relative Langlands duality.}
  arXiv: 2409.06303.
\bibitem[Nam]{Nam} Y. Namikawa, {\em Extension of 2-forms and symplectic varieties.}
  J. Reine Angew. Math. 539 (2001), 123-147.
  
\bibitem[Ng1]{Ng1} B.C. Ng\^o,
  {\em Le lemme fondamental pour les alg\`ebres de Lie.} Publ. Math. IHES.
  111 (2010), 1-169.
\bibitem[Ng2]{Ng2} \bysame, {\em
 Fibration de Hitchin et endoscopie.} Invent. math. 164 (2006), 399-453.
\bibitem[Sa]{Sa} M. Saito, {\em Modules de Hodge Polarisables.}
    Publ. Res. Inst. Math. Sci. 24 (1988),   849-995.

    \bibitem[T]{T} C. Teleman, {\em The r\^ole of Coulomb branches in $2D$ gauge theory.}
J. Eur. Math. Soc.  23 (2021),  3497-3520.
 \bibitem[YZ]{YZ} Z. Yun, X. Zhu,   {\em Integral homology of loop groups via Langlands dual groups.}
Represent. Theory.
15(2011),  {347-369}.


\end{thebibliography}

\end{document}